\documentclass{amsart}

\usepackage{amsmath,amsfonts,amsthm,amssymb,amscd}
\usepackage{float,graphicx}
\usepackage{hyperref}
\usepackage{upgreek}
\usepackage{tikz}
\newcommand*\circled[1]{\tikz[baseline=(char.base)]{ \node[shape=circle,draw,inner sep=1.2pt] (char) {#1};}}

\newtheorem{theorem}{Theorem}[section]
\newtheorem{lemma}[theorem]{Lemma}
\newtheorem{proposition}[theorem]{Proposition}

\theoremstyle{definition}
\newtheorem{definition}[theorem]{Definition}

\theoremstyle{remark}
\newtheorem{remark}[theorem]{Remark}

\numberwithin{equation}{section}

\newcommand{\ud}{u^{\dagger}}

\newcommand{\cF}{\mathcal{F}}
\newcommand{\cN}{\mathcal{N}}

\newcommand{\GP}{\mathcal{GP}}
\newcommand{\cX}{\mathcal{X}}

\newcommand{\bE}{\mathbb{E}}
\newcommand{\bP}{\mathbb{P}}
\newcommand{\bR}{\mathbb{R}}
\newcommand{\bT}{\mathbb{T}}
\newcommand{\bN}{\mathbb{N}}
\newcommand{\bZ}{\mathbb{Z}}
\newcommand{\bC}{\mathbb{C}}

\newcommand{\rd}{\mathrm{d}}

\newcommand{\sfT}{\mathsf{T}}
\newcommand{\sfL}{\mathsf{L}}

\DeclareMathOperator*{\argmin}{argmin}

\begin{document}

\title[Empirical Bayes and Kernel Flow]{Consistency of Empirical Bayes And Kernel Flow For Hierarchical Parameter Estimation}


\author{Yifan Chen}
\address{Applied and Computational Mathematics, Caltech, 91106}
\curraddr{}
\email{yifanc@caltech.edu}
\thanks{}

\author{Houman Owhadi}
\address{Applied and Computational Mathematics, Caltech, 91106}
\curraddr{}
\email{owhadi@caltech.edu}
\thanks{}

\author{Andrew M. Stuart}
\address{Applied and Computational Mathematics, Caltech, 91106}
\curraddr{}
\email{astuart@caltech.edu}
\thanks{}


\date{}

\dedicatory{}

\begin{abstract}

{Gaussian process regression has proven very powerful in statistics, machine learning and inverse problems. A crucial aspect of the success of this
methodology, in a wide range of
applications to complex and real-world problems, is  hierarchical modeling and learning of hyperparameters. The purpose of this paper is to study two paradigms of learning hierarchical parameters: one is from the probabilistic Bayesian perspective, in particular, the empirical Bayes approach that has been largely used in Bayesian statistics; the other is from the deterministic and approximation theoretic view, and in particular the kernel flow algorithm that was proposed recently in the machine learning literature. Analysis of their consistency in the large data limit, as well as explicit identification of their implicit bias in parameter learning, are established in this paper for a Mat\'ern-like model on the torus. A particular technical challenge we overcome is the learning of the regularity parameter in the Mat\'ern-like field, for which consistency results have been very scarce in the spatial statistics literature. Moreover, we conduct extensive numerical experiments beyond the Mat\'ern-like model, comparing the two algorithms further. These experiments
demonstrate learning of other hierarchical parameters, such as amplitude and lengthscale; they
also illustrate  the setting of model misspecification  in which
the kernel flow approach could show superior performance to the more traditional empirical Bayes approach.}
\end{abstract}

\maketitle
\tableofcontents
\section{Introduction}

\subsection{Background and Context}
\label{sec: Background and Context}
Gaussian process regression (GPR) is important in its own right,
and as a prototype for more complex inverse problems in which
there is a possibly indirect, nonlinear set of observations.
An important reason for the success of GPR in applications is its ability to learn hyperparameters, entering through a hierarchical prior, from data. Learning
of these hyperparameters is typically achieved through
fully Bayesian (sampling) or empirical Bayesian (optimization)
methods. {However, new approaches suggested in the machine
learning literature, particularly the kernel flow method \cite{owhadi2019kernel}, rely on approximation theoretic criteria that can be traced back to the classical idea of cross-validation for model selection. The primary goal of this paper is to study and compare these two approaches. Special attention will be paid to their large data consistency, implicit bias, and robustness to model misspecification.}
\subsection{Gaussian Process Regression}
\label{sec: GPR}
{We start with a brief introduction to GPR; for simplicity, we focus on the noise-free scenario. The target is to recover a function $u^\dagger: D \mapsto \bR$
from pointwise data $y_i=u^{\dagger}(x_i)$ for $1\leq i \leq N$, where $x_i \in D \subset \bR^d$ and $D$ is a compact domain. 
This problem often appears in fields such as supervised learning in machine learning, non-parameteric regression in statistics, and interpolation in numerical analysis. }

The GPR solution to this problem is as follows. Given a family of positive definite covariance/kernel functions $K_{\theta}:D\times D \to \bR$ where $\theta \in \Theta$ is a hyperparameter, GPR approximates $\ud$ with the conditional expectation
\begin{equation}
u(\cdot,\theta,\cX):=\mathbb{E}\,[\xi(\cdot,\theta) \mid \xi(\cX,\theta)=\ud(\cX) ]=   K_\theta(\cdot,\cX)[K_\theta(\cX,\cX)]^{-1} \ud(\cX)\, ,
\end{equation}
 where $\xi(\cdot,\theta) \sim \GP(0,K_\theta)$ is a centered Gaussian process\footnote{Recall that the covariance function $K_{\theta}$ of a Gaussian process $\GP(0,K_{\theta})$ is {the kernel of the integral operator representation of $C_{\theta}$
  in the covariance operator notation $\cN(0,C_{\theta}).$ Connections between these perspectives are reviewed in Subsection \ref{sec: set-up of recovering regularity}. We will use the covariance operator notation more frequently later in this paper.}}(GP) with covariance function $K_\theta$. 
  {We have used the following
compressed notation:
\[\cX:=(x_1,\ldots,x_N)^\mathsf{T}\quad 
\text{and}\quad 
\ud(\cX):=(\ud(x_1),\ldots, \ud(x_N))^\mathsf{T}\, .\]
Moreover,} $K_\theta(\cX,\cX)$ denotes the $N\times N$ dimensional Gram matrix with $(i,j)^{\rm{th}}$ entry $K_\theta(x_i,x_j)$, and $K_\theta(\cdot,\cX)$ is a function mapping $D$ to $\mathbb{R}^N$  with $i^{\rm{th}}$ component $K_\theta(\cdot,x_i):D \mapsto \mathbb{R}$. 
 
Normally, every $\theta \in \Theta$  produces a solution $u(\cdot,\theta,\cX)$ that agrees with $\ud$ on $\cX$. Nevertheless, different choices
may yield distinct out-of-sample errors, known as generalization errors in the machine learning context. Therefore, it is of paramount importance to learn a good hierarchical parameter $\theta$ adaptively from data.

\subsection{Two Approaches}
\label{sec: Two Approaches and Answers}
In this paper, we study two approaches to the question posed above, both based on selecting $\theta$ as the optimizer of a variational problem.  
\subsubsection{Empirical Bayes Approach}
{The empirical Bayes (EB) approach addresses the question by proposing a statistical model.} It formulates a prior distribution on
the pair $(\xi,\theta)$ by assuming that 
$\theta$ is sampled from a prior distribution and $\xi$ is then
sampled from the 
conditional distribution of
$\xi|\theta$; then, it finds the posterior distribution of the pair $(\xi,\theta)$ conditioned on $\xi(\cX)=u^\dagger(\cX)$,
and selects the parameter $\theta$ that maximizes the marginal probability of $\theta$ under this posterior. For simplicity, we work with uninformative priors, which lead to the following objective function:
\begin{equation}\label{eqkjdekjdbnd}
\mathsf{L}^{\mathrm{EB}}(\theta,\cX,\ud)= \ud(\cX)^\sfT [K_\theta(\cX,\cX)]^{-1} \ud(\cX)+\log \det K_\theta(\cX,\cX)\, .
\end{equation}
This is also twice the negative marginal log likelihood of $\theta$ given the data $u^\dagger(\cX)$. Then, EB will choose $\theta$ by minimizing this objective function, namely
\begin{equation}
    \theta^{\mathrm{EB}}(\cX,\ud):=\argmin_{\theta \in \Theta} \sfL^{\mathrm{EB}}(\theta,\cX,\ud)\, .
\end{equation}
\subsubsection{Approximation Theoretic Approach}
{Approximation theoretic considerations, on the other hand, provide a different answer without proposing statistical models.
This methodology proceeds by asking for an ideal $\theta$ that minimizes $\mathsf{d}(\ud,u(\cdot,\theta,\cX))$ for some cost function $\mathsf{d}$.}
Though in practice $\ud$ is not available, there are ideas in cross-validation that split $\cX$ into training data and validation data, and use the approximation error in validation data to estimate the exact error. Inspired by this idea, we could turn to optimize the following objective function:
\begin{equation}
    \mathsf{d}(u(\cdot,\theta,\cX),u(\cdot,\theta,\pi\cX))\, ,
\end{equation}
where we write $\pi\cX$ for a subset of $\cX$ obtained by subsampling a proportion, say one-half, of $\cX$. 


In this paper, we focus on a particular choice of $\mathsf{d}$ that originates from the Kernel Flow (KF) approach \cite{owhadi2019kernel}.  To describe it, we denote by $(\mathcal{H}_\theta,\|\cdot\|_{K_\theta})$ the associated \textit{Reproducing Kernel Hilbert Space} (RKHS) for the kernel $K_{\theta}$; note that $\|K_{\theta}(\cdot,x)\|_{K_\theta}^2=K_{\theta}(x,x)$. The objective function in KF is chosen as
\begin{equation}
    \mathsf{L}^{\mathrm{KF}}(\theta,\cX,\pi\cX,\ud):=\frac{\|u(\cdot,\theta,\cX)-u(\cdot,\theta,\pi\cX)\|^2_{K_{\theta}}}{\|u(\cdot,\theta,\cX)\|^2_{K_{\theta}}}\, .
\end{equation}
This measures the discrepancy in the RKHS norm between the GPR solution using the whole data $\cX$ and using a subset of  the data $\pi\cX$, normalized by the RKHS norm of the former. 
{
\begin{remark}
As explained above, we understand the numerator as an estimation of the error $\|\ud-u(\cdot,\theta,\cX)\|_{K_{\theta}}^2$. Such error estimate, based on comparing
solutions obtained via different data resolutions, is a widely used idea in numerical analysis.
\end{remark}}
Based on Garlerkin orthogonality (see \cite{owhadi2019kernel}), the objective function admits a finite dimensional representation formula that is convenient for numerical computation:
\begin{equation}
\label{eqjehdjhed}
\mathsf{L}^{\mathrm{KF}}(\theta,\cX,\pi\cX,\ud)= 1-\frac{\ud(\pi\cX)^\sfT [K_\theta(\pi\cX,\pi\cX)]^{-1} \ud(\pi\cX)}{\ud(\cX)^\sfT [K_\theta(\cX,\cX)]^{-1} \ud(\cX)}\, .
\end{equation}
Then, the KF estimator is defined as
\begin{equation}
    \theta^{\mathrm{KF}}(\cX,\pi\cX,\ud):=\argmin_{\theta \in \Theta} \sfL^{\mathrm{KF}}(\theta,\cX,\pi\cX,\ud)\, .
\end{equation}
{\begin{remark}
The existence of the finite-sample formula \eqref{eqjehdjhed} is attributed to the choice of the RKHS norm in comparing solutions. It is essentially a consequence of the standard representer theorem. Additional motivations for using the RKHS norm will be reviewed in Subsection \ref{subsec: Kernel Flow and Cross-validation}.
\end{remark}}
\subsubsection{Guiding Observations and Goals}
{The EB and KF algorithms estimate the parameter $\theta$ from the observed data, the number of which can vary considerably. Thus, a basic question to ask is whether the estimators attain meaningful limits as data accumulate:
\begin{enumerate}
    \item[(1)] Consistency: how do $\theta^{\text{EB}}$ and $\theta^{\text{KF}}$ behave in the large data limit, i.e., as the number of data $N$ goes to infinity?
\end{enumerate}
Meanwhile, since we have two estimators, it is natural to compare their performance. Indeed, we observe that EB and KF have distinct objectives: EB seeks to estimate the most likely parameters of the distribution assumed to generate the data, while KF chooses parameters to minimize an estimate of the approximation error in a parameter-dependent RKHS norm, targeting at the approximation efficiency of the underlying function. Moreover, EB is always probabilistic, while KF need not be.}

{These differences motivate the implicit bias question that has been popular in the machine learning community, and the model misspecification question that is common in mathematical modeling:
\begin{enumerate}
    \item[(2)] Implicit bias: what are the selection bias of EB and KF, or, how  should the obtained estimators $\theta^{\text{EB}}$ and $\theta^{\text{KF}}$ be interpreted in practice?
    \item[(3)] Model misspecification: how do $\theta^{\text{EB}}$ and $\theta^{\text{KF}}$ behave when there is a mis-match between the
    data-generating mechanism and the model used to regress the data?
\end{enumerate}
The precise goal of this paper is to address these questions for certain concrete models, either theoretically or experimentally.
}

 
\subsection{Our Contributions} {Our contributions in this paper are twofold and
explained in the following two subsections}.

\subsubsection{Consistency and Implicit Bias}
{The first part of this work is devoted to the questions of consistency and implicit bias. We study a Mat\'ern-like model on the torus, in which $\ud$ is a sample drawn from the Mat\'ern-like Gaussian process, with three parameters $\theta=(\sigma, \tau, s)$ that quantify the amplitude, inverse lengthscale and regularity of the process.
The detailed definition is in Subsection \ref{sec: set-up of recovering regularity}. }

{
Our main analysis concerns learning the regularity parameter $s$ using EB and KF.
When the sampled points $\cX$ are equidistributed, we achieve the following contributions:
\begin{itemize}
    \item Consistency: we prove that the EB estimator converges to $s$ in the large data limit, while the KF estimator converges to $\frac{s-d/2}{2}$, so that $s$
    is also determined. Their variances are also computed and compared.
    \item Implicit bias: we characterize the selection bias of EB and KF algorithms, in terms of the $L^2$ error between $\ud$ and the GPR solution using learned parameters --- this is the so-called generalization error. It is found that EB selects the parameter that achieves the minimal $L^2$ error in expectation, while KF selects the minimal parameter that suffices for the fastest rate of convergence of the $L^2$ error to $0$ as the data density increases.
\end{itemize}
We can interpret these contributions from two perspectives. From the machine learning side, we are able to show that KF, as a machine learning method, has a well-defined large data limit for the Mat\'ern-like model. Furthermore we can characterize clearly its implicit bias in terms of $L^2$ generalization errors. Thus, this paper leads to a \textit{first} theory for the KF learning algorithm. }

{From the spatial statistics side, our analysis contributes to a novel consistency theory for estimating the regularity parameter of Mat\'ern-like fields in general dimensions. Such results are scarce in the spatial statistics literature; the techniques we use to prove consistency may be of independent interest and applicable
beyond the setting considered here.
}

{We also include numerical studies concerning the learning of the amplitude parameter $\sigma$ and the inverse lengthscale parameter $\tau$; these experiments
contribute to a more complete picture of GPR using the Mat\'ern-like field with hierarchical parameters. Moreover, we provide numerical experiments for several other well-specified models beyond the Mat\'ern-like model, thus further extending the scope of discussions.}
\subsubsection{Model Misspecification} {The second part of this work considers model misspecification: the data generating model for  $\ud$ and the model
$K_{\theta}$ used for regression do not match. We adopt the following setting:
\begin{itemize}
    \item We model the truth $\ud$ either as a GP, using a variety of covariance functions, or as a deterministic function which solves a PDE. 
    \item The kernel $K_{\theta}$ is chosen to be Green's function of various differential operators, where $\theta$ encodes information beyond the amplitude, lengthscale, and regularity of the field. For example we choose $\theta$ to
    be the location of a discontinuity within a conductivity field.
\end{itemize}
In this setting we observe distinct behavior distinguishing EB and KF.  This 
raises the discussion of how to choose which algorithm to use when solving practical problems where misspecification is to be expected. Our numerical study explores several misspecification possibilities, showing that KF could be competitive with EB in certain scenarios.
}

\subsection{Literature Review} In this subsection, we review the related literature. Several fields are of relevance, so we label them to help organize the review.
\label{sec: literature review}
 \subsubsection{Regression and Inverse Problems}
Regression is a form of inverse problem \cite{dashti2013bayesian}, and if formulated in a Bayesian fashion, it falls within the scope of Bayesian nonparametric estimation \cite{ghosh2003bayesian,hjort2010bayesian}. In the paper \cite{knapik2011bayesian} a simple
class of linear inverse problems was studied from the perspective of posterior consistency,
and it was demonstrated that the rate of posterior convergence depends sensitively on the
relationship between regularity of the true function being sought, and the regularity
of draws from the prior. This motivates the need for hierarchical procedures that
adapt, on the basis of the data, the regularity of draws from the prior. In \cite{knapik2016bayes}
the work in \cite{knapik2011bayesian} was extended to cover the data-adapted learning
of the regularity parameter in the prior; as the authors note: theoretical work
``that supports the preference for empirical or hierarchical Bayes methods does
not exist at the present time, however. It has until now been unknown whether these
approaches can indeed robustify a procedure against prior mismatch. In this paper,
we answer this question in the affirmative.'' This analysis, however, requires
simultaneous diagonalization of a self-adjoint operator formed from the forward
model and the covariance operator, for all values of the hyper-parameter.
Consistency is studied without this assumption in \cite{wendland2004scattered},
and extended to the study of emulation within Bayesian inversion
in \cite{stuart2018posterior} and to empirical Bayesian procedures in \cite{teckentrup2019convergence}. The papers \cite{knapik2016bayes} and \cite{teckentrup2019convergence} also use the EB loss
function \eqref{eqkjdekjdbnd}. In \cite{dunlop2017hierarchical} estimation of hyper-parameters
in Gaussian priors is discussed in the context of MAP estimators.

\subsubsection{Kernel Flow and Cross-validation}
\label{subsec: Kernel Flow and Cross-validation}
The KF loss function in \eqref{eqjehdjhed} was originally derived in \cite{owhadi2019kernel} {and motivated from the perspective of optimal recovery theory}. It 
can be interpreted, from a numerical homogenization perspective  \cite{owhadi2019operator}, as the relative energy contained in the fine scales (in the unresolved part) of $\ud$. 
In the paper \cite{owhadi2019kernel}, the proposed loss function to be optimized (via SGD) has the form
\begin{equation}
    \bE_{\pi_1}\bE_{\pi_2} \mathsf{L}^{\mathrm{KF}}(\theta,\pi_1\cX,\pi_2\pi_1\cX,\ud)\, ,
\end{equation}
 where $\pi_1\cX$ is a subsampling of $\cX$, and $\pi_2\pi_1\cX$ is a further subsampling of $\pi_1\cX$.  This choice reduces the dimension of the
kernel matrix and enables fast computation per iteration. Although the KF loss appears to be new,
it can be seen as a variant of cross-validation (CV), which is a commonly used model selection/parameter estimation criteria \cite{allen1974relationship, geisser1975predictive, kohavi1995study}. A theoretical understanding of the consistency of CV
``is very much of interest'' \cite{yang2007consistency} since its convergence rate can be  shown to be asymptotically minimax \cite{stone1984asymptotically} or near minimax optimal \cite{van2006cross, van2006oracle} while having a lower computational complexity \cite{zhang2010kriging} than MLE (maximum likelihood estimation). The  consistency  of  parameter  estimation  for the  Ornstein-Uhlenbeck  process  has been  studied  in  \cite{ying1991asymptotic}  for MLE,  and  \cite{bachoc2017cross}  for  CV. 

In the setting of hyperparameters estimation of GPs, comparing MLE 
with CV can be traced back to Wahba \cite{wahba1980some} and Stein \cite{stein1990comparison} who compared variants of these procedures\footnote{modified maximum likelihood estimation and generalized cross validation} for choosing the smoothing parameter of a smoothing spline; they observed that  while MLE  is optimal when the model is well-specified, CV may perform better (than MLE) under misspecification (see also \cite{bachoc2013cross} for theoretical analysis and \cite{warnes1987problems} for a practical example involving real data) and has a comparable rate of convergence when the model is correct (Stein \cite{stein1990comparison} observed that ``both estimates are asymptotically normal with the CV estimate having twice the asymptotic variance of the MLE estimate'' and suggested that
``The penalty for using CV instead of MLE when the stochastic model is correct is greater for higher-order smoothing splines, both in terms of the efficiency in estimating the smoothing parameter and the impact on subsequent predictions''). We also refer to \cite{kohn1991performance} for a detailed numerical comparison between MLE and CV for estimating spline smoothing parameters. 
As observed in \cite{scheuerer2013interpolation}, these comparisons ``are relevant for both numerical analysts
and statisticians'' since kernel interpolation can be interpreted as both approximating a deterministic unknown function from quadrature points or as estimating a sample from a Gaussian process from pointwise measurements.

\subsubsection{Machine Learning and Kernel Learning}
Kernel methods and GPs have long been used in machine learning \cite{hofmann2008kernel,rasmussen2003gaussian}. Learning a good kernel for a given task is very important in practice. Many works have tried to learn a kernel from data based on different criteria; for example, in \cite{amari1999improving}, the kernel is modified to make the model have a large margin in classification, and in \cite{cortes2013learning}, the kernel is selected to have a small local Rademacher complexity. EB and KF loss functions in this paper have also been used in \cite{rasmussen2003gaussian, wilson2016deep,owhadi2019kernel}. 

{The recent discovery of the neural tangent kernel regime for overparameterized models \cite{jacot2018neural} and the identification \cite{owhadi2020ideas} of warping kernels \cite{sampson1992nonparametric,perrin1999modelling, schmidt2003bayesian, owhadi2019kernel} as the infinite depth limit of residual neural networks \cite{he2016deep}
also suggest that a theoretical understanding of kernel selections may lead to important insights for neural network based machine learning. 
This line of work suggests that it may be fruitful to consider machine learning directly as the problem of selecting an underlying kernel (by minimizing nonlinear functionals of the empirical distribution such as \eqref{eqkjdekjdbnd} or \eqref{eqjehdjhed}) and learning based on this kernel; in this perspective one has
hierarchical GPR with kernel itself as the hyperparameter. This may be more effective than simply fitting the data by minimizing a generalized moment, i.e., a linear functional, of the empirical distribution, which is popularly used in empirical risk minimization. Numerical experiments presented in \cite{yoo2020deep} and \cite{hamzi2020learning}, based on the KF methodology in \cite{owhadi2019kernel}, provide evidence that (1) this point of view could improve test errors, generalization gaps, and robustness to distribution shifts in the training of of ANNs, and (2)   kernel methods can be a simple and effective approach for learning dynamical systems and surrogate models, with the underlying kernel also learned from data (using KF and its variants).
This further motivates the desire to understand the KF-based estimation of $\theta.$}

\subsection{Organization} 
{The rest of this paper is organized as follows. 
Section \ref{sec: recover regularity} is devoted to learning the regularity parameter of the Mat\'ern-like model, where the large data consistency is proved and implicit bias is characterized. Most of the detailed proofs are deferred to Section \ref{sec: proof}, and concise intuitive ideas are presented in Section \ref{sec: recover regularity} for the sake of readability. Section \ref{sec: numerical experiments} considers other well-specified models, including the learning of the lengthscale and amplitude parameters in the Mat\'ern-like model, or beyond the Mat\'ern-like model. Experiments are provided concerning consistency and variance of these EB and KF estimators. Section \ref{sec: experiments, robustness} covers discussions on model misspecification through numerical studies. The purpose of the numerical experiments is twofold: (i) to demonstrate the extent to
which the ideas learned through the analysis of consistency, which focuses primarily on the regularity parameter, extends to other parameters; (ii) to compare the performance of the EB and KF estimators quantitatively, since use of the latter is somewhat new in this area and its potential pros and cons need to be evaluated. Finally, we conclude this paper in Section \ref{sec: discussions}.}

\section{Regularity Parameter Learning for the Mat\'ern-like Model}
\label{sec: recover regularity}
{In this section, we study a Mat\'ern-like model on the torus. We start with definitions of this model in Subsection \ref{sec: set-up of recovering regularity}, followed with definitions of EB and KF estimators in this context in Subsection \ref{Regularity Parameter Learning}. Then, in Subsection \ref{sec: Main Theorem, Implications, and Proof Technique}, we present our theory for the consistency of EB and KF estimators in learning the regularity parameter, with experiments included to demonstrate the correctness and implications of the theory. In particular, the implicit bias of these two estimators is explained. We outline the sketch of proofs for the theoretical result in Subsections  \ref{sec: Fourier}, \ref{sec: consistency EM estimator} and \ref{sec: consistency KF estimator}, and summarize several observations
in Subsection \ref{sec:regularity recovery, discussion}. 
Subsection \ref{sec: experiments, variance of estimators} provides additional experiments discussing the variance of these estimators. }


\subsection{The Mat\'ern-like Model}
\label{sec: set-up of recovering regularity}
{We follow the general set-up in Subsections \ref{sec: GPR} and \ref{sec: Two Approaches and Answers}, where we have mentioned all the abstract ingredients such as the physical domain $D$, the truth $\ud$, the kernel $K_{\theta}$, and the data location $\cX$. In the current and next subsections, we will specify the exact meaning of these terms for a Mat\'ern-like model on the torus. We will also make remarks to explain its connection to the standard Whittle-Mat\'ern process in the whole domain; see Remark \ref{connect to whittle matern}.}
\subsubsection{The Physical Domain} {We set $D$ to be $\bT^d=[0,1]^d_{\mathrm{per}}$, the $d$ dimensional unit torus; this will be the domain that we use for all our analysis.
We need to introduce some mathematical concepts related to functions
defined on this torus $\bT^d.$ First, the space of square integrable functions on $\bT^d$ with mean $0$ is denoted by
\begin{equation}
    \dot{L}^2(\bT^d):=\Bigl\{v: \bT^d \to \bR:\   \int_{\bT^d} |v(x)|^2\, \rd x < \infty,\  \int_{\bT^d} v(x)\, \rd x = 0 \Bigr\}\, . 
\end{equation}
The $L^2$ inner product and norm are denoted by $[\cdot,\cdot]$ and $\|\cdot\|_0$ respectively. }

{In order both to define covariance operators and Sobolev spaces
it is convenient to introduce the Laplacian operator. Let $-\Delta$ be the negative Laplacian equipped with periodic boundary conditions on $\bT^d$ and restricted to functions with zero mean. This operator has orthonormal eigenfunctions $\phi_m(x)=e^{2\pi i\left<m,x\right>}$ with corresponding eigenvalues $\lambda_m=4\pi^2|m|^2$, for every $m \in \bZ^d \backslash \{0\}$, where $\bZ^d$ denotes the $d$-fold tensor product of $\bZ$, the set of non-negative integers.  Here, $i$ is the imaginary number, and $\left<m,x\right>$ denotes the Euclidean inner product between $m, x \in \bR^d$.}

{Now, we can write functions in $\dot{L}^2(\bT^d)$ as Fourier series:
\begin{equation}
    v(x)=\sum_{m \in \bZ^d} \hat{v}(m)e^{2\pi i\left<m,x\right>}\, ,
\end{equation}
 where $\hat{v}:\bZ^d \to \bR$ is the Fourier coefficient that satisfies $\hat{v}(0)=0$ and $\hat{v}(m)=[v,\phi_m]$ for $m \in \bZ^d \backslash \{0\}$. This representation can be used to define useful Sobolev-like spaces. For every $t > 0$, the Sobolev-like space $\dot{H}^t(\bT^d) \subset \dot{L}^2(\bT^d)$ consists of functions with bounded $\|\cdot\|_t$ norm:
 \begin{equation}
     \|v\|_t^2:=\sum_{m\in \bZ^d} (4\pi^2|m|^2)^t|\hat{v}(m)|^2 < \infty \, .
 \end{equation}
 We note that $\dot{H}^0(\bT^d)=\dot{L}^2(\bT^d)$. For $t < 0$, the space $\dot{H}^t(\bT^d)$ is defined through duality. The Hilbert scale of function spaces defined through varying $t$ serves as the basic ingredient to model the regularity of a function on $\bT^d$.} 

\subsubsection{The Mat\'ern-like Kernel and Process}
\label{sssec:need}
{The Mat\'ern-like covariance operator on the torus is defined by \begin{equation}
\label{eqn: matern like kernel operator}
    C_{\theta}=\sigma^2(-\Delta+\tau^2 I)^{-s} \, ,
\end{equation}
where the parameter $\theta=(\sigma, \tau, s)$. The roles of the three parameters are reviewed in Remark \ref{connect to whittle matern}. The orthonormal eigenfunctions of this operator are  $\phi_m(x)=e^{2\pi i\left<m,x\right>}$ with corresponding eigenvalues $\sigma^2(4\pi^2|m|^2+\tau^2)^{-s}$, for $m \in \bZ^d \backslash \{0\}$.}

{The Mat\'ern-like kernel function $K_{\theta}$ is related to the operator $C_{\theta}$ via 
\begin{equation}
    K_{\theta}(x,y)=[\updelta(\cdot-x),C_{\theta}\updelta(\cdot-y)]
\end{equation}
 where $\updelta(\cdot-x)$ is the Dirac function centered at $x$. Equivalently, $K_{\theta}$ can be understood as the Green function of the differential operator $C_{\theta}^{-1}$. Note that by Sobolev's emdedding theorem, $s>d/2$ is required to make $K_{\theta}(x,y)$ pointwise well-defined (See Section 7.1.3 and Lemma 7.2 in \cite{dashti2013bayesian}): $K_{\theta}(\cdot,y)$ then lies in the space of continuous functions
 for any $y \in \bT^d.$
 \begin{remark}
 \label{rmk:mercer}
  We also have the Mercer decomposition of the kernel function:
\begin{equation}
\label{eqn:mercer decomposition}
    K_{\theta}(x,y)=\sum_{m \in \bZ^d \backslash  \{0\}}\sigma^2(4\pi^2|m|^2+\tau^2)^{-s} \phi_m(x)\phi^*_m(y)\, ,
\end{equation}
where $\phi^*_m$ is the complex conjugate of $\phi_m$. 
 \end{remark}}

{Given these function spaces and operators, we can define the Mat\'ern-like process using the Gaussian measure notation:
\begin{equation}
\label{eqn: Gaussian prior notation}
    \xi \sim \cN\Bigl(0,\sigma^2(-\Delta+\tau^2 I)^{-s}\Bigr)\, .
\end{equation} 
This covariance operator viewpoint could be understood as follows:  for any $f \in \dot{L}^2(\bT^d)$, the quantity $[f,\xi]$ is a Gaussian random variable with mean $0$ and variance $[f, \sigma^2(-\Delta+\tau^2 I)^{-s}f]$. We note that \eqref{eqn: Gaussian prior notation} is equivalent to the GP notation $\xi \sim \GP(0,K_{\theta})$.
For more details on how to define Gaussian measures using operators we refer to  \cite{bogachev1998gaussian, owhadi2019operator}.
A sample from this process can be realized by the Karhunen–Lo\`eve expansion
\begin{equation}
\label{eqn:sample drawn GP}
    \xi(x) = \sum_{m \in \bZ^d \backslash  \{0\}}\sigma(4\pi^2|m|^2+\tau^2)^{-s/2} \phi_m(x)\xi_m\, ,
\end{equation}
where $\xi_m$ ($m \in \bZ^d \backslash  \{0\}$) are i.i.d. standard normal random variables;
we have $\mathbb{E}\,\xi(x)\xi(y)=K_{\theta}(x,y).$ Numerically, we can draw a sample by truncating this series and restricting to a grid of values on the torus. Alternatively it is
possible to discretize the differential operator $C_{\theta}^{-1}$ on a grid first, and then compute the discrete eigenfunctions to draw a sample. Such an idea is useful when the 
eigenvalues and eigenfunctions of $C_{\theta}^{-1}$ are not analytically known a priori. Indeed, when the operator is discretized into a matrix, the infinite dimensional Gaussian measure becomes a finite dimensional one with the covariance matrix being the discretization of $C_{\theta}$.  Drawing samples is then straightforward. In this section, however, we work on the torus and so the eigenvalues and eigenfunctions are known explicitly and the truncated Karhunen–Lo\`eve expansion could be employed.
\begin{remark}
\label{connect to whittle matern}
 The three parameters $\sigma,\tau$ and $s$ quantify the amplitude, inverse lengthscale, and regularity of the process, respectively. This setting is similar to that of the standard Mat\'ern process \cite{stein99book,guttorp2006studies}, defined on the whole space $\mathbb{R}^d$, whose kernel function and associated covariance operator are both characterized by three parameters; see \cite{linkGFGMRF} for links to the solution
 of stochastic PDEs, an approach attributable to
 Whittle \cite{whittle1954stationary,guttorp2006studies}. 
 The Mat\'ern kernel function is
 \[K_{\sigma,l,\nu}(x,y)=\sigma^2\frac{2^{1-\nu}}{\Gamma(\nu)}\left(\frac{|x-y|}{l}\right)^\nu B_\nu\left(\frac{|x-y|}{l}\right)\, , \]
 for $x,y \in \bR^d$, where $B_\nu$ is the modified Bessel function of the second kind of order $\nu$. On $\bR^d$, this kernel function corresponds to the covariance operator
 \[C_{\sigma,l,\nu} = \frac{\sigma^2l^d\Gamma(\nu+d/2)(4\pi)^{d/2}}{\Gamma(\nu)}(I-l^2\Delta)^{-\nu-d/2}\, .\]
 From this formula, the connection between the Mat\'ern covariance operator in $\bR^d$ and the Mat\'ern-like kernel operator \eqref{eqn: matern like kernel operator} on $\bT^d$ becomes apparent.
 We restrict our analysis to the torus to exploit powerful Fourier series techniques. We will also comment on other boundary conditions in Subsection \ref{sec:regularity recovery, discussion}.
For related results regarding the Mat\'ern process in $\bR^d$ or other bounded domains, we recommend the book \cite{stein99book}. We note that  \cite[Sec.~6.7]{stein99book} also considers a periodic version of the Mat\'ern model and discusses (via the
Fisher information matrix) the fixed domain asymptotics of the maximum likelihood estimate of the three parameters.
 By using the Mercer decomposition \eqref{eqn:mercer decomposition}, the periodic case there is mathematically equivalent to the Mat\'ern-like model on the torus that is considered in this paper. In the next subsection, we prove the consistency of estimators for the regularity parameter, providing a rigorous theory for this periodic model. It
 would be interesting, in future work, to combine this
 consistency with the properties of the Fisher information matrix established in \cite[Sec.~6.7]{stein99book} to obtain
 Bernstein-von-Mises type theorems characterizing asymptotic
 normality of the estimator.
\end{remark}}
\subsection{Regularity Parameter Learning}
\label{Regularity Parameter Learning}
{With the Mat\'ern-like kernel and process defined, we move to discuss the parameter learning problem in this subsection. We fix  $\sigma=1$ and  $\tau=0$ in the Mat\'ern-like model and focus on the regularity parameter only. To proceed, we need to make precise the ground truth $\ud$, the kernel, and the data location $\cX$, of the learning problem.
 \subsubsection{The Ground Truth}
Our theoretical results regarding the consistency of EB and KF estimators will be based on the assumption that $\ud$ is drawn from the GP $\cN(0,(-\Delta)^{-s})$ for some $s>d/2$.}
{
\begin{remark}
\label{rmk: regularity GP}
We note some regularity properties of this GP here. The Cameron-Martin space for $\xi \sim \cN(0,(-\Delta)^{-s})$ is $\dot{H}^s(\bT^d)$ (for readers not familiar with the Cameron-Martin space, see Theorem 7.33 in \cite{dashti2013bayesian}). However, $\xi$ is not an element of this space, almost surely. Indeed, it holds that $\xi$ belongs to $\dot{H}^{s-d/2-\eta}(\bT^d)$ 
for any $\eta > 0$ almost surely (and to H\"older spaces with the same number of fractional derivatives; see Theorem 2.12 in \cite{dashti2013bayesian}). Furthermore, since the Laplacian operator is homogeneous and thus the covariance operator is stationary in space, the regularity of the path is spatially homogeneous (the measure is space translation-invariant).
Here,  we refer, for this phenomenon, to $\xi$ (as a function) having 
\emph{homogeneous critical regularity} $s-d/2$ across $\bT^d$. If we drop the term ``homogeneous'', we mean the property holds without the requirement of spatial homogeneity. Such behavior may occur for functions with spatial singularities.
\end{remark}
\begin{remark}
\label{rmk: s>d/2}
We always require $s >d/2$, which ensures the continuity of the sample path of $\xi$ almost surely and guarantees that $\dot{H}^s(\bT)$ is a RKHS, according to discussions in Remark \ref{rmk: regularity GP}. Thus, the pointwise value of $\xi$ makes sense.
\end{remark}
\subsubsection{The Equidistributed Data} We observe
equidistributed pointwise values of $\ud$ over the torus, i.e., the data lie on a lattice. To describe the data locations we introduce a level parameter $q \in \bN$ such that, for a given $q$, we have the data locations $\cX_q:=\{x_j: j \in J_q$\}, where $x_j=(j_1,j_2,...,j_d)\cdot 2^{-q}$ and $J_q:=\{(j_1,j_2,...,j_d) \in \bN^d: 0\leq j_k\leq 2^q-1,\forall \ 1 \leq k \leq d \}$. We also use the simplified notation $x_j=j2^{-q}$ throughout the paper. }
{
\subsubsection{The EB and KF Estimators} 
We follow the definitions in Subsection \ref{sec: Two Approaches and Answers}. Here, the kernel function for the regularity learning problem will be \[K_{\theta}(x,y)=[\updelta(\cdot-x),(-\Delta)^{-t}\updelta(\cdot-y)]\, ,\]
where the parameter $\theta=\{t\}$. Similar to Remark \ref{rmk:mercer}, it has the following Mercer decomposition
\begin{equation}
    K_{\theta}(x,y)=\sum_{m \in \bZ^d \backslash  \{0\}}(4\pi^2|m|^2)^{-s} \phi_m(x)\phi^*_m(y)\, .
\end{equation}
Numerically, we can compute it by truncating this infinite series. Fast Fourier Transform could be applied to speed up computation of the kernel matrix.}

{
We adapt several notations from Subsection \ref{sec: Two Approaches and Answers} to this specific problem, by writing $t$ instead of $\theta$, and $q$ instead of $\cX_q$, and $K(t,q)$ instead of $K_{\theta}(\cX_q,\cX_q)$. These simplified notations make the analysis cleaner to present. Under such convention, the EB estimator for the regularity parameter is:
\begin{align}
\label{eqn: def of estimators, recover regularity, EB}
s^{\text{EB}}(q,\ud)=\argmin_{t \in [d/2+\delta,1/\delta]} \mathsf{L}^{\text{EB}}(t,q,\ud), \  \mathsf{L}^{\text{EB}}(t,q,\ud):=\|u(\cdot,t,q)\|_t^2 + \log \det K(t,q)\, .
\end{align}
    Here, $u(\cdot,t,q)$ is the GPR solution using the kernel function $K_t$ and the observational data of $\ud$ at $\cX_q$. 
     \begin{remark}
The formula \eqref{eqn: def of estimators, recover regularity, EB} is the continuous formulation of the EB loss function, which is more convenient for theoretical analysis of consistency. The finite-sample formula \eqref{eqkjdekjdbnd} is more useful in numerical computation, and it can be derived from \eqref{eqn: def of estimators, recover regularity, EB} by using the representer theorem. 
   \end{remark} 
     \begin{remark} As in Remark \ref{rmk: s>d/2}, we require the regularity parameter $t>d/2$. Here, furthermore, we introduce a number $\delta>0$ and select the domain of the parameter to be $t \in [d/2+\delta,1/\delta]$;  $\delta$ can be any arbitrary positive number, and this compactification of the parameter domain will simplify the subsequent analysis. The reader should not confuse real number $\delta$ with Dirac delta function $\updelta$.
\end{remark}
    For the KF loss function, we fix the subsampling operator to be equidistributed subsampling so that $\pi \cX_q= \cX_{q-1}$; for this
    choice, we can omit the dependence of the estimator on the subsampling operator $\pi$ in the notation and write:
    \begin{equation}
        \label{eqn: def of estimators, recover regularity, KF}
    s^{\text{KF}}(q,\ud)=\argmin_{t \in [d/2+\delta,1/\delta]} \mathsf{L}^{\text{KF}}(t,q,\ud), \
    \mathsf{L}^{\text{KF}}(t,q,\ud):= \frac{\|u(\cdot,t,q)-u(\cdot,t,q-1)\|_t^2}{\|u(\cdot,t,q)\|_t^2}\, .
    \end{equation}}
    \subsection{Consistency and Implicit Bias} In this subsection, we present our theory of consistency and characterize the implicit bias via numerical experiments. The sketch of proofs is given in the next subsections.
    \label{sec: Main Theorem, Implications, and Proof Technique}
    \subsubsection{Main Theorem} We have the following theorem regarding the consistency of the two statistical estimators in the large data limit:
    \begin{theorem}
        \label{thm: consistency thm}
        Fix $\delta >0$. Suppose $\ud$ is a sample drawn from the Gaussian process $\cN(0,(-\Delta)^{-s})$. If
        $s \in [d/2+\delta,1/\delta]$ then, for the Empirical Bayesian estimator,
        \[\lim_{q \to \infty} s^{\mathrm{EB}}(q,\ud)=s\, ; \]
        if $\frac{s-d/2}{2} \in [d/2+\delta,1/\delta]$ then for the Kernel Flow estimator, 
        \[\lim_{q \to \infty} s^{\mathrm{KF}}(q,\ud)=\frac{s-d/2}{2}\, . \]
        In both cases the convergence is in probability with respect to randomly chosen $\ud.$
    \end{theorem}
    
    {\begin{remark}
    Strictly speaking this theorem shows that EB consistently
    estimates the regularity parameter, whilst
    KF does not. However we make two observations about this. Firstly, the true value of $s$ can be recovered from the KF estimator by a simple linear transformation. And, secondly, the value selected
    by KF is optimal with respect to minimizing
    a specific measure of generalization error (as we will show in the discussion of implicit bias in Subsection \ref{sec: implicit bias}),
    and is of clear interest from this perspective.
    \end{remark}
    \begin{remark}
    The use of $\delta$ in the proof (and hence statement) of this theorem helps by compactifying the parameter space. In practice, numerics demonstrate that it is not intrinsic to the problem. We leave for future work the problem of a more refined theorem, and proof, which does not rely on it.
    \end{remark}
    }
    \begin{remark}
    For economy of notation we will drop explicit reference to the dependence of the loss functions
    and the estimators on $\ud$ in what follows; we will simply write 
    $\mathsf{L}^{\mathrm{EB}}(t,q)$, $\mathsf{L}^{\mathrm{KF}}(t,q)$,  $s^{\mathrm{EB}}(q)$,
    $s^{\mathrm{KF}}(q)$.
    \end{remark}
    The remainder of this subsection is devoted to
    numerical experiments illustrating the theory, discussion of the
    implications of the theory (i.e. implicit bias), and an overview of the proof techniques
    we adopt.
    
    \subsubsection{Numerical Illustration of Theory} 
    \label{section: numerical example for s, s-d/2 /2 demonstration} We present a numerical example to demonstrate the main theorem,
    and its consequences for regression. Consider the one dimensional case, i.e., $d=1$. We set the ground truth $s=2.5$ and so $\frac{s-d/2}{2}=1$. The domain is discretized with $N=2^{10}$ equidistributed grid points. For our first set of experiments we fix the resolution level of the data points to be $q=9$, i.e., we have $2^9$ equidistributed observations of the unknown function $u^\dagger$. In what follows the Laplacian is as defined
    in Subsection \ref{sssec:need}. Given a sample of $u^\dagger$ from  $\cN(0,(-\Delta)^{-s})$, we form the loss function for the EB and the KF estimators. {We draw this sample using the formula \eqref{eqn:sample drawn GP} with $\sigma=1$ and $\tau=0$; we truncate the series to the grid resolution.} A single realization of these loss functions is then shown in Figure \ref{fig: 1 d example}.
    \begin{figure}[ht]
    \centering
    \includegraphics[width=6cm]{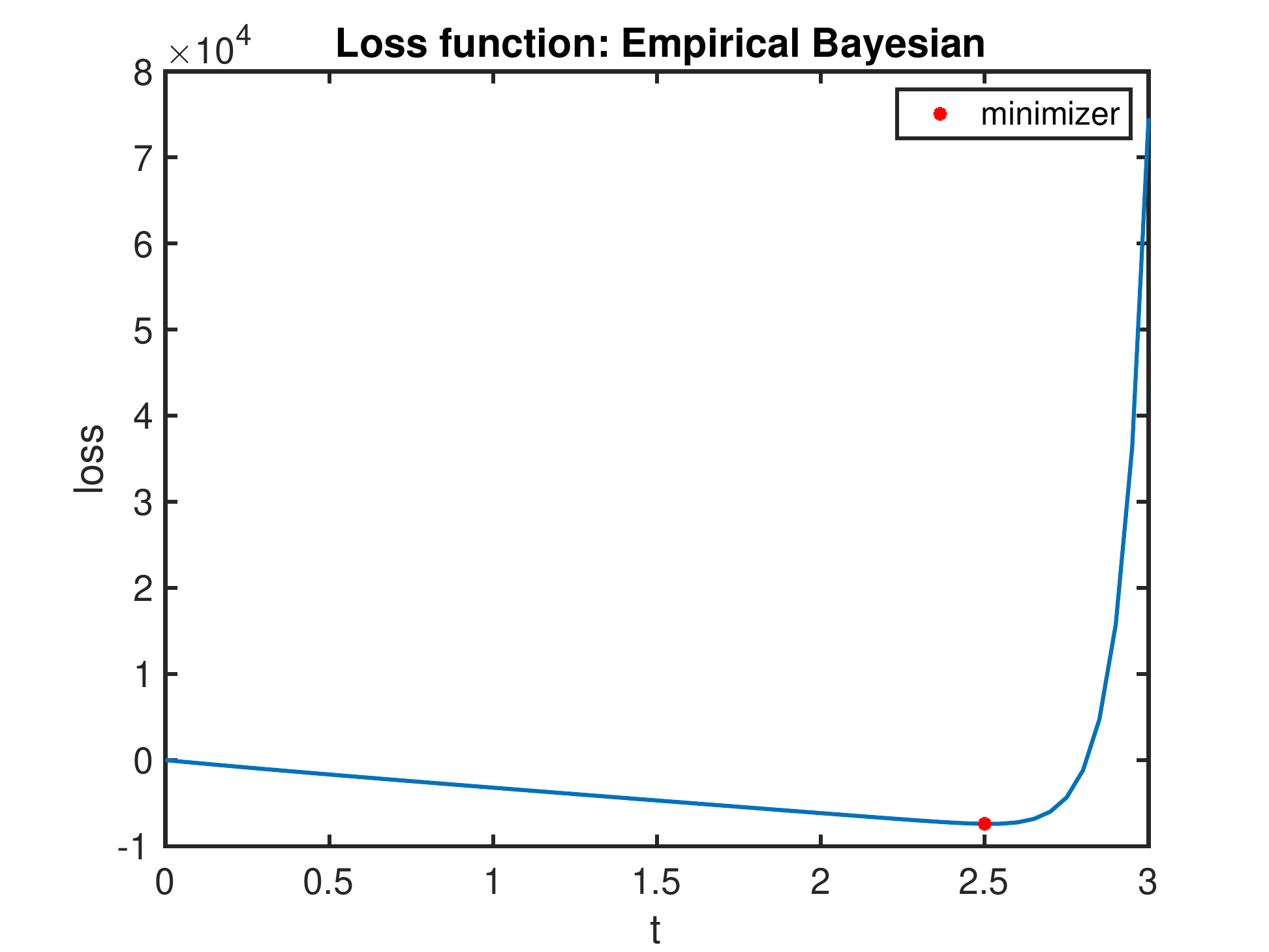}
    \includegraphics[width=6cm]{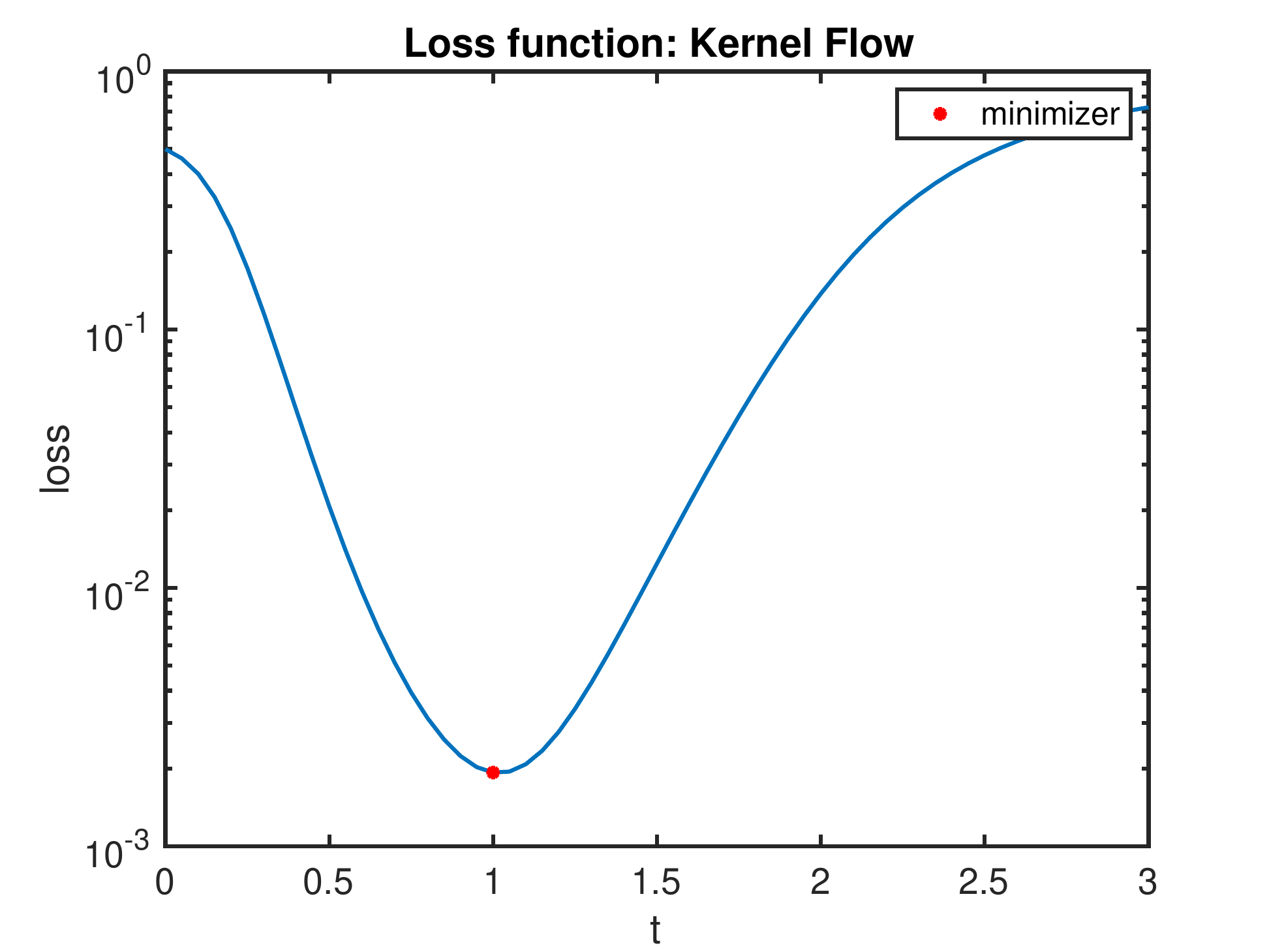}
    \caption{Left: EB loss; right: KF loss}
    \label{fig: 1 d example}
    \end{figure} 
    
     We observe that the minimizer of the EB loss function is very close to $t=2.5$, while the minimizer of the KF loss function is very close to $t=1$, matching the predictions of Theorem \ref{thm: consistency thm}. Furthermore, the loss functions exhibit some interesting features. Specifically, the EB loss function behaves as a linear function of $t$, for $t$ less than $s$, and then blows up rapidly when $t$ exceeds $s$. The KF loss function is more symmetric with respect to the minimizer $t=\frac{s-d/2}{2}$ in the logarithmic scale. We will make remarks that explain these observations in our theoretical analysis.

    \subsubsection{{Implicit Bias}}
    \label{sec: implicit bias}
     We present here a second set of numerical experiments looking at the effect of the parameter value $s$ selected by EB and KF on the approximation of the function $\ud$, which is (typically) the
     primary goal of hierarchical parameter estimation. The experimental set-up is the same, but now we vary the resolution of the data points $q=3,4,...,9$. We focus on the $L^2$ error between $\ud$ and the GPR solution using learned parameters, i.e.,
    \[\|\ud(\cdot)-u(\cdot,t,q)\|_0^2\, .\]
    
    \begin{figure}[ht]
    \centering
    \includegraphics[width=9cm]{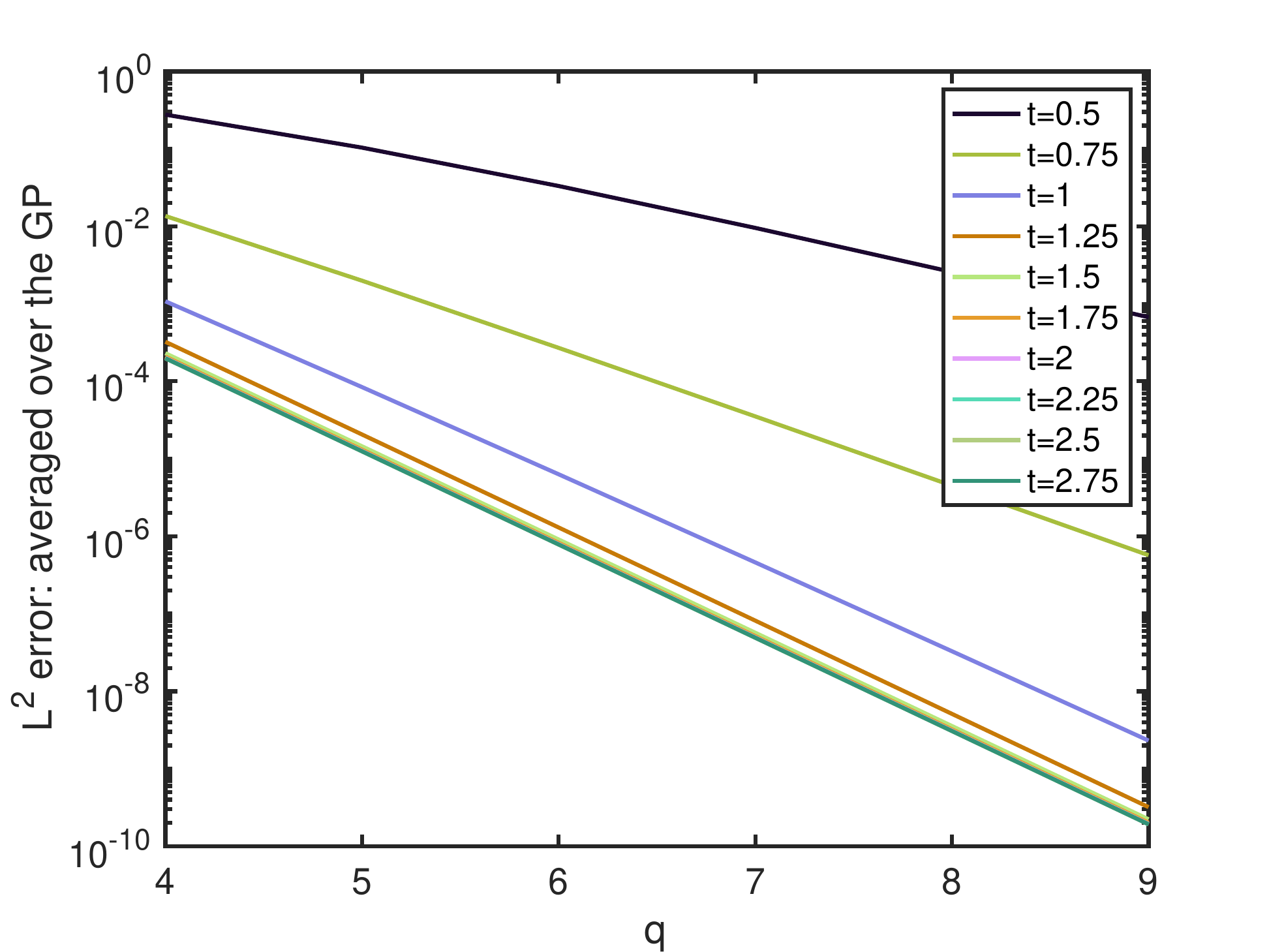}
    \caption{$L^2$ error: averaged over the GP}
    \label{fig: L^2 error}
    \end{figure} 

    We start, in Figure \ref{fig: L^2 error}, by considering the error as a function of $q$, for different $t$.
    As we increase $t$, the regularity of the GP used for regression increases. In order to illustrate clear trends, the $L^2$ error is averaged over the random draw of $\ud \sim \cN(0,(-\Delta)^{-s})$, so the effective error is $\bE_{\ud}\|\ud(\cdot)-u(\cdot,t,q)\|_0^2$. From the figure, we can see that when $t$ increases from $0.5$ to $1$, the convergence rate of the $L^2$ approximation error increases. Then, if we increase $t$
    further from $1$ to $3$, the slope of the convergence curve remains nearly the same. This
    demonstrates the fact that $1=\frac{s-d/2}{2}$ is the minimal $t$ that suffices to achieve the fastest rate of $L^2$ error convergence. We have observed that this phenomenon is very stable with respect to the
    specific random draw: the general shape of the curves seen in
    Figure \ref{fig: L^2 error} is still observed when one specific draw of the true random process is used, although the resulting figure contains fluctuations and is not as clear as the average case that we show.  
    
    
    On the other hand, we can compute   $\bE_{\ud} \|\ud(\cdot)-u(\cdot,t,q)\|_0^2$ for $q=9$ as a function of $t$; see Figure \ref{fig: L^2 error: ave GP}.
    The optimality of the value $s=2.5$ is clear. However, unlike the experiments in Figure~\ref{fig: L^2 error}, this result is not stable 
    with respect to the random instance of the GP: the minimizer of the $L^2$ error fluctuates wildly in our experiments.
    
    \begin{figure}[ht]
    \centering
    \includegraphics[width=9cm]{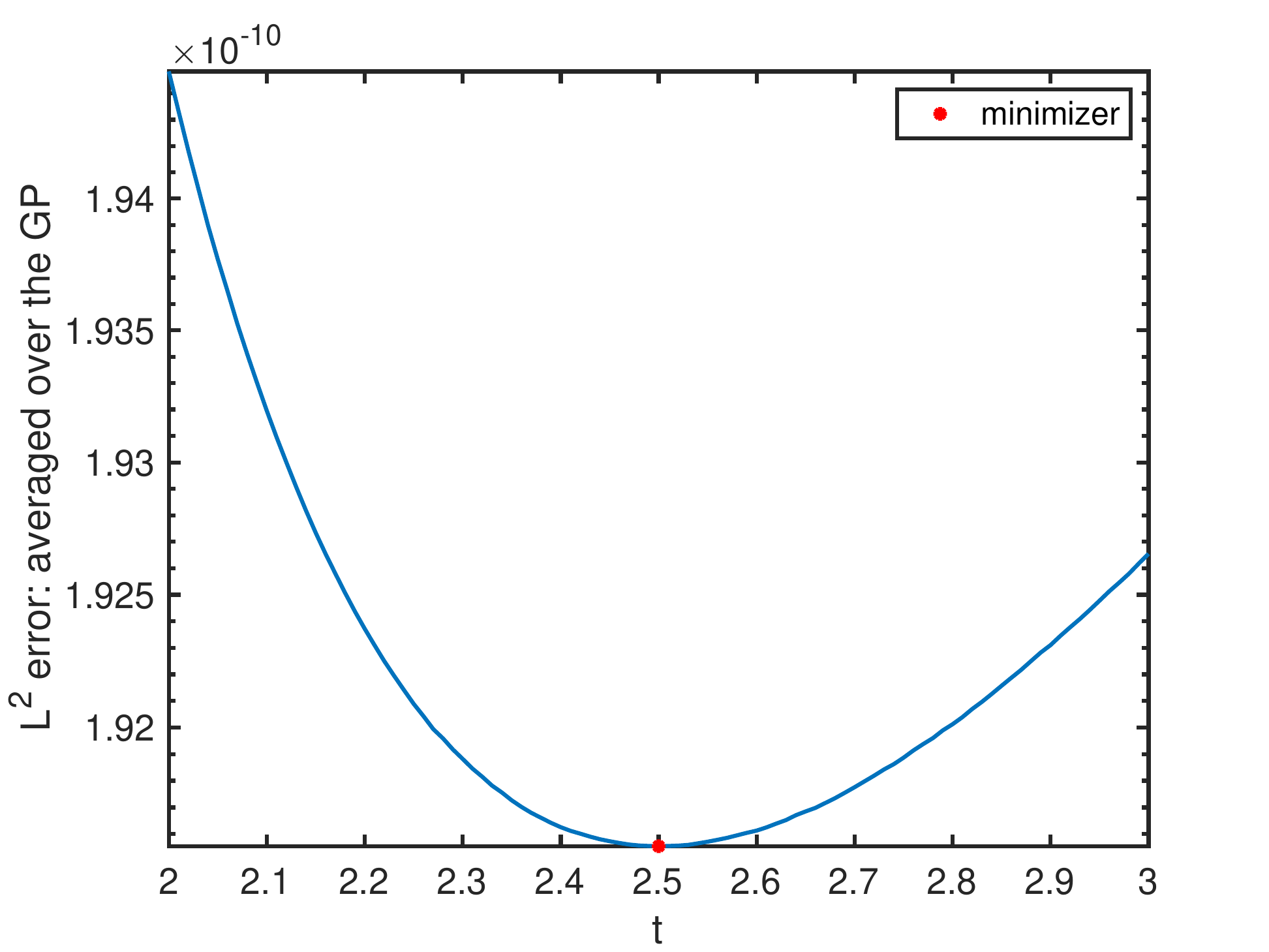}
    \caption{$L^2$ error: averaged over the GP, for $q=9$}
    \label{fig: L^2 error: ave GP}
    \end{figure}

    In summary, the second set of numerical experiments indicates the following implications for the regression accuracy of the EB and KF approaches to hierarchical parameter estimation. The KF estimator selects the minimal $t$ that suffices to achieve the fastest rate of approximation error in the $L^2$ norm for a given fixed truth; in contrast, the EB estimator converges to the $t$ that achieves the minimal $L^2$ error, averaged over the draw $\ud \in \cN(0,(-\Delta)^{-s})$. Note that KF is based on purely approximation
    theoretic considerations whilst EB is founded on statistical considerations --- they attain very different implicit bias in selecting parameters.
    
    \subsubsection{Further Discussion of The Theory} We provide some further discussions of the implications of Theorem~\ref{thm: consistency thm} in this subsection. The theory shows that the EB estimator recovers the ground truth parameter $s$ of the statistical model. This is in line with expectations since the methodology is designed to recover the most likely value of $s$,
    given the data, and since the Gaussian measures occurring for different $s$ are mutually singular. In the literature, such consistency results are primarily for observational data in the Fourier domain; thus, the observation operator commutes with the prior. Here, our data model is in the physical domain, which leads to the need for considerably more sophisticated analysis, due to the noncommutativity of the observation operator and the prior operator, and yet is a much more practically useful setting, justifying the investment in the somewhat involved analysis. Our proof provides a novel sharp upper and lower bound on the terms $\|u(\cdot,t,q)\|_t^2$ and $\log \det K(t,q)$, based on techniques in approximation theory and the multiresolution analysis developed in \cite{owhadi2019operator}. Our techniques may have broader applications in analyzing the observational model in the physical domain.
    
    Another interesting phenomenon shown in Theorem~\ref{thm: consistency thm} is that the KF estimator, first proposed in \cite{owhadi2019kernel} as a method to learn kernels for machine learning tasks, achieves a rather different consistency behavior, with the large data limit being $\frac{s-d/2}{2}$. This fact has the following consequence: if the ground truth function $\ud$ has homogeneous critical regularity $s-d/2$, then the KF estimator will converge to half the critical
    regularity in the large data limit. 
    
To understand the mechanism behind this effect, we observe that the KF loss is a surrogate for the (relative) $\|\cdot\|_t$-norm approximation error between $u^\dagger$ and $u(\cdot,t,q)$. Furthermore, approximation theory implies that the 
    GP regressor $u(\cdot,t,q)$ is also the optimal $\|\cdot\|_t$-norm approximant of $u^\dagger$ in the linear span of the basis functions $\{(-\Delta)^{-t}\updelta(x-x_j) \}_{j \in J_q}$. Under this perspective, we see the KF loss incorporates two competing factors in the approximation:  increasing $t$ improves the approximation error by increasing the regularity of the basis functions while worsening the measurement of that approximation error by using a stronger norm. 
    The balance between these two competing factors is achieved when $t$ is half the critical regularity, which is the parameter that KF eventually picks. Our proof provides a detailed demonstration of this phenomenon.
    
    In short, EB learns hierarchically based on statistical principles,
    whilst KF learns based on approximation theoretic ones. The consistency results presented here provide evidence that the interplay between statistical estimation and numerical approximation can be very useful for parameter estimation and kernel learning in general, thus suggesting new ways of thinking hierarchically. This perspective is one of the main messages that we convey in this paper.
    
     \subsubsection{Proof Strategy}
     The following Subsections \ref{sec: Fourier}, \ref{sec: consistency EM estimator}, \ref{sec: consistency KF estimator} are devoted to proving the above Theorem \ref{thm: consistency thm}. For the sake of understanding, we provide a high-level view of our proof strategies in this subsection. Fourier analysis plays an important role in the proof. It allows us to analyze the approximation error in a very precise way under this equidistributed design setting.
     
     In our proof, we begin by establishing tight bounds on the terms that appear in the objective functions, i.e., $\|u(\cdot,t,q)\|_t^2$, $\log \det K(t,q)$ and $\|u(\cdot,t,q)-u(\cdot,t,q-1)\|_t^2$, using the toolkit we develop in Subsection \ref{sec: Fourier}. 
    The norms $\|u(\cdot,t,q)\|_t^2$ and $\|u(\cdot,t,q)-u(\cdot,t,q-1)\|_t^2$ are expressed
    as random (as a function of $u^\dagger$) series and we carefully analyze the dependencies of the random variables to establish the convergence in probability. For $\log \det K(t,q)$, we employ the multiresolution approach introduced in \cite{owhadi2019operator} to establish a tight estimate of the spectrum of the Gram matrix from below and above. Given these estimates, we provide an intuitive understanding of how the loss functions behave and how the minimizers converge in Subsections \ref{sec: consistency EM estimator}, \ref{sec: consistency KF estimator}. In the rigorous treatment, the sharp bounds on the different components of the objective functions will be combined with the uniform convergence result of random series in \cite{van1996weak} to obtain the convergence of minimizers.
    
    \subsubsection{Notations}  In many parts of the analysis, we need to develop tight estimates on the terms appearing in the loss functions. Some useful notation for comparing different terms are introduced here. We write $A \simeq B$ if there exists a constant $C$ independent of $q, t$ such that 
    \[\frac{1}{C}B\leq A\leq CB\,. \]
    The constant may depend on the dimension $d$ and on $\delta$. Correspondingly, if we use $A \gtrsim B$ or $A \lesssim B$, then only one side of the above inequality holds. 
    
    Fourier analysis plays a critical role in the analysis. We always use $\ud$ for the ground truth function, while we omit the $\dagger$ symbol for ease of notation when discussing its Fourier transform, and write $\hat{u}$; we will also use $\hat{u}$, with
    more arguments, to denote the Fourier transform of the Gaussian process
    mean; see the discussion following Theorem \ref{thm: Fourier representation for Pu}. In the Fourier domain, we let $B_q:=\{m \in \bZ: -2^{q-1}\leq m \leq 2^{q-1}-1\}$ and $B_q^d=B_q\otimes B_q\otimes \cdots\otimes B_q$ be the tensor product of $d$ multiples of $B_q$. {We have that $B_q^d$ is a box concentrating around the origin, so only the low-frequency part of the Fourier coefficients are considered. }

    \subsection{Toolkit: Fourier Series Characterization} \label{sec: Fourier}
    In this subsection, we prepare the necessary tools that are used to prove the main theorem of this paper. 
    
    We start by establishing a Fourier series characterization for $u(\cdot,t,q)$. This is a key ingredient in expressing the terms in the loss functions as random series. Our approach,
    using Fourier series, is motivated by the papers \cite{de1994approximation, ron_l2-approximation_1992}, where the approximation power of shift-invariant subspaces of $L^2(\bR^d)$ is studied; in our case we use related ideas in the $\dot{L}^2(\bT^d)$ setting.
    
    To find the representation of the term $u(\cdot,t,q)$, we invoke its definition, i.e. $u(\cdot,t,q)$ is obtained by GP regression with the $q$-level data and the covariance function $(-\Delta)^{-t}$. We use the representer theorem from GPR. Concretely, let the set of basis functions be \[\cF_{t,q}=\text{span}_{j \in J_q} \{(-\Delta)^{-t}\updelta(\cdot-x_j)\}\, ,\] then, $u(\cdot,t,q)$ is the best approximation in $\cF_{t,q}$ to the true function under the $\|\cdot\|_t$ norm. Let us define \[\hat{\cF}_{t,q}:=\{g: \bZ^d \to \bC, \text{there exists an} \ f \in \cF_{t,q}\  \text{such that}\  g=\hat{f}  \},\] the Fourier coefficients of functions in $\cF_{t,q}$. A quick observation is that for every $g \in \hat{\cF}_{t,q}$, we must have $g(0)=0$ because of the mean zero property of $f \in \cF_{t,q}$. The following proposition gives a complete characterization of the basis functions in $\hat{\cF}_{t,q}$, for $t > d/2$. 
    \begin{proposition}
        \label{prop: characterize F(t,q)}
        For any $g \in \hat{\cF}_{t,q}$, there exists a $2^q$-periodic function $p$ on $\bZ^d$, such that
        \begin{equation*}
        g(m)=\begin{cases}
        |m|^{-2t}p(m), &m \neq 0\\
        0, &m=0\, .
        \end{cases}
        \end{equation*} 
    \end{proposition}
        The proof is in Subsection \ref{Proof of Proposition prop: characterize F(t,q)}. Next, we define a $2^q$-periodization operator, which will be used to compute the representation of $\hat{u}(m,t,q)$.
        \begin{definition}
            \label{def: periodization}
            The operator $T_q$ is defined as a mapping from the space of functions on $\bZ^d$ to itself, such that
            \[(T_qg)(m):=\sum_{\beta \in \bZ^d}g(m+2^q\beta),\quad m \in \bZ^d\, , \]
         whenever the right hand side series converges for the function $g: \bZ^d \to \bR$. We also define
         \begin{equation}
         \label{eqn: def of M_q^t(m)}
         M_q^t(m):=
         \begin{cases} \sum_{\beta \in \bZ^d \backslash \{0\}} |2^q\beta|^{-2t}, & \text{if}  \ m=j\cdot2^q \ \text{for some} \ j \in \bZ^d\\
         \sum_{\beta \in \bZ^d} |m+2^{q}\beta|^{-2t}, & \text{else}\, .
         \end{cases}
         \end{equation}
        \end{definition}
        
        Both $T_q g$ and $M_q^t$ are $2^q$-periodic functions on $\bZ^d$. 
        Based on this definition, Theorem \ref{thm: Fourier representation for Pu} presents the explicit form of the Fourier transform of $u(\cdot,t,q)$; the proof is in Subsection \ref{Proof of Theorem thm: Fourier representation for Pu}. The proof relies on the Galerkin orthogonality property of $u(\cdot,t,q)$ due to its being the optimal approximate solution. 
        
        \begin{theorem}
        \label{thm: Fourier representation for Pu}
        Let $\hat{u}(\cdot,t,q)$ be the Fourier coefficients of $u(\cdot,t,q)$, then for $m \in \bZ^d$, we have
        \[\hat{u}(m,t,q)=
        \begin{cases}
        0, & \text{if} \ m = 0\\
        |m|^{-2t}\frac{(T_q\hat{u})(m)}{M_q^t(m)}, & \text{else}
        \end{cases} \]
        where $\hat{u}$ denotes the Fourier coefficients of $\ud$.
    \end{theorem}
    
    This above representation is very useful for analyzing the terms $\|u(\cdot,t,q)\|_t^2$ and $\|u(\cdot,t,q)-u(\cdot,t,q-1)\|_t^2$. 
    As well as studying the Fourier coefficients of $u(\cdot,t,q)$, which
    we denote by $\hat{u}(\cdot,t,q)$, we will also need to study the Fourier coefficients of $\ud(\cdot)$ which, for ease of notation we 
    will denote by $\hat{u}(\cdot)$, henceforth, omitting the $\dagger$ symbol. It is thus important to look at the number of arguments
    of $\hat{u}$ to determine which object it is the Fourier transform of.
    Note also that $u(\cdot,t,q)$ is determined by $\ud$; hence if
    $\ud$ is random, so is $u(\cdot,t,q)$.
    
    We will use the above Fourier analysis toolkit to study the consistency of EB and KF in the following two subsections.
    \subsection{Proof for the Empirical Bayesian Estimator}
    \label{sec: consistency EM estimator}
    In this subsection, we prove the consistency of the EB estimator. As explained before, our roadmap is to give a tight estimate of the loss functions first and then analyze the minimizers. 
   For the norm term $\|u(\cdot,t,q)\|_t^2$, we invoke Theorem \ref{thm: Fourier representation for Pu}, based on which 
   this term is expressed as a random series:
    \begin{proposition}
        \label{prop: representation of H^t norm of u(,t,q)}
        The $\dot{H}^t(\bT^d)$ norm of $u(\cdot,t,q)$ has the representation
        \[\|u(\cdot,t,q)\|_t^2=(4\pi^2)^{t}\sum_{m \in B_q^d} \frac{|T_q\hat{u}(m)|^2}{M_q^t(m)}\, . \]
        Moreover, suppose $\ud \sim \cN(0,(-\Delta)^{-s})$ for $s > \frac{d}{2}$, then
        \[\|u(\cdot,t,q)\|_t^2=(4\pi^2)^{t-s}\sum_{m \in B_q^d} \frac{M_q^s(m)}{M_q^t(m)}\xi_m^2\, , \]
        where $\{\xi_m \}_{m \in B_q^d}$ are independent unit scalar Gaussian random variables.
    \end{proposition}
    \begin{proof}
        Using Theorem \ref{thm: Fourier representation for Pu}, we get
        \begin{align*}
        \|u(\cdot,t,q)\|_t^2=&\sum_{m \in \bZ^d \backslash \{0\}} (4\pi^2)^{t}|m|^{2t}|\hat{u}(m,t,q)|^2\\
        =&(4\pi^2)^{t}\sum_{m \in \bZ^d \backslash \{0\}}|m|^{-2t}\frac{|T_q\hat{u}(m)|^2}{|M_q^t(m)|^2}\\
        =&(4\pi^2)^{t} \sum_{m \in B_q^d} M_q^t(m)\frac{|T_q\hat{u}(m)|^2}{|M_q^t(m)|^2}\\
        =&(4\pi^2)^{t}\sum_{m \in B_q^d} \frac{|T_q\hat{u}(m)|^2}{M_q^t(m)}\, .
        \end{align*}
        where in the third equality, we use the periodicity of the function $\frac{|T_q\hat{u}(m)|^2}{|M_q^t(m)|^2}$. 
        
        If we further assume $\ud \sim \cN(0,(-\Delta)^{-s})$, then $\hat{u}(m) \sim \cN(0,(4\pi^2)^{-s}|m|^{-2s})$. For different $m$, these Gaussian random variables are independent. Thus, for different $m \in B_q^d$, we have $T_q\hat{u}(m) \sim \cN(0,(4\pi^2)^{-s}M^s_q(m))$, and they are independent. So we can write
        \[\sum_{m \in B_q^d} \frac{|T_q\hat{u}(m)|^2}{M_q^t(m)}= (4\pi^2)^{-s}\sum_{m \in B_q^d} \frac{M_q^s(m)}{M_q^t(m)}\xi_m^2\, ,\]
        where $\{\xi_m \}_{m \in B_q^d}$ are independent unit scalar Gaussian random variables. 
    \end{proof}
    
    The independence of the random variables established in the preceding
    representation is crucial for the analysis.  The terms $M_q^s(m), M_q^t(m)$ appear in the preceding; to analyze them we present a useful lemma below. The proof is in Subsection \ref{Proof of Lemma lemma: estimate of the M t q term}.
    \begin{lemma}
    \label{lemma: estimate of the M t q term}
    For $t \in [d/2+\delta, 1/\delta]$ and $q \geq 0$, we have 
    \begin{equation*}
        M_q^t(m)\simeq \begin{cases}
        2^{-2qt},\ \text{if} \ m = 0\\
        |m|^{-2t},\  \text{if} \ m \in B_q^d \backslash \{0\}
        \end{cases}
    \end{equation*}
    Moreover, for $m \in B_q^d \backslash \{0\}$, we have $M_q^t(m)-|m|^{-2t} \simeq 2^{-2qt}$.
    \end{lemma}
    Now, we are ready to get the estimates of the loss function. The following proposition shows an upper and lower bound on the norm term. 
    \begin{proposition}[Bound on the norm term]
        \label{prop: |u(x,t,q)|^2 bound}
        Suppose $\ud$ is a sample drawn from the Gaussian process $\cN(0,(-\Delta)^{-s})$ for $d/2+\delta \leq s \leq 1/\delta$, then
        \[\|u(\cdot,t,q)\|_t^2 \simeq 2^{-q(2s-2t)}\xi_0^2+\sum_{m \in B_q^d \backslash \{0\}} |m|^{2t-2s}\xi_m^2\, ,\]
        where $\{\xi_m \}_{m \in B_q^d}$ are independent unit scalar Gaussian random variables.
    \end{proposition}
    \begin{proof}
        According to Lemma \ref{lemma: estimate of the M t q term},  for $m \in B_q^d \backslash \{0\}$, we have $M_q^t(m) \simeq |m|^{-2t}$; for $m = 0$, we have $M_q^t(m) \simeq 2^{-2tq}$. Thus,
        \begin{align*}
        \|u(\cdot,t,q)\|_t^2&=(4\pi^2)^{t-s}\sum_{m \in B_q^d} \frac{M_q^s(m)}{M_q^t(m)}\xi_m^2\\
        &=(4\pi^2)^{t-s}\left(\sum_{m \in B_{q}^d \backslash \{0\}}\frac{M_q^s(m)}{M_q^t(m)}\xi_m^2 + \frac{M_q^s(0)}{M_q^t(0)}\xi_0^2\right)\\
        &\simeq 2^{-q(2s-2t)}\xi_0^2+\sum_{m \in B_q^d \backslash \{0\}} |m|^{2t-2s}\xi_m^2\, .
        \end{align*}
        This completes the proof.
    \end{proof}
    Proposition \ref{prop: |u(x,t,q)|^2 bound} states that the behavior of the norm term is nothing but a weighted sum of squares of independent Gaussian random variables, which is amenable to
    analysis. With this in mind, we state a lemma useful in the analysis of such random series,
    with proof deferred to Subsection \ref{Proof of Lemma lemma: uniform convergence of series}.
    \begin{lemma}
        \label{lemma: uniform convergence of series}
        Suppose $\{\xi_m\}_{m \in \bZ^d}$ are independent unit Gaussian random variables. 
        \begin{itemize}
            \item For $r > 0$, define the random series 
                \[ \alpha(r,q)=2^{-qr}\sum_{m \in B_q^d \backslash \{0\}} |m|^{r-d}\xi_m^2\, . \]
                Fix $\epsilon>0$, then there exists a function $\gamma(r) > 0$ such that $\lim_{q\to \infty} \alpha(r,q) = \gamma(r)>0$ uniformly for $r \in [\epsilon,1/\epsilon]$, where the convergence is in probability.
            \item For $r=0$, define 
        \[\alpha(0,q)=\frac{1}{q}\sum_{m \in B_q^d \backslash \{0\}} |m|^{-d}\xi_m^2\, , \]
        then there exists $\gamma(0)\in (0,\infty)$ such that $\lim_{q \to \infty} \alpha(0,q)=\gamma(0)$ in probability.
        \end{itemize} 
    \end{lemma}
    
    We then move to the second term in the loss function, i.e., the log determinant term. It is deterministic and to study it we need a way of analyzing the spectrum of the Gram matrix. The following Proposition \ref{prop: bound on the log term} gives upper and lower bounds on this term. The proof is in Subsection \ref{Proof of Proposition prop: bound on the log term} and is motivated by analysis developed in the paper \cite{owhadi2019operator}. The idea is to use the Schur complement of the Gram matrix and rely on the variational characterization of the Schur complement to get a tight control on the spectrum. {This technique is quite general and has been used in \cite{owhadi2019operator} to characterize the spectrum of heterogeneous Laplacian operators; here we adapt it to fractional operators. On the other hand, for the homogeneous fractional Laplacian operators in this paper, it is also possible to calculate an explicit formula for the spectrum of $K(t,q)$, as has been used in Section 6.7 of \cite{stein99book}. We 
    describe this simple proof in Subsection \ref{Proof of Proposition prop: bound on the log term} but retain the
    proof employing the more general methodology as it may be useful for other problems.}
    \begin{proposition}[Bound on the $\log \det$ term] 
        \label{prop: bound on the log term}
        For $d/2+\delta \leq t\leq 1/\delta$, we have 
        \[(2t-d)g_1(q)-Cg_2(q)+K(t,0)\leq \log \det K(t,q) \leq (2t-d)g_1(q)+Cg_2(q)+K(t,0) \, , \]
        where $g_1(q)=\sum_{k=1}^q (2^{kd}-2^{(k-1)d})(-k\log 2)$ and $g_2(q)=(2^{qd}-1)(2t-d)$. The constant $C$ is independent of $t,q$. Moreover, $g_1(q)\simeq -q2^{qd}$.
    \end{proposition}
    
   With the loss function analyzed by the above results, the consistency of
   the EB estimator is readily stated as follows.
    \begin{theorem}[Consistency of Empirical Bayesian estimator]
        \label{thm: Consistency of Empirical Bayesian estimator}
        Fix $\delta >0$. Suppose $\ud$ is a sample drawn from the Gaussian process $\cN(0,(-\Delta)^{-s})$. If
        $s \in [d/2+\delta,1/\delta]$ then
        \[\lim_{q \to \infty} s^{\mathrm{EB}}(q)=s\quad \mathrm{in\ probability}\, . \]
    \end{theorem}
    The detailed proof is in Subsection \ref{Proof of Theorem thm: Consistency of Empirical Bayesian estimator}. We can understand the theorem intuitively by using the established results above. Recall there are two terms in the loss function: (1) the norm term $\|u(\cdot,t,q)\|_t^2$; (2) the log det term. For the norm term, from Proposition \ref{prop: |u(x,t,q)|^2 bound} and Lemma \ref{lemma: uniform convergence of series}, its behavior for $q \to \infty$ is roughly
    \begin{itemize}
        \item Growing like $2^{q(2t-2s+d)}$ if $t>s-d/2$;
        \item Growing like $q$ if $t=s-d/2$;
        \item Remaining bounded if $t < s- d/2$.
    \end{itemize}
    The log det term  decreases like $-(2t-d)q2^{qd}$ according to Proposition \ref{prop: bound on the log term}. Noticing that the EB loss function has the form
    \[\mathsf{L}^{\text{EB}}(t,q)=\|u(\cdot,t,q)\|_t^2 + \log \det K(t,q)\, , \]
    we arrive at the following intuitive observations:
    \begin{itemize}
        \item When $t < s$, the dominant behavior of $\mathsf{L}^{\text{EB}}(t,q)$ is controlled by the log determinant term, since the growth rate of the norm term $2^{q(2t-2s+d)} = o(q2^{qd})$. As a consequence, $\mathsf{L}^{\text{EB}}(t,q)$ exhibits the overall behavior $-(2t-d)q2^{qd}$. Therefore, the loss function decreases linearly with $t$  in this regime. This is consistent
        with what is observed in  Figure~\ref{fig: 1 d example}.
        \item When $t \geq s$, the increasing speed of the norm term beats the decreasing rate of the log det term, so the norm term dominates the behavior of $\mathsf{L}^{\text{EB}}(t,q)$. Overall, it is like  $2^{q(2t-2s+d)}$, which increases exponentially with $t$; again
        this is consistent with what is observed in Figure~\ref{fig: 1 d example}.
    \end{itemize}
    According to the above observations, the minimizer of $\mathsf{L}^{\text{EB}}(t,q)$ will converge to $s$. To make the intuition leading to
    this conclusion rigorous, we need to use techniques of uniform convergence for random series. For details we refer to Subsection \ref{Proof of Theorem thm: Consistency of Empirical Bayesian estimator}.
    
    \subsection{Proof for the Kernel Flow Estimator} 
    \label{sec: consistency KF estimator}
    In this subsection, we establish the consistency of the KF estimator. As before, we start by estimating the growth behavior of terms that appear in the loss function. We begin with the interaction term $\|u(\cdot,t,q)-u(\cdot,t,q-1)\|_t^2$. Similar to the analysis of the norm term in the preceding subsection, we represent it by using Fourier series.
    \begin{proposition}
        The $\dot{H}^t(\bT^d)$ norm of $u(\cdot,t,q)-u(\cdot,t,q-1)$ has the representation
        \begin{equation}
        \label{eqn: uq-uq-1 representation}
        \|u(\cdot,t,q)-u(\cdot,t,q-1)\|_t^2=(4\pi^2)^t\sum_{m \in B_q^d} M_q^t(m)\left(\frac{T_q\hat{u}(m)}{M_q^t(m)}-\frac{T_{q-1}\hat{u}(m)}{M_{q-1}^t(m)}\right)^2\, .
        \end{equation}
    \end{proposition}
    \begin{proof}
        By Theorem \ref{thm: Fourier representation for Pu}, we have
        \[\hat{u}(m,t,q)-\hat{u}(m,t,q-1)=\begin{cases}
        0, & \text{if} \ m = 0\\
        |m|^{-2t}\left(\frac{T_q\hat{u}(m)}{M_q^t(m)}-\frac{T_{q-1}\hat{u}(m)}{M_{q-1}^t(m)}\right)
        , & \text{else} \, .
        \end{cases} \]
        Thus, 
        \begin{align*}
        \|u(\cdot,t,q)-u(\cdot,t,q-1)\|_t^2
        =&(4\pi^2)^t\sum_{m \in \bZ^d \backslash \{0\}} |m|^{2t}|\hat{u}(m,t,q)-\hat{u}(m,t,q-1)|^2\\
        =&(4\pi^2)^t\sum_{m \in \bZ^d \backslash \{0\}} |m|^{-2t}\left(\frac{T_q\hat{u}(m)}{M_q^t(m)}-\frac{T_{q-1}\hat{u}(m)}{M_{q-1}^t(m)}\right)^2\\
        =&(4\pi^2)^t\sum_{m \in B_q^d} M_q^t(m)\left(\frac{T_q\hat{u}(m)}{M_q^t(m)}-\frac{T_{q-1}\hat{u}(m)}{M_{q-1}^t(m)}\right)^2\, .
        \end{align*}
    \end{proof}    
    By carefully studying the correlation between the random variables appearing in the preceding proposition, we obtain lower and upper bounds
    in the following two propositions;  
    proofs can be found in Subsections \ref{Proof of Proposition prop: Lower bound on the interaction term} and \ref{Proof of Proposition prop: Upper bound on the interaction term}.
    
    \begin{proposition}[Lower bound on the interaction term]
        \label{prop: Lower bound on the interaction term}
           Suppose $\ud$ is a sample drawn from the Gaussian process $\cN(0,(-\Delta)^{-s})$ for $d/2+\delta \leq s \leq 1/\delta$, then
        \[\|u(\cdot,t,q)-u(\cdot,t,q-1)\|_t^2 \gtrsim \sum_{m \in B_{q-1}^d \backslash \{0\}}2^{-2tq}|m|^{4t-2s}\xi_{m}^2 \, , \]
        where $\{\xi_{m} \}_{m \in B_{q-1}^d\backslash \{0\}}$ are independent unit scalar Gaussian random variables.
    \end{proposition}
    The upper bound has a more complex form. We introduce the notation $\bZ^d_2=\{0,1\}^d$ comprising $d$ dimensional vectors with each component being in $\{0,1\}$. In the following proposition, we also use the convention that $|m|^{\alpha} = 0$ for $m=0$ and any $\alpha \in \bR$ to make the notation more compact.
    \begin{proposition}[Upper bound on the interaction term]
        \label{prop: Upper bound on the interaction term}
        Suppose $\ud$ is a sample drawn from the Gaussian process $\cN(0,(-\Delta)^{-s})$ for $d/2+\delta \leq s \leq 1/\delta$, then
        \[\|u(\cdot,t,q)-u(\cdot,t,q-1)\|_t^2 \lesssim \sum_{k \in \bZ_2^d} \sum_{m \in B_{q-1}^d} (2^{-q(2s-2t)}+2^{-2tq}|m|^{4t-2s})\xi_{k,m}^2 \, ,\]
        where for a fixed $k \in \bZ^d_2$, $\{\xi_{k,m} \}_{m \in B_{q-1}^d}$ are independent unit scalar Gaussian random variables.
    \end{proposition}
    We remark that in the upper bound, the random variables for different $k$ may exhibit correlation. However, since the term $\sum_{m \in B_{q-1}^d} (2^{-q(2s-2t)}+2^{-2tq}|m|^{4t-2s})\xi_{k,m}^2$ has the same form for each $k$, and the number of different $k$ is finite, it suffices to analyze the random series for a single $k$, in which we have the independence of random variables. The theorem is stated below.
    \begin{theorem}[Consistency of the Kernel Flow estimator]
    \label{thm: Consistency of the Kernel Flow estimator}
    Fix $\delta >0$. Suppose $\ud$ is a sample drawn from the Gaussian process $\cN(0,(-\Delta)^{-s})$. If
         $\frac{s-d/2}{2} \in [d/2+\delta,1/\delta]$ then for the Kernel Flow estimator, 
    \[\lim_{q \to \infty} s^{\mathrm{KF}}(q)=\frac{s-d/2}{2}\quad \mathrm{in\ probability}\, . \]
    \end{theorem}
    The idea behind the proof of the theorem is to combine Propositions \ref{prop: Lower bound on the interaction term}, \ref{prop: Upper bound on the interaction term} and Lemma \ref{lemma: uniform convergence of series}. Together they imply the growth behavior of the loss function \[\mathsf{L}^{\text{KF}}(t,q)= \frac{\|u(\cdot,t,q)-u(\cdot,t,q-1)\|_t^2}{\|u(\cdot,t,q)\|_t^2}\]
    as follows:
    \begin{itemize}
        \item When $t < \frac{s-d/2}{2}$, the numerator decays like $2^{-2tq}$ since $4t-2s<-d$, in which case the summation $\sum_{m \in B_q^d \backslash \{0\}} |m|^{4t-2s}\xi_m^2$ remains bounded. The denominator remains bounded. So the overall behavior is $2^{-2tq}$.
        \item When $\frac{s-d/2}{2}<t<s-d/2$, the numerator decays like $2^{-2tq}\times 2^{q(4t-2s+d)}=2^{q(2t-2s+d)}$ according to Lemma \ref{lemma: uniform convergence of series}. The denominator remains bounded, The overall behavior is $2^{q(2t-2s+d)}$. 
        \item When $t>s-d/2$, the numerator behaves like $2^{q(2t-2s+d)}$, while the denominator behaves like $2^{q(2t-2s+d)}$. The overall behavior is of order $1$.
    \end{itemize}
    These observations are consistent with what is observed in Figure~\ref{fig: 1 d example}. Based on them we deduce that the minimizer converges to $\frac{s-d/2}{2}$. The loss function exhibits symmetric behavior with respect to $\frac{s-d/2}{2}$ for $t \in (d/2,s-d)$. The detailed rigorous treatment is presented in Subsection \ref{Proof of Theorem thm: Consistency of the Kernel Flow estimator}.
    {\subsection{Discussions} 
    \label{sec:regularity recovery, discussion} In the preceding three subsections, we have presented the consistency theory, its implication for implicit bias, as well as the tools and strategies underlying our proofs. This subsection adds to several discussions on
    the theory and proofs.}
    
    {First, our theory applies to the torus domain. One may wonder whether these techniques can be applied to boundary conditions beyond the periodic ones.
    The main tool used in the proofs is Fourier's series (based on the eigenfunctions of the Laplacian operator). These are used to characterize the norm term and determinant term. We expect these techniques to generalize to other problems,
    such as the box with Dirichlet or Neumann boundary conditions in which
    the Fourier sine or cosine series are natural; the detailed analysis is left as future work. However, we need to point out that the limitation of this proof idea is that it requires a clear analytic understanding of the spectral properties of the kernel operator, i.e., its eigenfunctions. In Subsection \ref{subsubsec: Recovery of regularity parameter for variable coefficient elliptic operator}, we present numerical experiments beyond this setting, which involves more challenging Laplacians with discontinuous coefficients that can model more complicated heterogeneous random fields.}
    
    {Second, this section considers the regularity parameter only. In spatial statistics literature, consistency results on this parameter (for general Mat\'ern type model) are very scarce and difficult. Here, we obtain a proof for the torus model, which is the main technical contribution of this paper. We will discuss the learning of other parameters in the next section, to make the story of the Mat\'ern-like model on the torus more complete.}
    
    {Finally, as we get two algorithms that can  ``consistently'' learn the information of the regularity parameter when the number of data is large, a natural question is when to choose which. To answer this question, we presents numerical study of the variances of both estimators for the Mat\'ern-like model in the next subsection.}
    
    \subsection{Variance of Regularity Parameter Estimation}
\label{sec: experiments, variance of estimators}
In this subsection, we compare the variance of the two estimators for recovering the regularity parameter $s$. We return to the experimental set-up in Subsection \ref{section: numerical example for s, s-d/2 /2 demonstration}. We form the EB and KF estimators for $50$ instances of different draws of the GP, normalized by the
limiting optimum values $s$ and $\frac{s-d/2}{2}$ respectively. The statistics of the two estimators are summarized in the histogram (see Figure \ref{fig: Histogram of estimator of s, for different draws}).
\begin{figure}[ht]
    \centering
    \includegraphics[width=6cm]{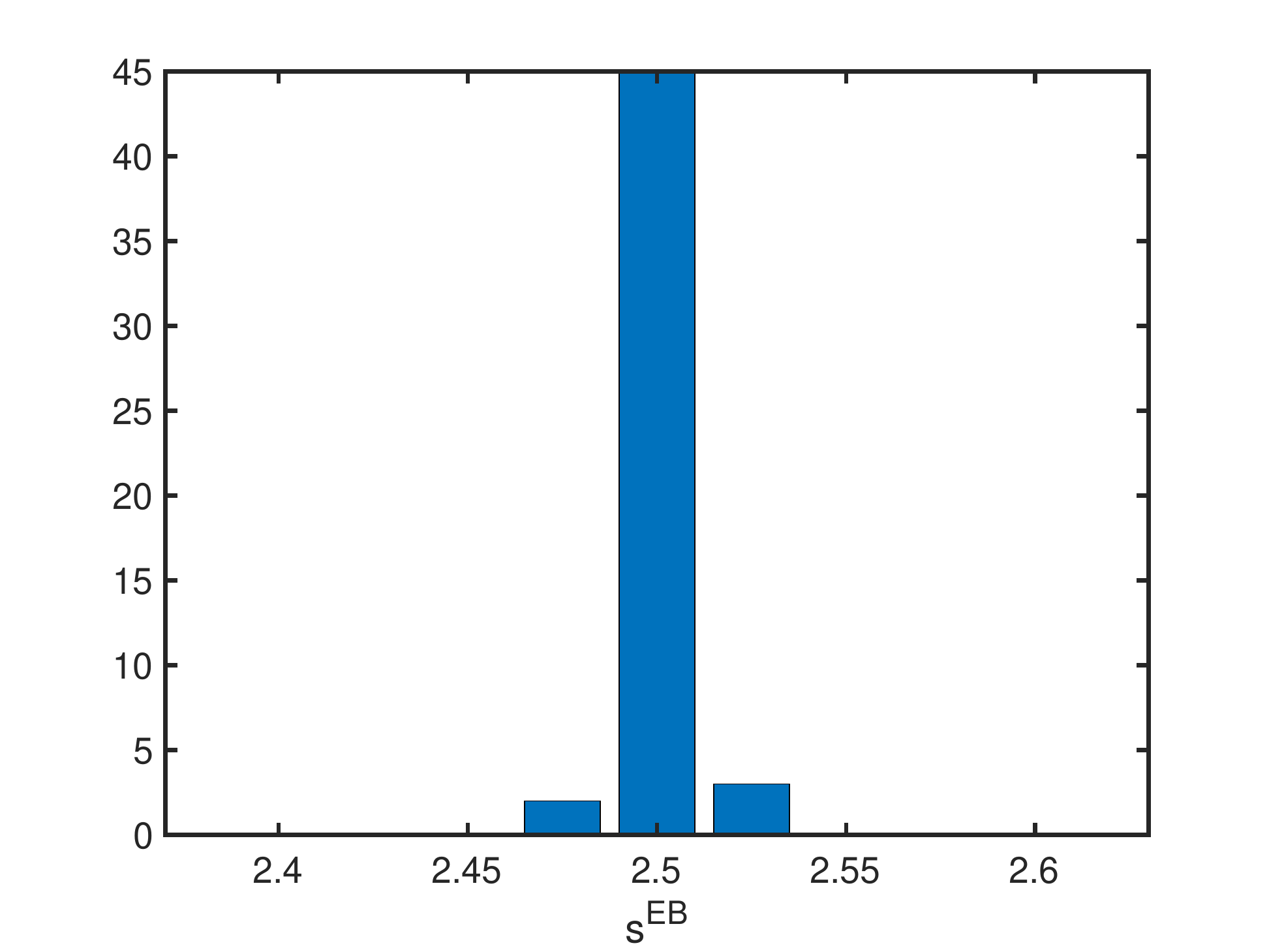}
    \includegraphics[width=6cm]{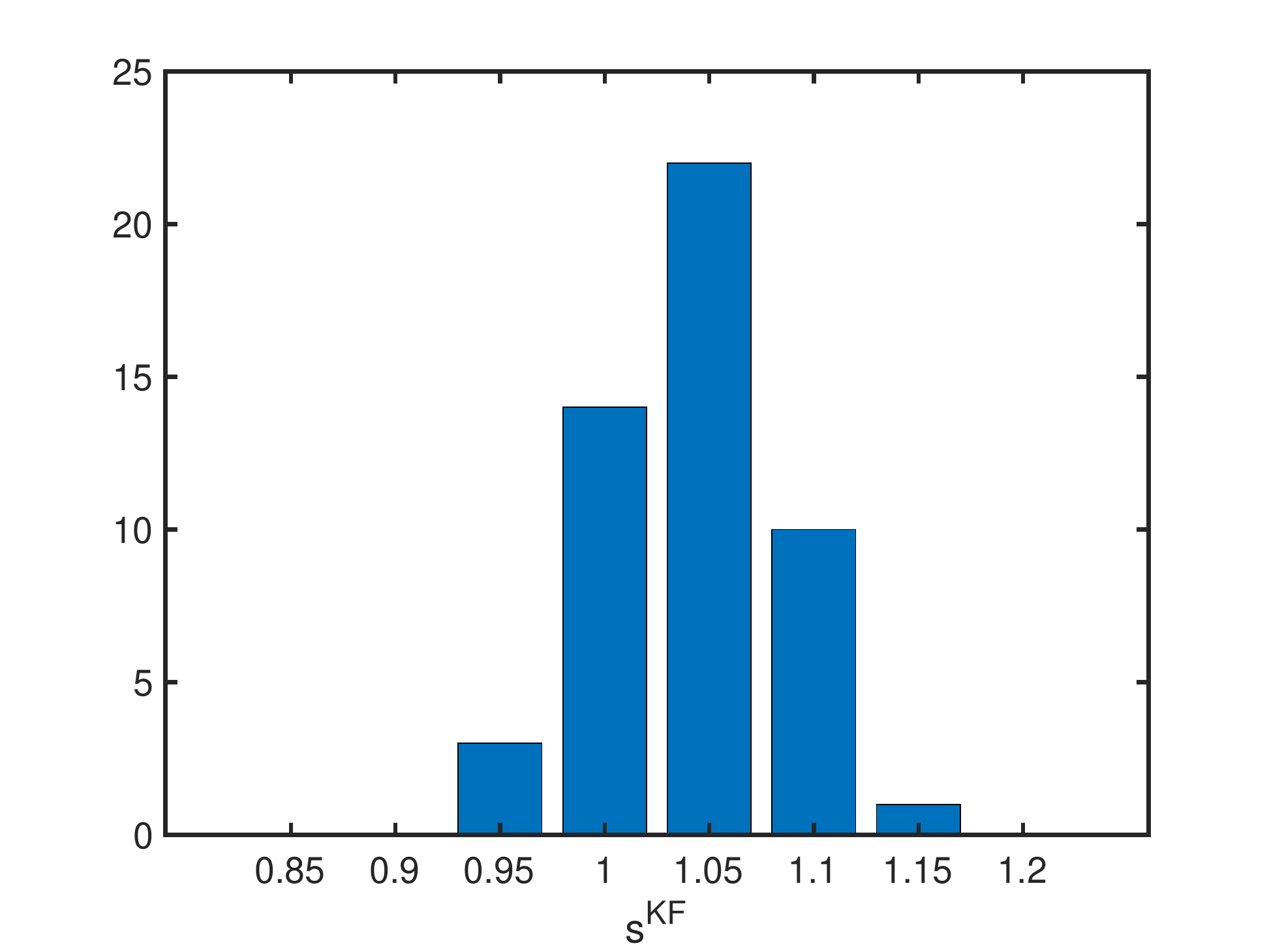}
    \caption{Histogram of the regularity estimators for the Mat\'ern-like process. Left: EB; right: KF}
    \label{fig: Histogram of estimator of s, for different draws}
    \end{figure} 
Clearly, EB exhibits smaller variance than KF. We compute the estimated variance using the $50$ instances. Finally we get 
\[\frac{\text{Var}(s^{\text{EB}})}{s^2} \approx 1.44\times 10^{-5} \quad \text{and}\quad \frac{\text{Var}(s^{\text{KF}})}{\left((s-d/2)/2\right)^2}\approx 3.6 \times 10^{-3}\, . \]
Since the variance of EB is smaller, if our target is to estimate $s$ for the exact GP model, then this suggests that the EB method is preferable.
    
    \section{More Well-specified Examples}

    \label{sec: numerical experiments}

{The setting in Section \ref{sec: recover regularity} concerns regularity parameter of the Mat\'ern-like model only. This section aims to extend this discussion to a wider range of settings by means of numerical experiments. 
First, we study the learning of lengthscale and amplitude parameters in the Mat\'ern-like model in Subsection \ref{sec: recovery of sigma and tau}; these experiments lead to a more complete story for the Mat\'ern-like model on the torus. Then, in Subsection \ref{sec: Other Well-specified Examples}, we consider other well-specified models, extending beyond the Mat\'ern-like process example. In Subsection \ref{sec: experiments, computational aspects}, we also discuss some computational aspects of the EB and KF approaches. }

\subsection{Recovery of Amplitude and Lengthscale}
\label{sec: recovery of sigma and tau}
We start with the learning of amplitude and lengthscale parameters in the Mat\'ern-like model, via either EB or KF method.

In spatial statistics, an important general principle in looking at the recovery of hyperparameters
via EB is to determine whether or not the family of measures are mutually singular with respect
to changes in the parameter to be estimated; learning parameters which give rise to mutually singular families
is usually easy, since different almost sure properties can often be
used to distinguish measures and this can be achieved without an abundance of data; in contrast
those parameters that do not
give rise to mutually singular measures typically require an abundance of realizations to be accurately learned. We illustrate this issue in the context of estimating
one parameter by EB, the changing of which leads to mutually singular
measures, and estimating two parameters by EB, changing one of which leads
to mutual singularity, and the other to equivalence, for the Mat\' ern-like process. We also study analogous
questions about identifiability for the KF method. In all cases
we work with loss functions that are natural generalizations of 
\eqref{eqn: def of estimators, recover regularity, EB},
\eqref{eqn: def of estimators, recover regularity, KF}.

\subsubsection{Recovery of $\sigma$} A first observation is that the KF loss function is invariant under change of $\sigma$, so it cannot recover this parameter. We also note that measures are mutually singular with respect to changes in $\sigma$, and so we do expect to be able to recover
$\sigma$ by EB.
For the EB estimator, we design the experiment as follows. We study whether the EB method can recover $\sigma$ while $s,\tau$ are fixed. In detail, we consider a problem with domain the one dimensional torus $\bT^1$. The Mat\'ern-like kernel has regularity $s=2.5$, amplitude $\sigma=1$ and lengthscale $\tau=0$. We assume the values of $s,\tau$ are known, but not $\sigma$. We want to recover $\sigma$ by seeing a single discretized realization $\ud \sim \cN(0,\sigma^2(-\Delta+\tau^2 I)^{-s})$. The domain $\bT^1$ is discretized into $N=2^{10}$ equidistributed grid points. The data we observe is the values of $\ud$ in $2^9$ equidistributed  points. We build the EB loss function (see equation \eqref{eqn: def of EB estimator of sigma}) and plot the figure for a single instance; see Figure \ref{fig: recover sigma}.
\begin{figure}[ht]
    \centering
    \includegraphics[width=9cm]{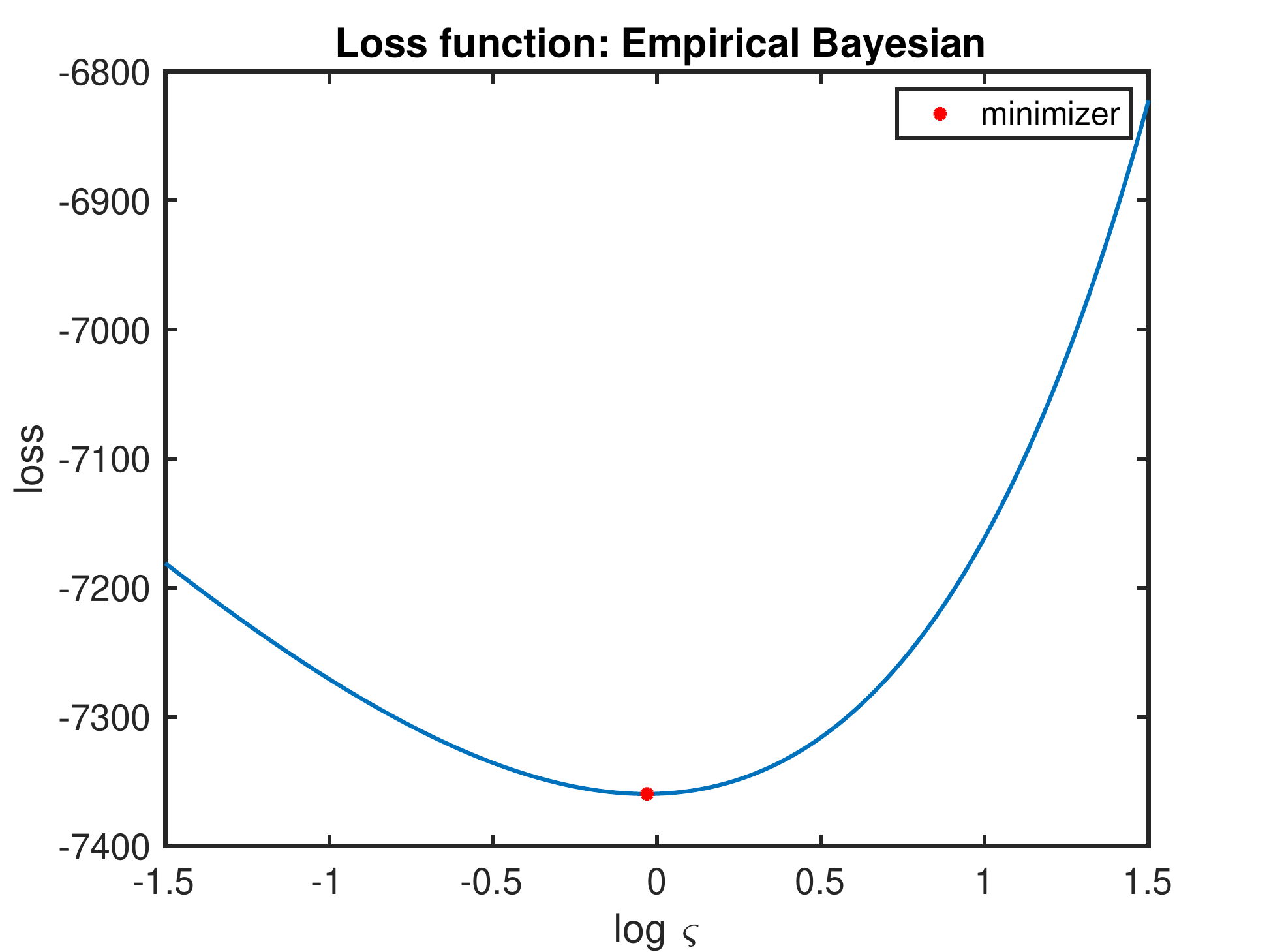}
    \caption{EB loss function for recovering $\sigma$}
    \label{fig: recover sigma}
    \end{figure} 
We introduce $\varsigma$ as the variable to be
maximized over to determine our estimate of $\sigma.$
In our experiments we work with the parameterization $\varsigma=\exp(\varsigma')$ in order to ensure that the estimated
$\sigma$ is positive. Hence, the $x$-axis of  Figure~\ref{fig: recover sigma} is $\varsigma'$.
The figure shows that the minimizer of the loss function is close to the
point $\varsigma'=0$ ($\varsigma=1$), so the estimator $\sigma^{\text{EB}}$ is close to the ground truth $\sigma$.

We can theoretically analyze the convergence. The same set-up in Subsection \ref{sec: set-up of recovering regularity} is adopted, except now we assume the function is drawn from $\cN(0,\sigma^2(-\Delta)^{-s})$ with $s$ known and we want to recover $\sigma$ by seeing the equidistributed spatial samples on the torus. After calculating the likelihood in such a case, we get the EB estimator below. Here we abuse the notation to write 
\begin{equation}
\label{eqn: def of EB estimator of sigma}
\begin{aligned}
    \sigma^{\text{EB}}(q,\ud)=&\argmin_{\varsigma>0} \mathsf{L}^{\text{EB}}(\varsigma,q,\ud),\\
    &\mathsf{L}^{\text{EB}}(\varsigma,q,\ud):=\frac{\sigma^2\|u(\cdot,s,q)\|_s^2}{\varsigma^2} + \log \det K(s,q) + 2^{qd}\log \varsigma^2\, .
    \end{aligned}
\end{equation}
The definition of $u(\cdot,s,q), K(s,q)$ is the same as in Subsection \ref{sec: set-up of recovering regularity}. Recall that $u(\cdot,s,q)$ is the mean of the GP found by conditioning
    a prior measure $\cN(0,(-\Delta)^{-s})$ on observations of $\ud$ at the observation data with level $q$. The definition of $\|\cdot\|_s$ also follows from Subsection \ref{sec: set-up of recovering regularity}. We abuse notation to write $\mathsf{L}^{\text{EB}}(\varsigma,q,\ud)$ for the EB loss function used in the estimation of $\sigma$; the reader should not confuse this with $\mathsf{L}^{\text{EB}}(t,q,\ud)$ in Subsection \ref{sec: set-up of recovering regularity} which is used for recovering the regularity parameter $s$. 

In this setting we have the following consistency result:
\begin{theorem}
Fix $\delta > 0$. Suppose $\ud$ is a sample drawn from the Gaussian process $\cN(0,\sigma^2(-\Delta)^{-s})$ for some $s \in [d/2+\delta,1/\delta]$. 
Then, for the Empirical Bayesian estimator of $\sigma$, it holds that
        \[\lim_{q \to \infty} \sigma^{\mathrm{EB}}(q,\ud)=\sigma\, , \]
where the convergence is in probability with respect to randomly chosen $\ud$.
\end{theorem}
\begin{proof}
    By taking the derivative of  $\mathsf{L}^{\text{EB}}(\varsigma,q,\ud)$ with respect to $\varsigma$ and setting it to $0$, we get the explicit formula:
    \begin{equation}
     \sigma^{\text{EB}}(q,\ud)=\sigma\sqrt{\frac{\|u(\cdot,s,q)\|_s^2}{2^{qd}}}\, .   
    \end{equation}
    Due to Proposition \ref{prop: representation of H^t norm of u(,t,q)}, we get our $\|u(\cdot,s,q)\|_s^2=\sum_{m \in B_q^d} \xi_m^2$. By the Law of Large Numbers, we have
    \[\lim_{q \to \infty}\frac{\|u(\cdot,s,q)\|_s^2}{2^{qd}} =1\, ,  \]
    from which the consistency follows.
\end{proof}
{\begin{remark}
We note that consistency results for the amplitude parameter have been well studied in the literature; see \cite{stein99book}. The purpose of this subsection is to tie those results to the rather explicit setting of our paper. One important feature of the torus model is that we are able to get an explicit and simple formula for $\sigma^{\text{EB}}$, so the consistency results are very clear. Moreover, since $\sigma^{\text{EB}}$ is the average of i.i.d. Gaussian random variables, one can also easily read off other statistical properties of this estimator (although the result of asymptotic distribution is also not completely new; see for example the discussion on page 201 in \cite{stein99book}).
\end{remark}}
\subsubsection{Recovery of $s,\sigma$ simultaneously}
We now build on the previous experiment to study whether the EB method can recover $s,\sigma$ simultaneously when $\tau$ is fixed. 
We reemphasize that since the measures
are mutually singular with respect to changes in $\sigma$ and $s$ we do expect to be able to recover
$(\sigma,s)$ by EB.
The basic set-up is the same as the last subsection, and now we minimize the EB loss function to recover $s,\sigma$ where, again,
$\sigma=\exp(\sigma')$. We run $50$ instances (each instance corresponds 
to a random draw of $\xi$), and collect the estimators $(s^{\text{EB}},\log \sigma^{\text{EB}})$ of the EB loss function for each instance. We present the histogram of the two values obtained in the experiments as follows (Figure \ref{fig: recover s sigma}).
\begin{figure}[ht]
    \centering
    \includegraphics[width=6cm]{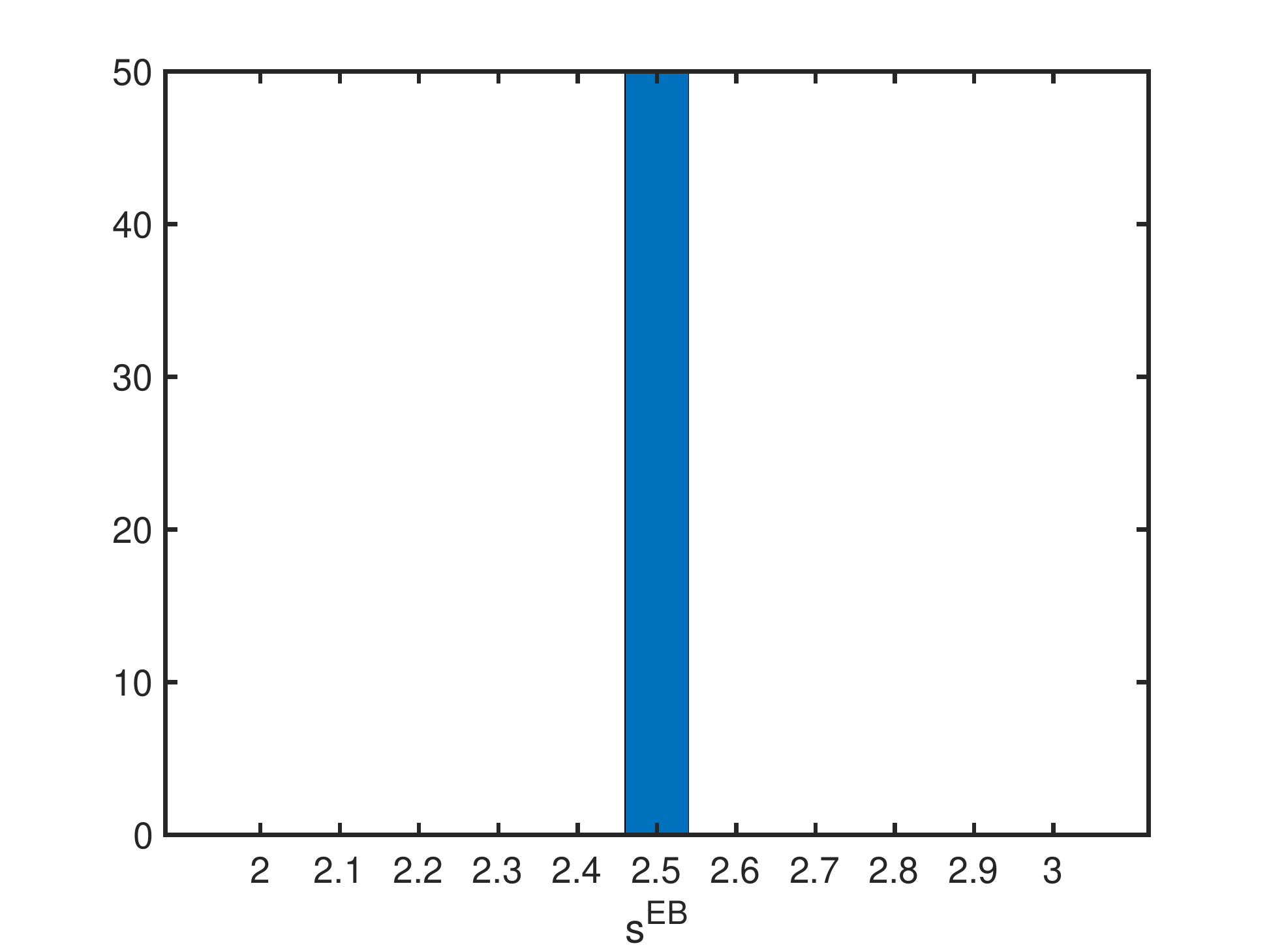}
    \includegraphics[width=6cm]{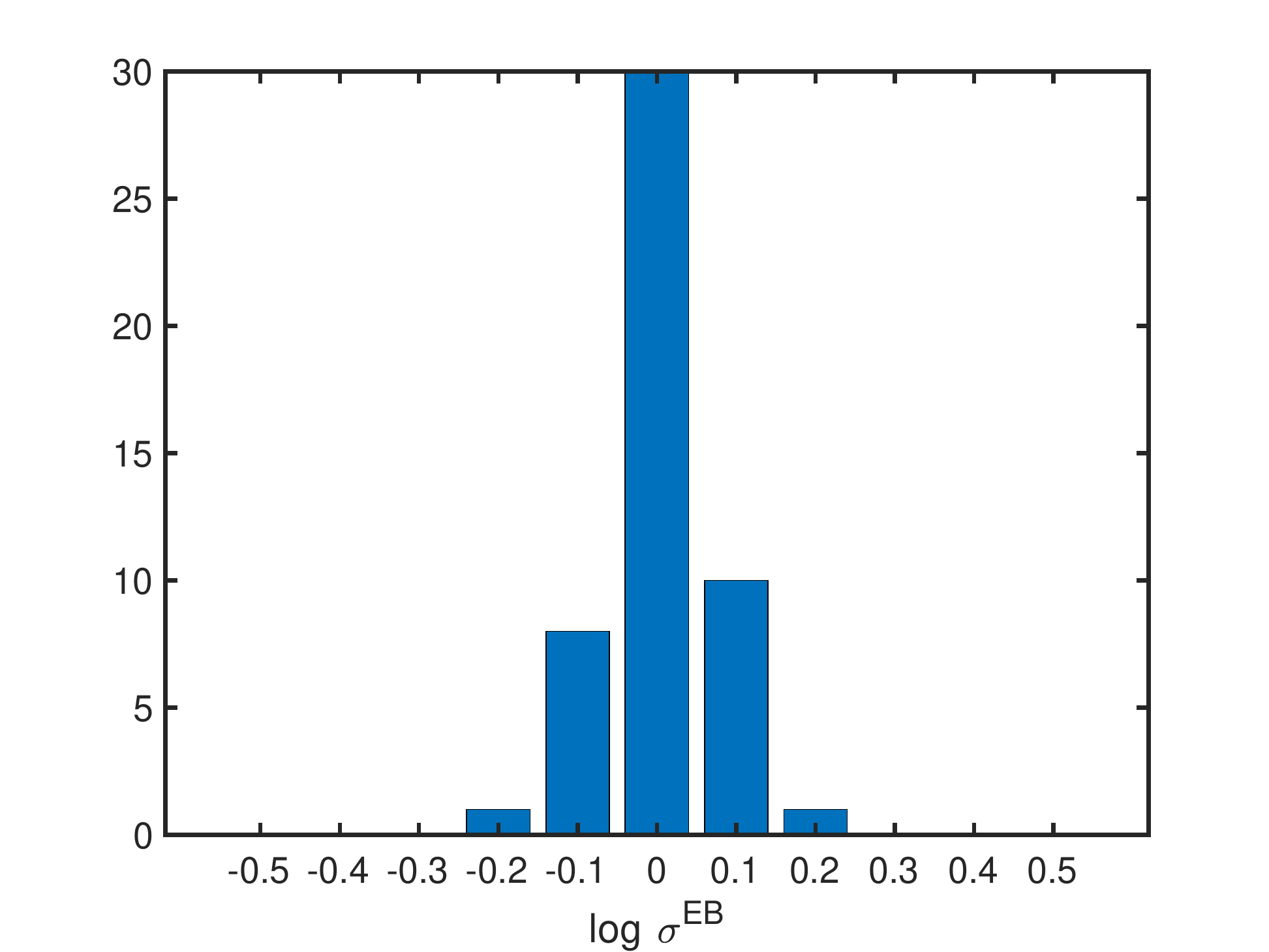}
    \caption{Left: histogram of the $s^{\text{EB}}$; right: histogram of the $\log \sigma^{\text{EB}}$}
    \label{fig: recover s sigma}
    \end{figure} 
    
From the figure, we observe that in the $50$ runs, the minimizer $(s^{\text{EB}},\sigma^{\text{EB}})$ is close to the ground truth $(2.5,1)$. We conclude that the EB method can recover the two parameters simultaneously in such a context.
\subsubsection{Recovery of $\tau$}
We consider whether EB and KF can recover the inverse lengthscale parameter $\tau$. We assume
that $\sigma$ is fixed at $1$, $s$ is chosen to be $2.5$, and sample $\ud \sim \cN(0,(-\Delta + \tau^2 I)^{-s})$ with $\tau=1$. As in the preceding experiments we consider the one dimensional torus example, and the same discretization precision and data acquisition setting as before. We draw $50$ instances of $\ud$, and for each of them, calculate the minimizers of the EB and KF loss function. We write $\tau=\exp(\tau')$ and  the estimator is $\log \tau^{\text{EB}}$ for $\tau'$, which we constrain to be in the
interval $[-2,2]$. In the EB loss function we fix $t=s$ within the loss function;
for the KF method, we select $t=s$ (case 1) and $t=\frac{s-d/2}{2}$ (case 2) respectively within
the loss function. The histogram of the minimizers of the resulting EB loss function and KF loss functions (in both cases) are presented in Figure \ref{fig: recover tau}, expressed in terms of
$\log \tau^{\text{EB}}$ and $\log \tau^{\text{KF}}$.
\begin{figure}[ht]
    \centering
    \includegraphics[width=6cm]{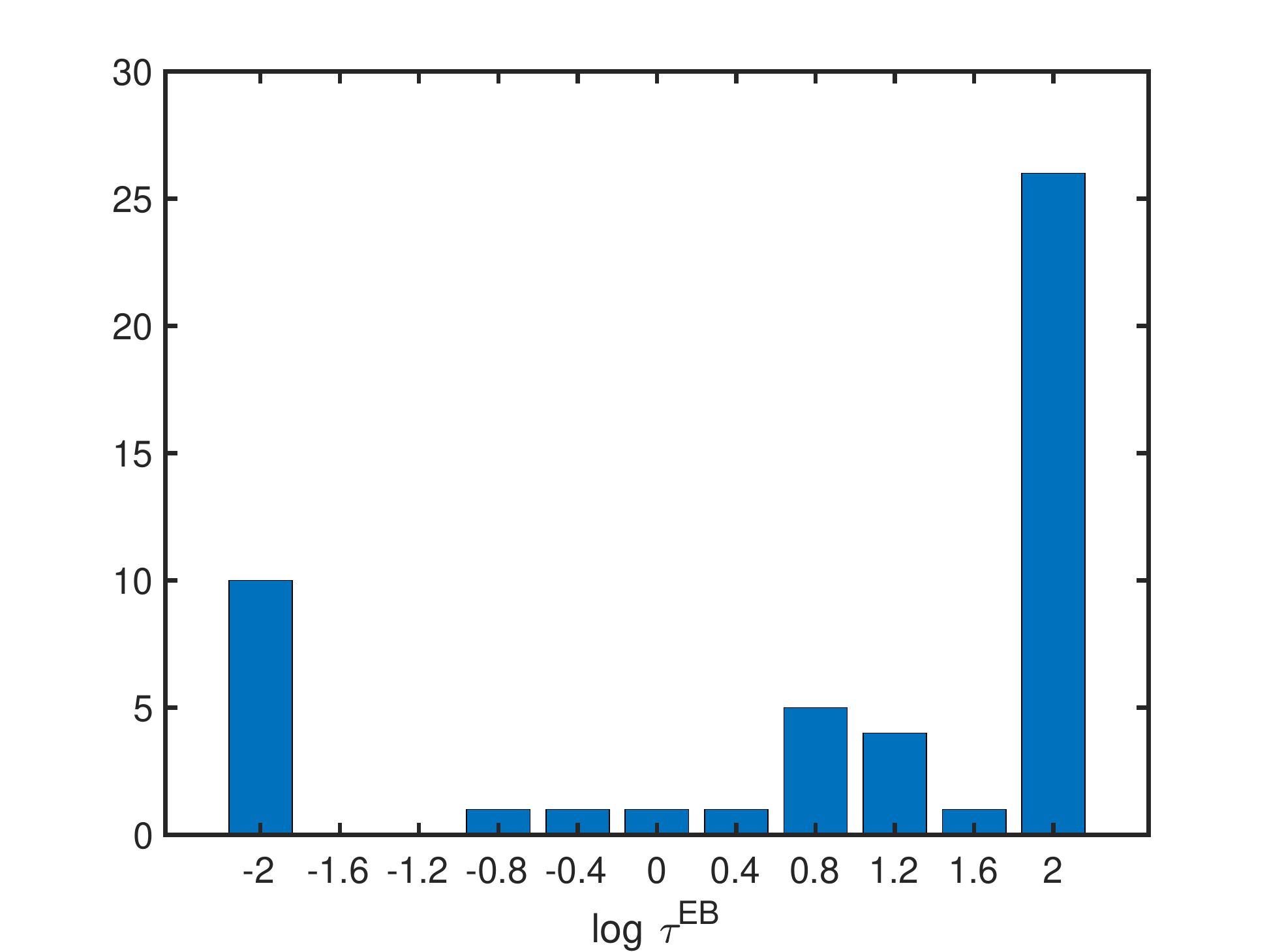}
    \includegraphics[width=6cm]{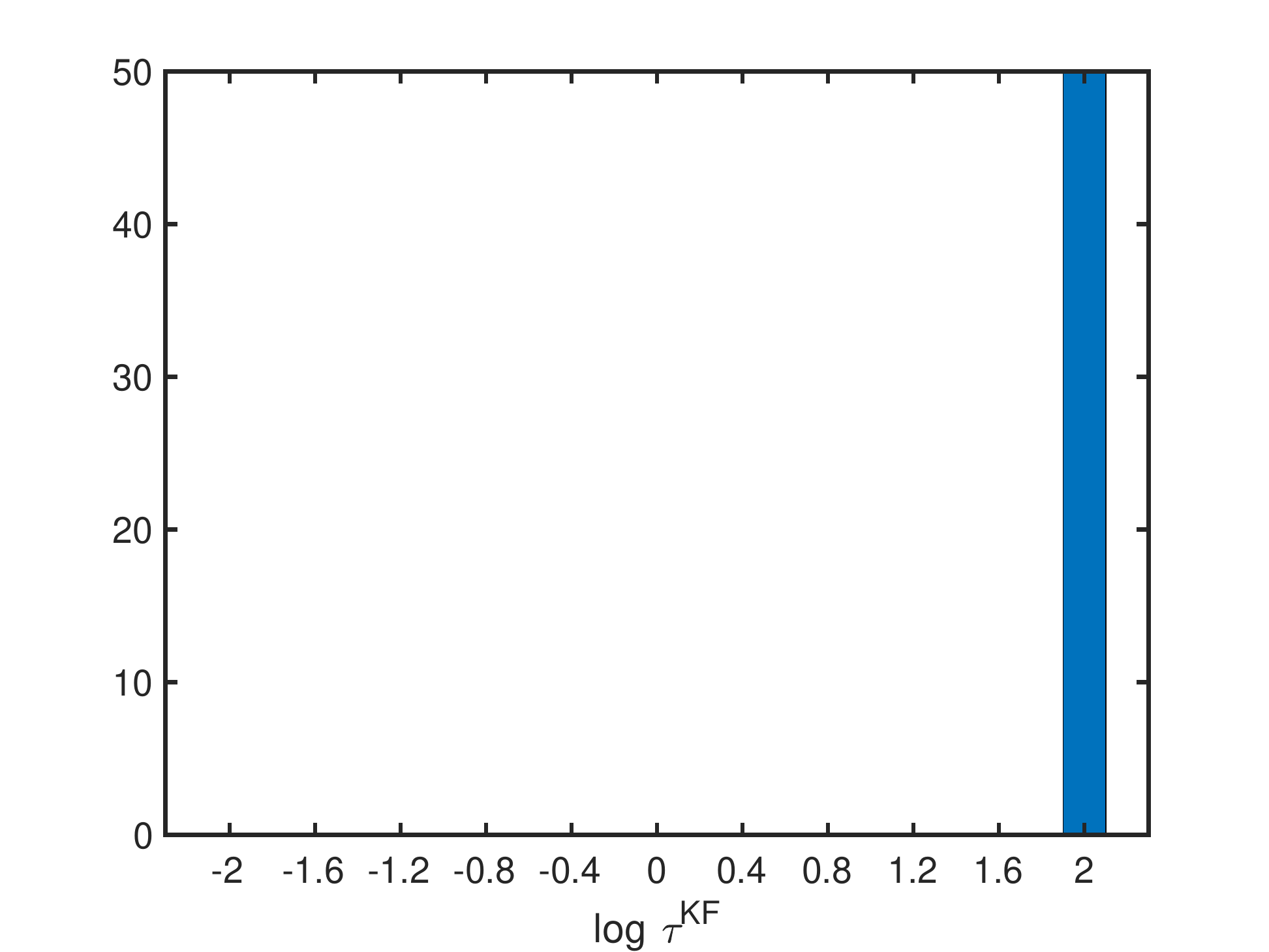}
    \includegraphics[width=6cm]{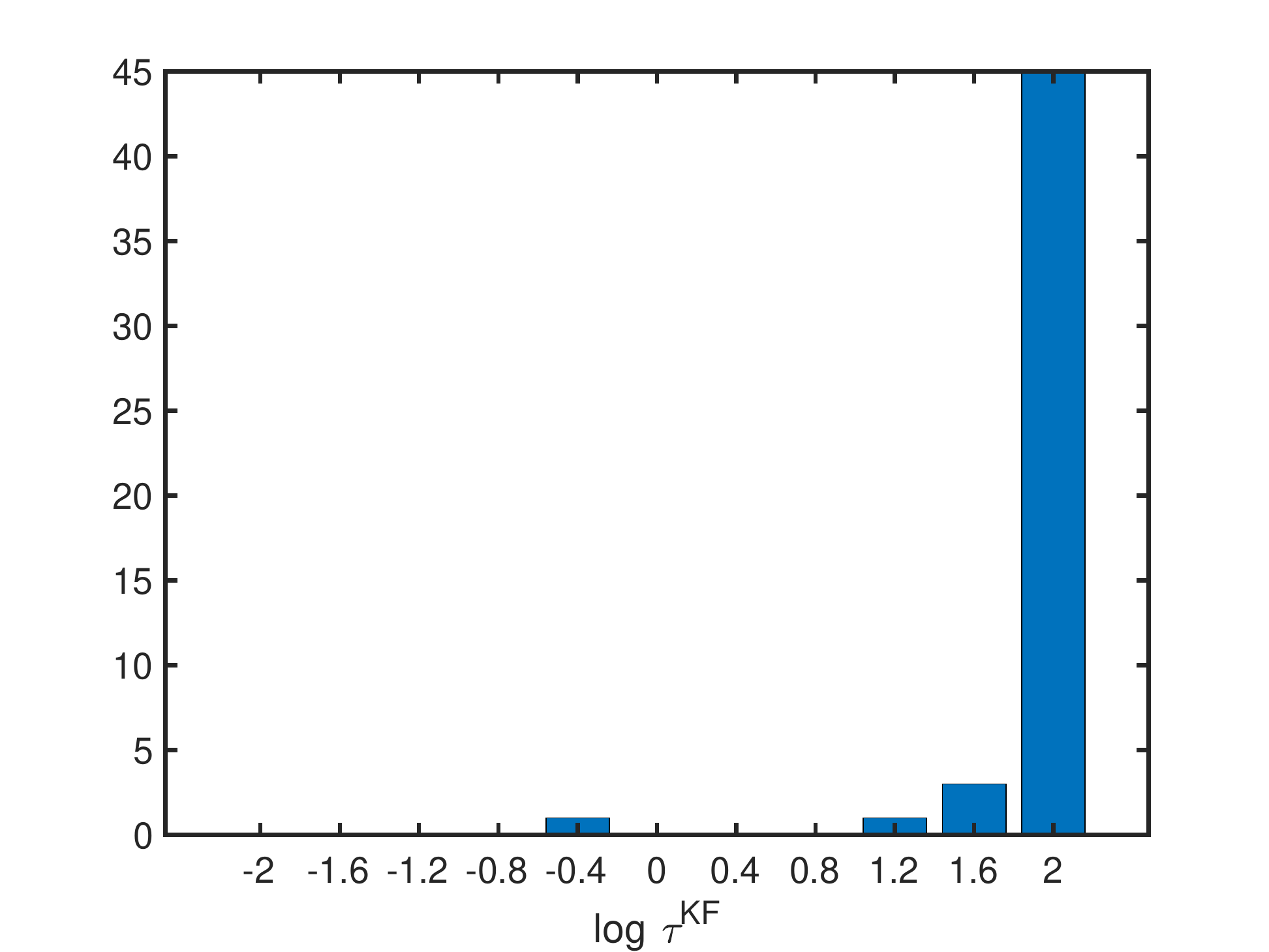}
    \caption{Histogram of the $\log \tau^{\text{EB}}$ or $\log \tau^{\text{KF}}$. Upper left: EB loss; upper right: KF loss (case 1); bottom: KF loss (case 2)}
    \label{fig: recover tau}
    \end{figure} 
In the 50 runs, the EB estimator takes many different values with no apparent pattern. For both case 1 and case 2, the KF estimator of $\tau'$ takes the value $2$ very often, which is the maximal value of the constrained decision variable. None of the estimators recover the true $\tau'=0$.  

The behavior of the KF estimator can be explained by the observation that when $\tau$ increases, the function drawn from the Gaussian prior becomes smoother, and hence the subsampling step in the KF loss does not sacrifice too much information. Therefore, the KF loss exhibits a tendency to get smaller as $\tau$ increases.
We can understand why EB cannot recover $\tau$ by studying the equivalence of Gaussian measures. As shown in \cite{dunlop2017hierarchical}, when dimension $d\leq 3$, the Gaussian measures $\cN(0,(-\Delta+\tau^2 I)^{-s})$ for different $\tau$ are equivalent; thus one cannot expect to recover $\tau$ using the information from one sample. 

We can also consider the problem of recovering $s,\tau$ simultaneously, i.e., we solve a joint minimization problem to get $s^{\text{EB}},\log \tau^{\text{EB}}$ and $s^{\text{KF}}, \log \tau^{\text{KF}}$. The set-up is the same as above, with the sample drawn from $\cN(0,(-\Delta+\tau^2I)^{-s})$ for $\tau=1$ and $s=2.5$. We form the EB and KF loss for $50$ instances of different draws and find the minimizers as corresponding estimators. The histograms of the estimators are shown in Figure \ref{fig: recover s tau EB} and \ref{fig: recover s tau KF}.
These figures show that in this joint optimization, the EB method picks the correct value
$s^{\text{EB}}=2.5$ for estimating $s$, and exhibit no patterns for $\log \tau^{\text{EB}}$; the KF method finds values close to $1$ for $s^{\text{KF}}$,
as it would in the absence of simultaneous estimation of $\tau'$, and selects the largest possible value in the constraint for $\log \tau^{\text{KF}}$, here being $2$. The conclusion is that the fact that  $\tau'$ cannot be learned
accurately does not influence the estimation of the regularity parameter $s$ in a context in which the two
are learned simultaneously. Indeed, this conclusion also holds when we are recovering the three parameters $(s,\sigma,\tau)$ simultaneously.

\begin{figure}[ht]
    \centering
    \includegraphics[width=6cm]{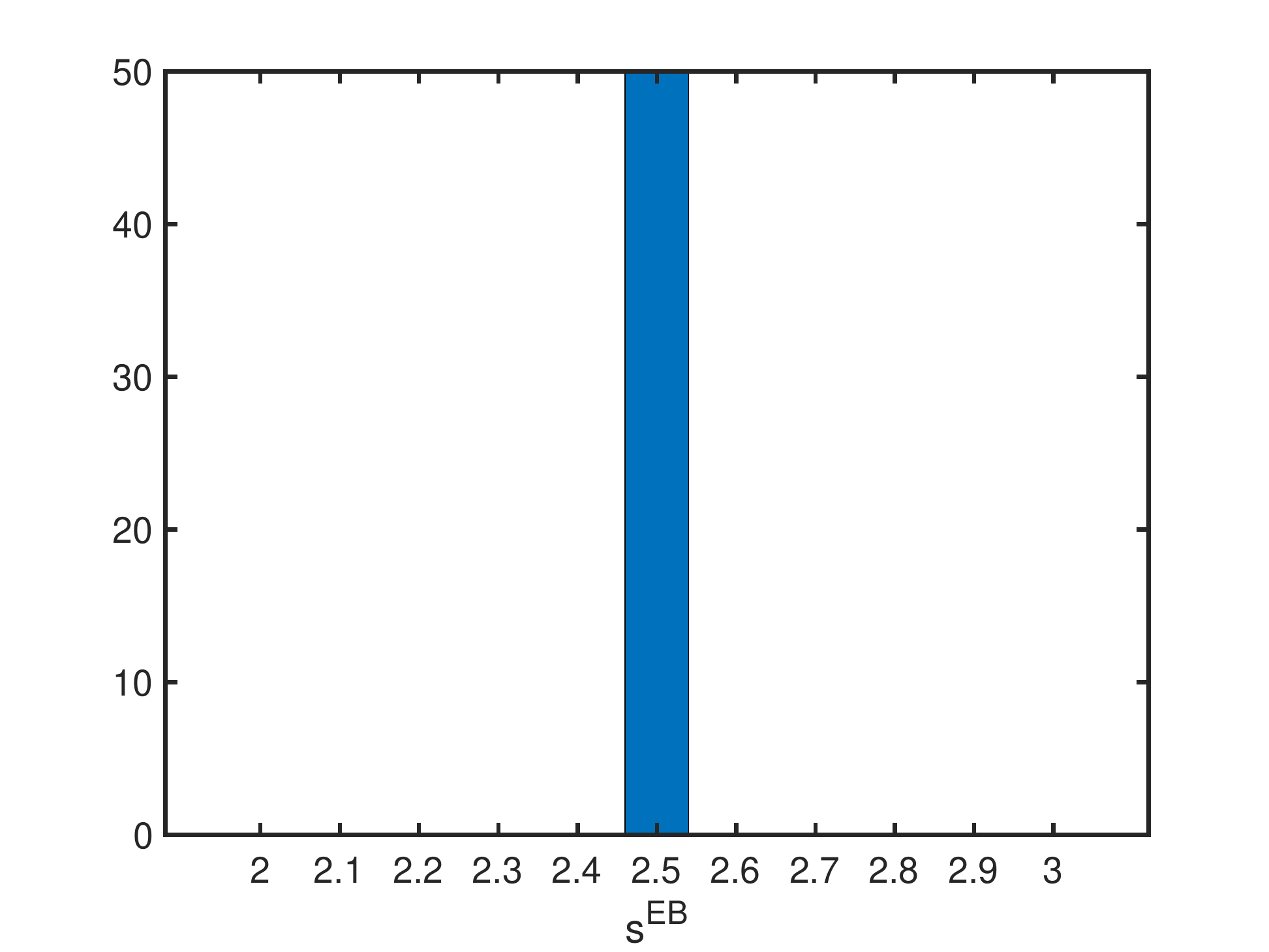}
    \includegraphics[width=6cm]{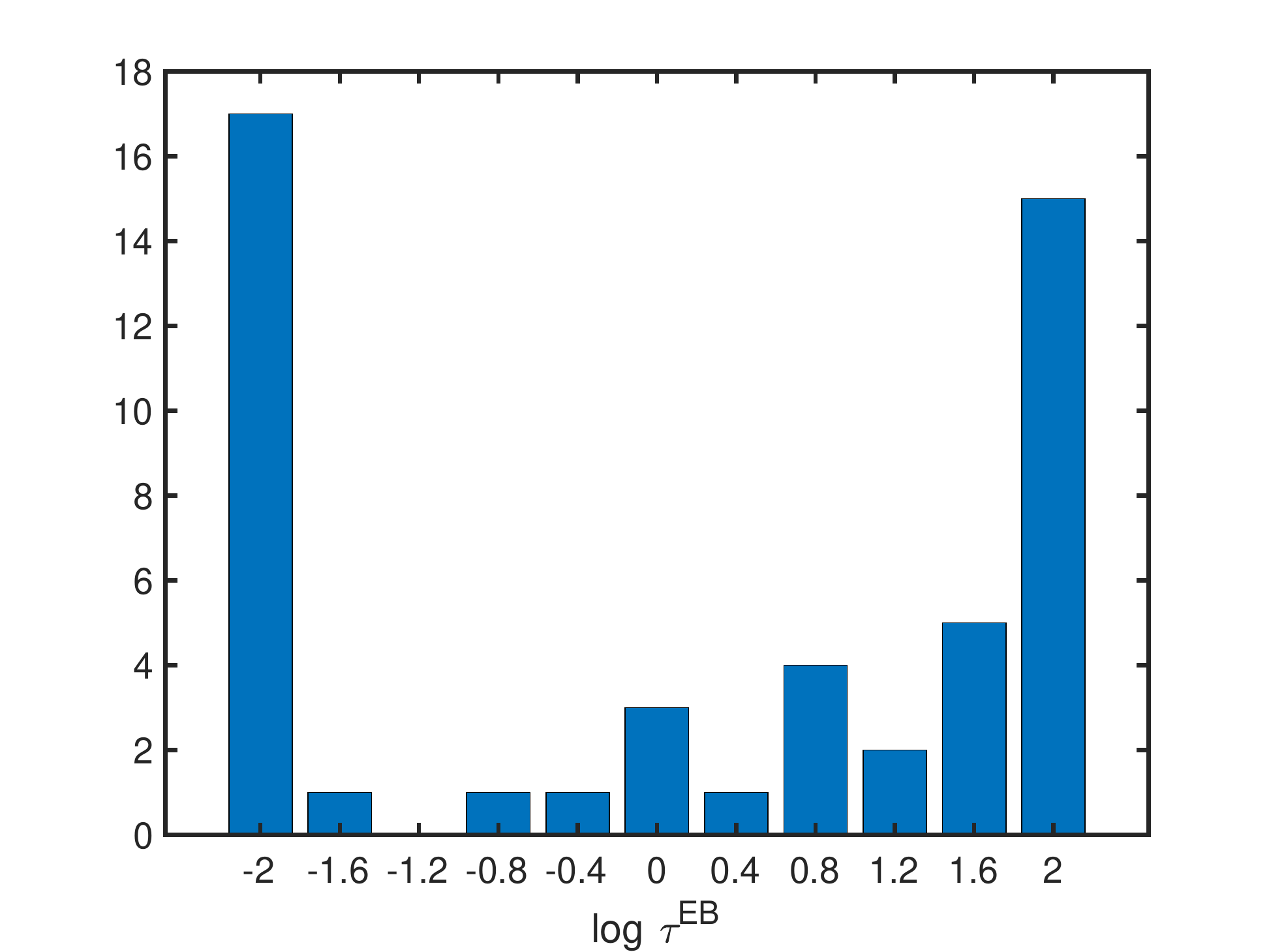}
    \caption{EB approach. Left: histogram of the $s^{\text{EB}}$; right: histogram of the $\log \tau^{\text{EB}}$}
    \label{fig: recover s tau EB}
    \end{figure} 

\begin{figure}[ht]
    \centering
    \includegraphics[width=6cm]{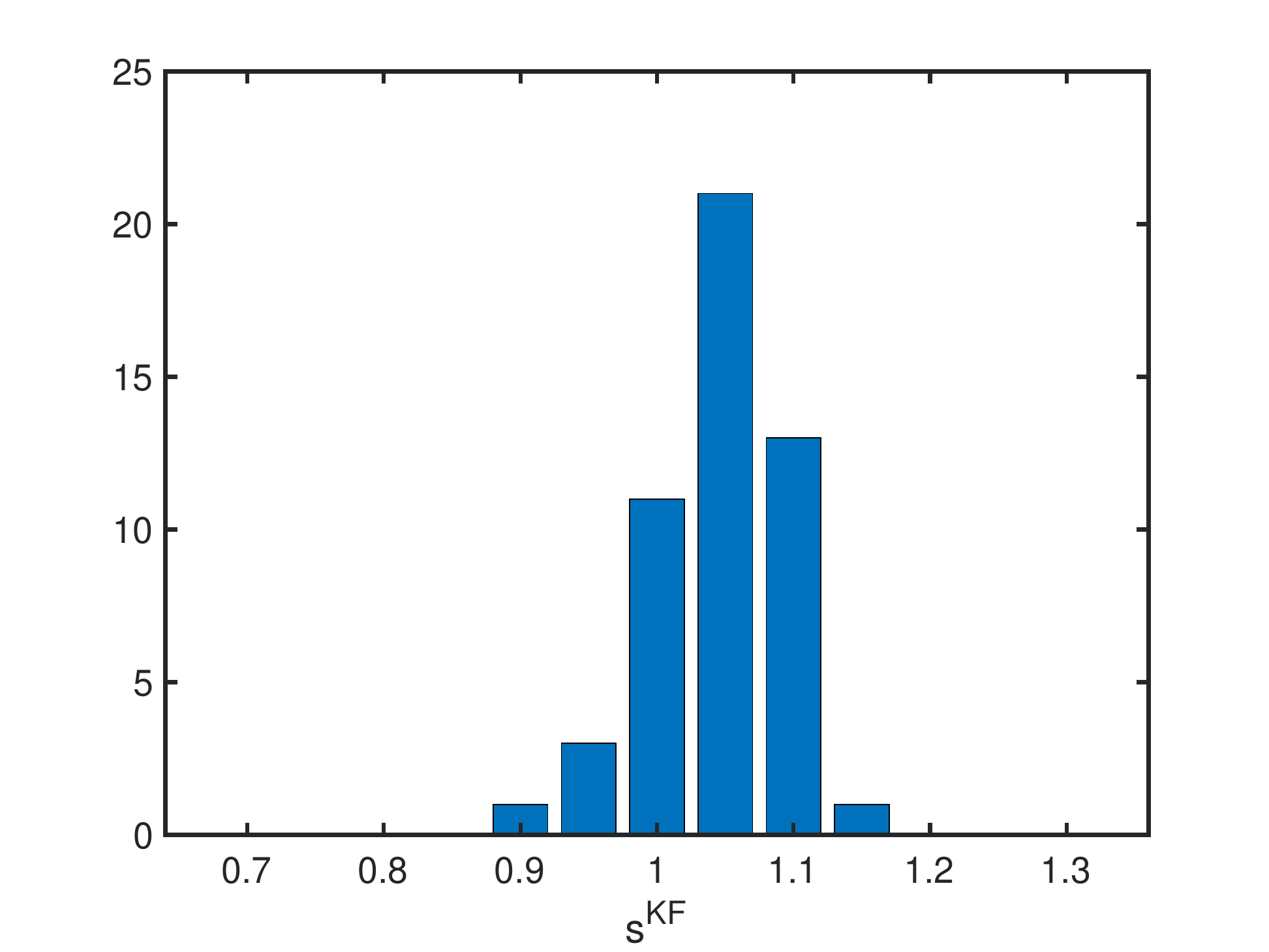}
    \includegraphics[width=6cm]{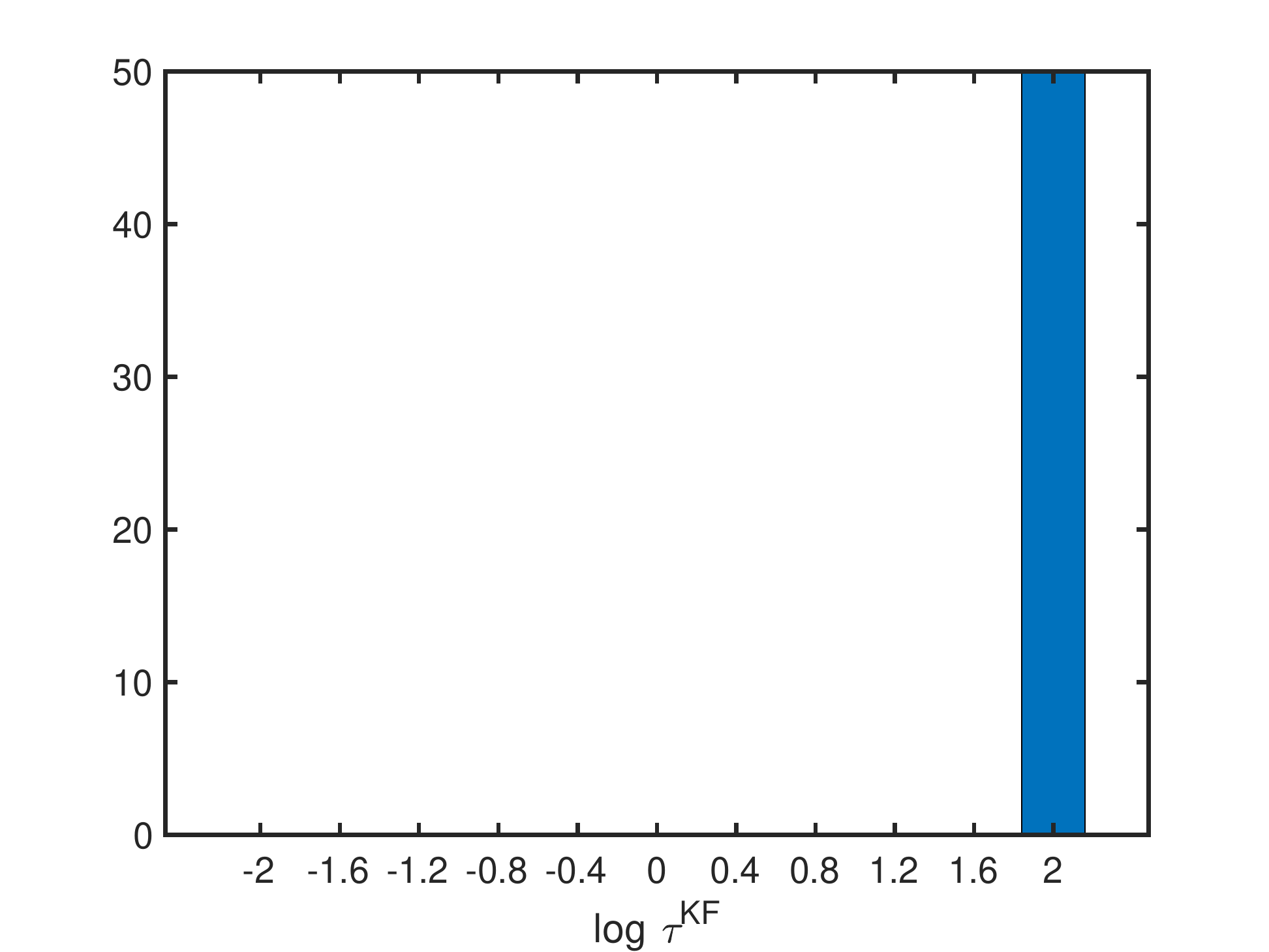}
    \caption{KF approach. Left: histogram of $s^{\text{KF}}$; right: histogram of the $\log \tau^{\text{KF}}$}
    \label{fig: recover s tau KF}
    \end{figure}

\subsection{Other Well-specified Examples}
\label{sec: Other Well-specified Examples}
In this subsection, we consider numerical examples for recovering parameters of a random field in the well-specified case, going beyond the Mat\'ern process studied thus far.
\subsubsection{Recovery of regularity parameter for variable coefficient elliptic operator} 
\label{subsubsec: Recovery of regularity parameter for variable coefficient elliptic operator}

Set $D=[0,1]$ so that $d=1$. The theoretical result in Section \ref{sec: recover regularity} assumes the function observed $\ud$ is drawn from $\cN(0,(-\Delta)^{-s})$ on a torus. In this subsection, we assume $\ud$ is drawn from $\cN(0,(-\nabla\cdot (a\nabla \cdot))^{-s})$ for some non-constant function $a$, and that the elliptic operator  implicit in
this defintion of a Gaussian measure is equipped with homogeneous 
Dirichlet boundary condition on  $D$.
We observe its values on the $2^9$ equidistributed points of the total $2^{10}$ grid points used for discretization. 

Here we select a  coefficient $a(x)$ that exhibits a discontinuity at $x=1/2$:
\begin{equation}
\label{eq:aref}
    a(x)=\begin{cases}
      1 & x \in [0,1/2]\\
      2 & x \in (1/2,1]\, .
    \end{cases} 
\end{equation}
As a consequence the induced operator is not the Laplacian. We pick $s=2.5$ to draw a sample $\ud$. 

In the well-specified case, the GP used in defining the EB and KF estimators is parameterized  by $\cN(0, (-\nabla\cdot (a\nabla \cdot))^{-t})$ and
we aim to learn parameter $t$ given a data calculated using a draw
from the same measure with $t=s$. We consider the well-specified case here (the misspecified case will be considered in Subsection \ref{sec: misspecification, recover regularity}.) We output the histogram of the EB and KF estimators for $50$ different draws of $\ud$ in Figure \ref{fig: Histogram of estimator of s, variable coeff, for different draws}.
\begin{figure}[ht]
    \centering
    \includegraphics[width=6cm]{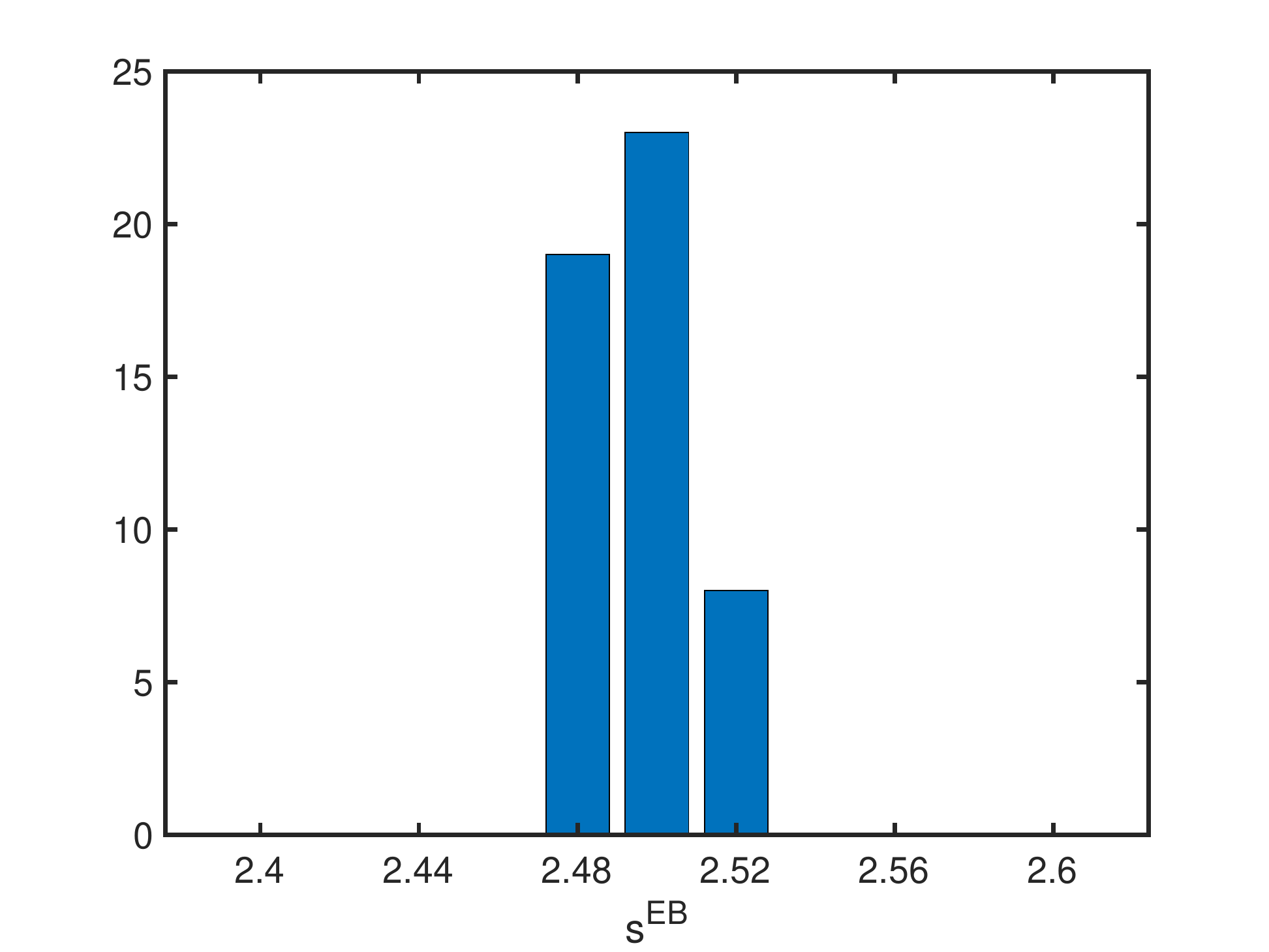}
    \includegraphics[width=6cm]{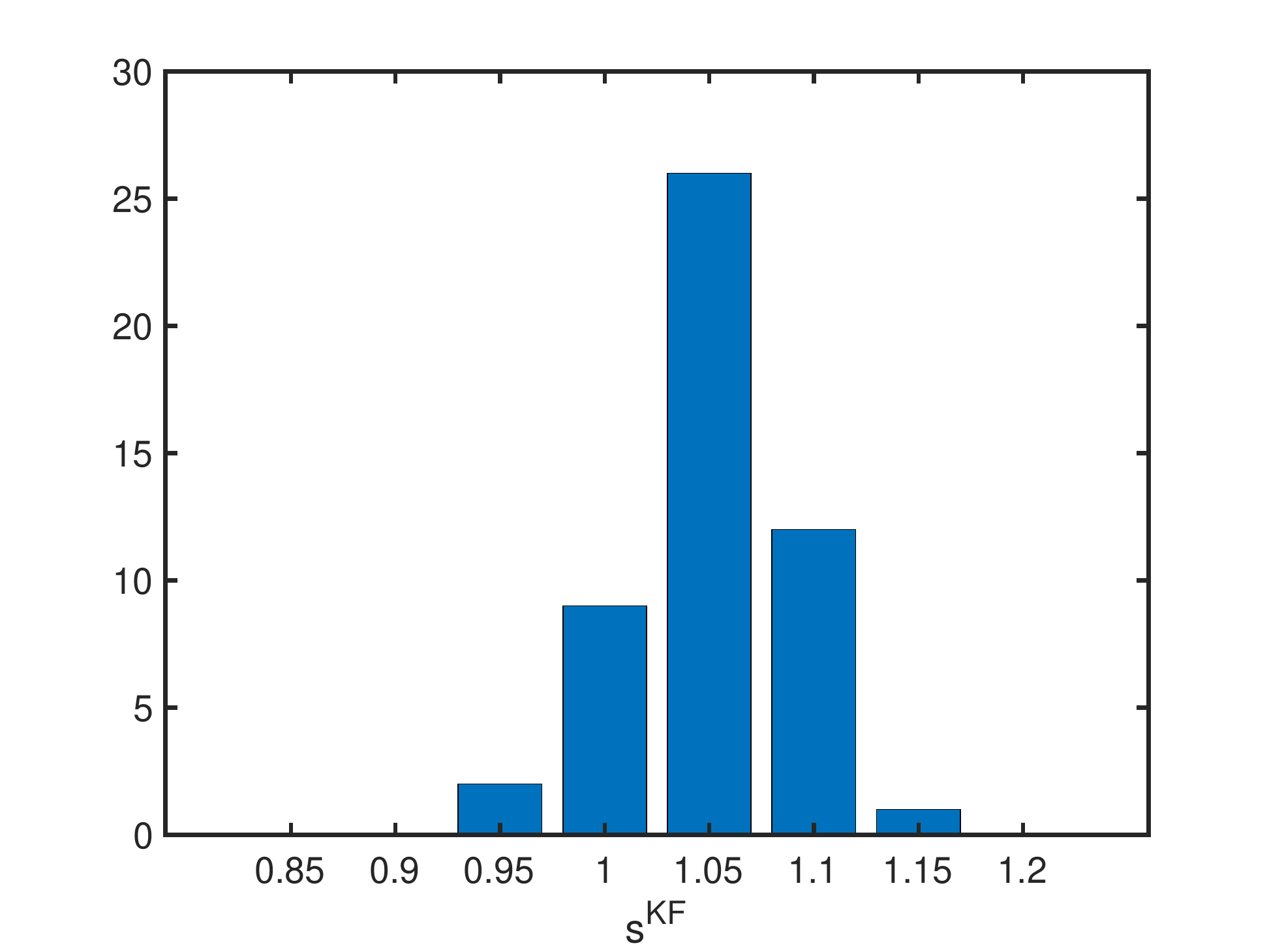}
    \caption{Histogram of the regularity estimators for the variable coefficient covariance case. Left: EB; right: KF}
    \label{fig: Histogram of estimator of s, variable coeff, for different draws}
    \end{figure} 
The experiments show that for the variable coefficient elliptic operator model, EB and KF succeed in converging to the correct limits. We can calculate the (normalized) variance of the two estimators based on the histograms:
\[\frac{\text{Var}(s^{\text{EB}})}{s^2} \approx 7.8\times 10^{-5} \quad \text{and}\quad \frac{\text{Var}(s^{\text{KF}})}{\left((s-d/2)/2\right)^2}\approx 4 \times 10^{-3}\, . \]
The relative magnitude is similar to the one in Subsection \ref{sec: experiments, variance of estimators}.
\subsubsection{Recovery of discontinuity position for conductivity field} 
\label{sec: Recovery of discontinuity position for conductivity field}
Define  the conductivity field
$a_{\theta}:[0,1] \mapsto \mathbb{R}$, and parameterized by  $\theta \in [0,1]$, via
\begin{equation}
    a_{\theta}(x)=\begin{cases}
      1 & x \in [0,\theta]\\
      2 & x \in (\theta,1]\, .
    \end{cases} 
\end{equation}
In this subsection, we assume that our data $\ud$ is obtained by solving the SPDE 
\[-\nabla\cdot (a_{1/2}\nabla  \ud)=\xi\, ,\]
subject to a homogeneous Dirichlet boundary condition on $[0,1].$ We choose
$\xi$ as a random draw from $\cN(0, (-\Delta)^{-1})$. We can view $\ud$ is a sample drawn from $\cN(0,C_a)$ where
\begin{equation}
    \label{eq:cas1}
    C_a=(-\nabla\cdot (a_{1/2}\nabla  \cdot))^{-1}(-\Delta)^{-1}(-\nabla\cdot (a_{1/2}\nabla  \cdot))^{-1}.
\end{equation}

We observe the value of $\ud$ on the $2^9$ equidistributed points of the total $2^{10}$ grid points used for discretization. We use EB and KF to estimate $\theta$ from the partial observation of the function  $\ud$ based on the GP model $\cN(0,C_{a,s})$ where
\begin{equation}
    \label{eq:cas}
C_{a,s}=(-\nabla\cdot (a_{\theta}\nabla  \cdot))^{-1}(-\Delta)^{-s}(-\nabla\cdot (a_{\theta}\nabla  \cdot))^{-1}.
\end{equation}
The model is well-specified for $s=1$ and misspecified for $s\not=1$. 
Here consider the well-specified case in this subsection, i.e., $s=1$,
and $C_{a,s}=C_a$;  the misspecified case is covered in Subsection \ref{sec: Stochastic model misspecification for recovering discontinuity}.

We let the domain for $\theta$ be $[0.3,0.7]$ in the definition of EB and KF estimators. We compute the estimators for $50$ different draws of $\ud$. The histograms of the EB and KF estimators are shown in Figure \ref{fig: Histogram of the discontinuity position estimators (well-specified). }.  The loss functions for one random instance are shown in Figure \ref{fig: Loss function for recovering the discontinuity (well-specified)}. 
\begin{figure}[ht]
    \centering
    \includegraphics[width=6cm]{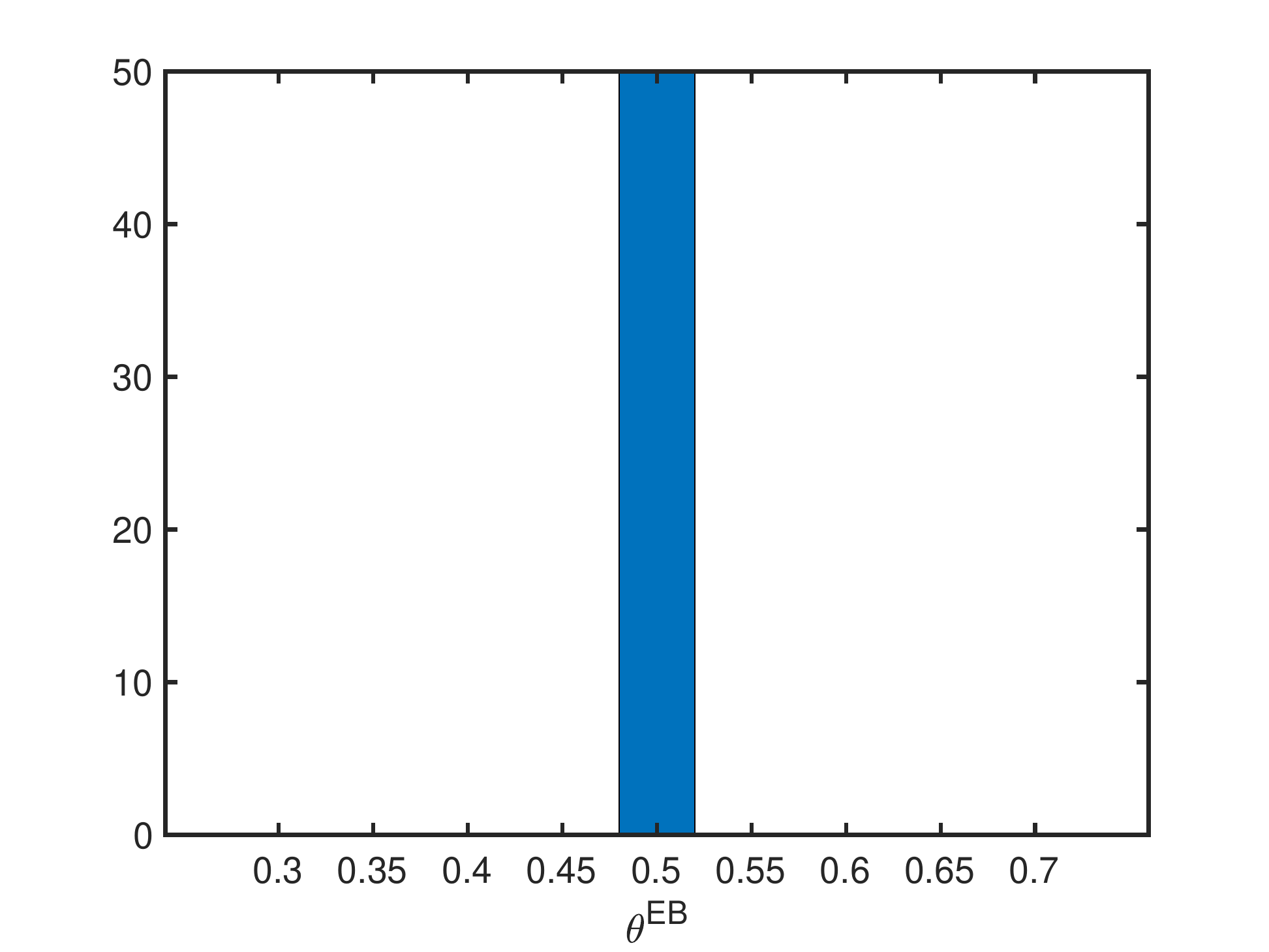}
    \includegraphics[width=6cm]{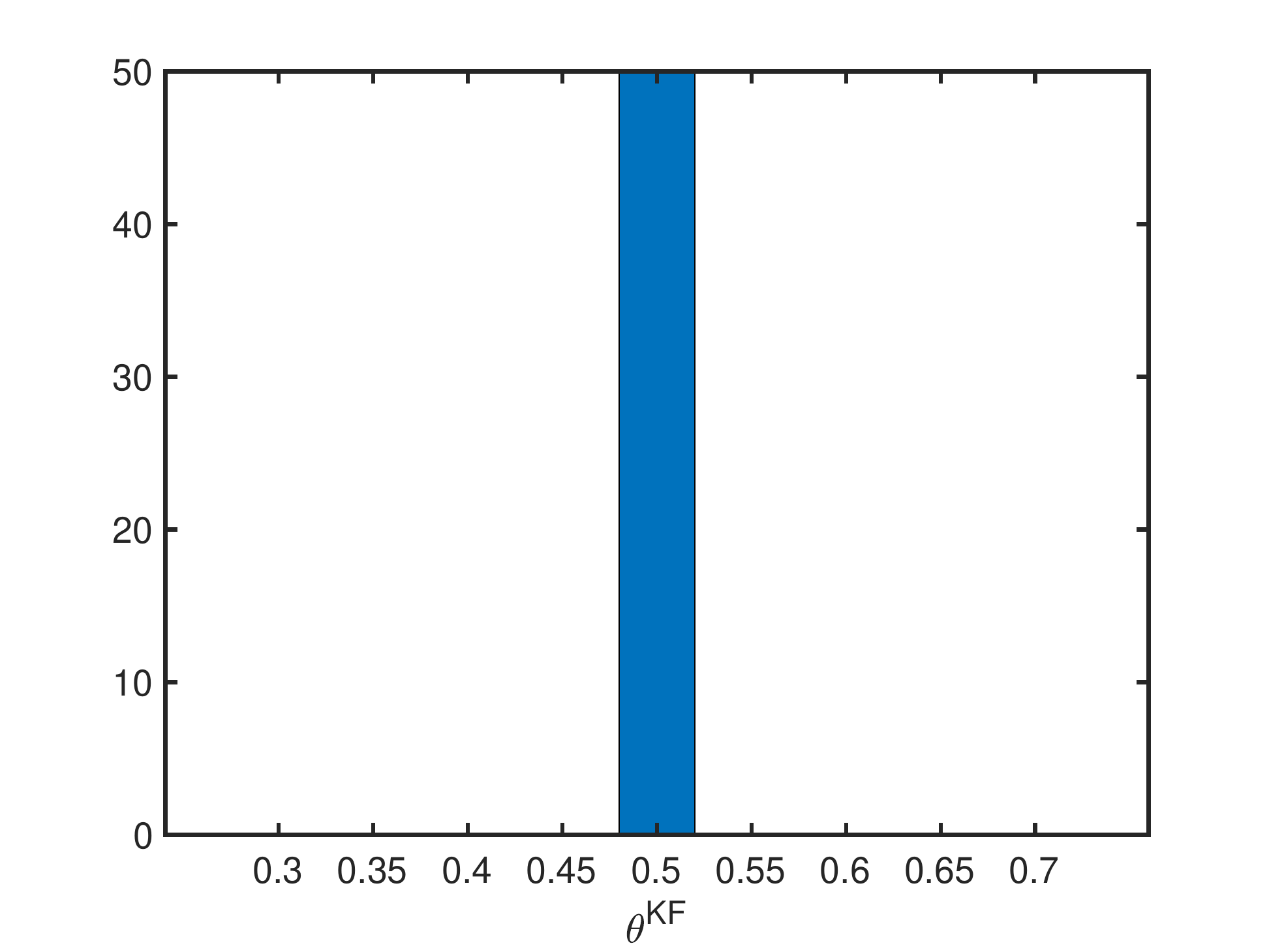}
    \caption{Histogram of the discontinuity position estimators (well-specified). Left: EB; right: KF}
    \label{fig: Histogram of the discontinuity position estimators (well-specified). }
    \end{figure} 
    
\begin{figure}[ht]
    \centering
    \includegraphics[width=6cm]{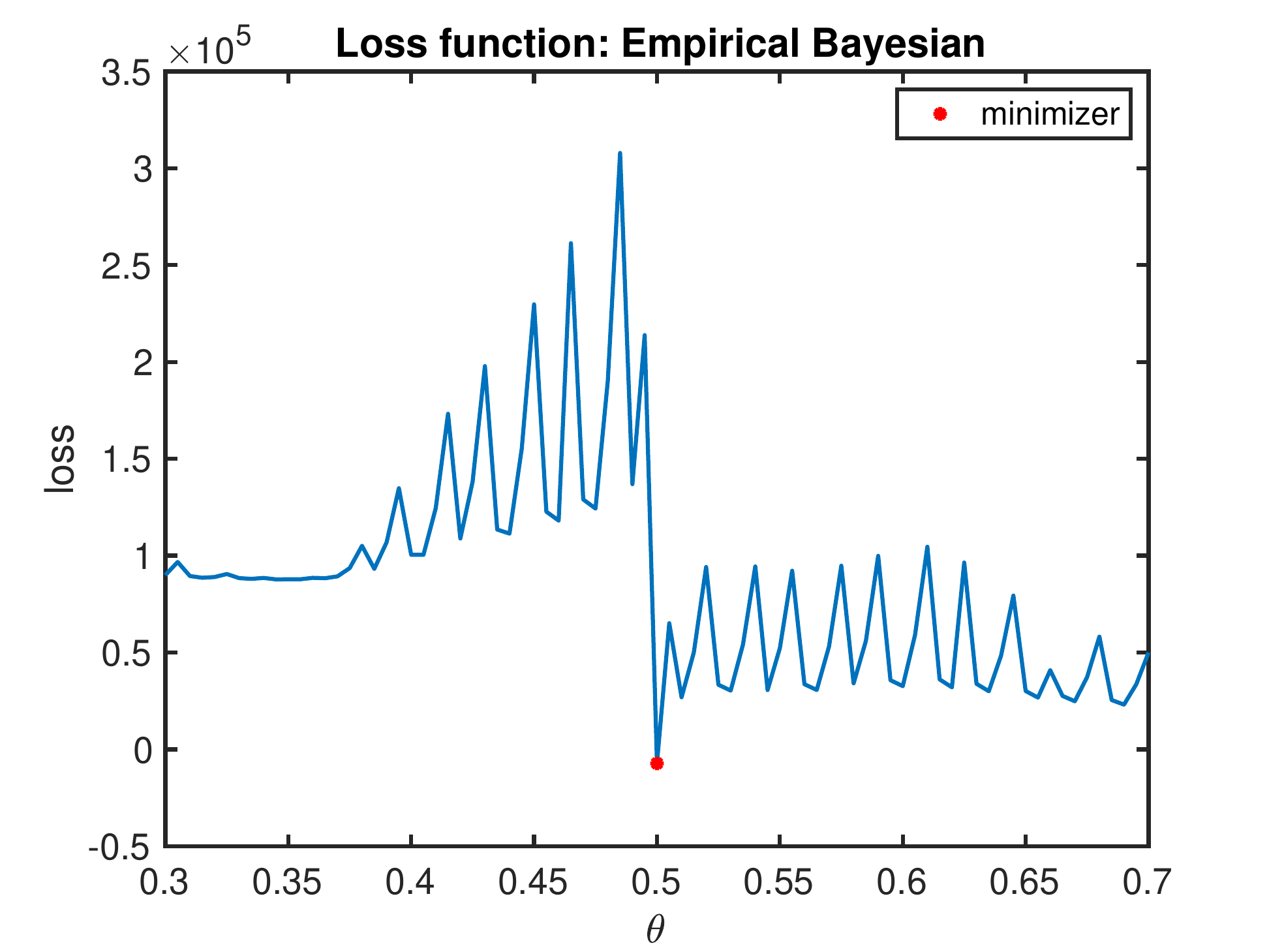}
    \includegraphics[width=6cm]{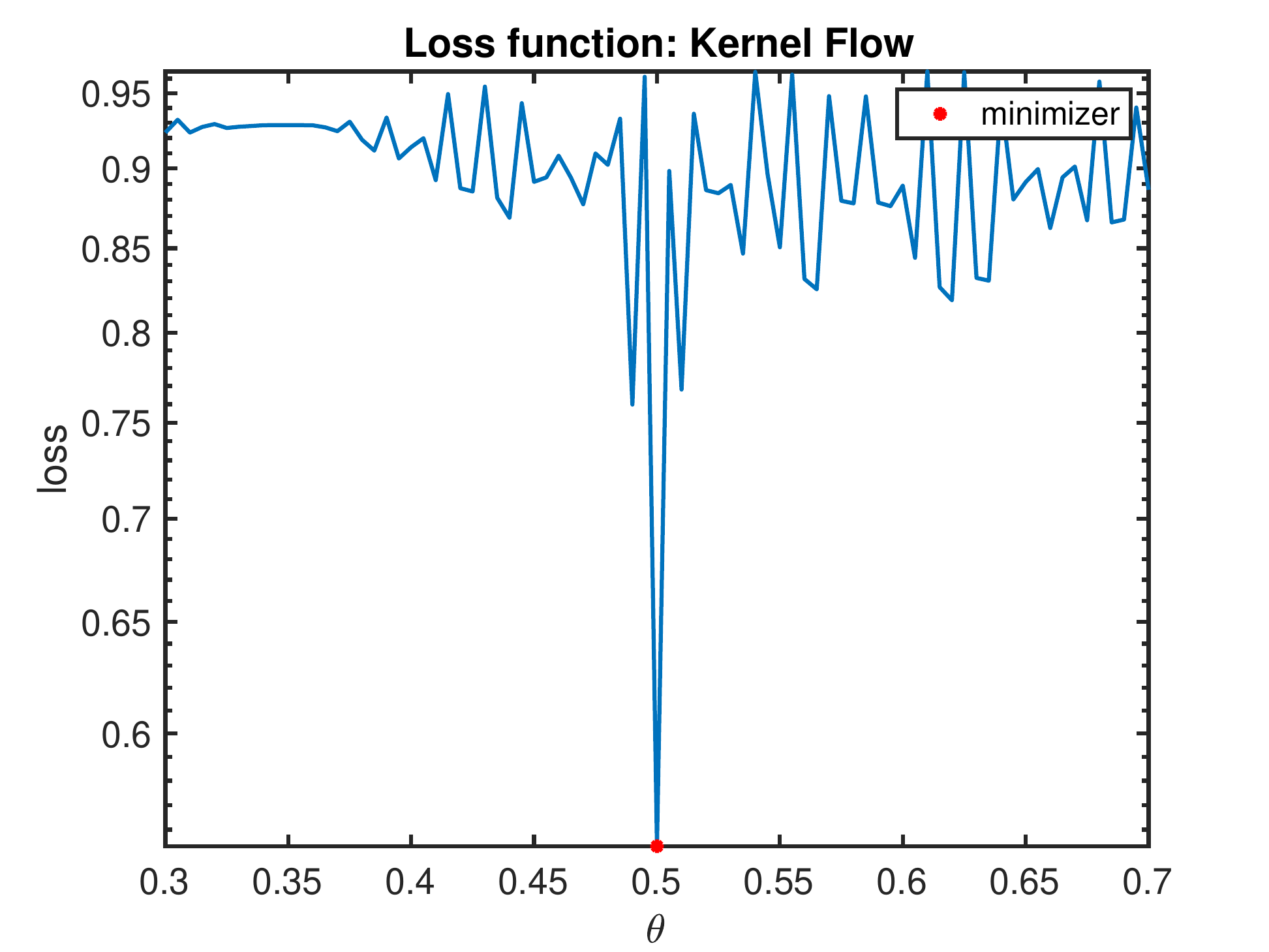}
    \caption{Loss function for recovering the discontinuity (well-specified). Left: EB; right: KF}
    \label{fig: Loss function for recovering the discontinuity (well-specified)}
    \end{figure} 

Our experiments show that both EB and KF can recover $\theta=1/2$, and the recovery is very stable with respect to different draws of $\ud$ from the SPDE. We conclude that the EB and KF can go beyond the Mat\'ern-like kernel model in practice; recovering the point of discontinuity of the conductivity field is an example of this fact.
\subsection{Computational Aspects}
\label{sec: experiments, computational aspects}
In this subsection, we add some discussions about the computational aspects. {
We start by remarking on how to compute the kernel function and sample the GP realization generally. Every kernel operator we consider involves certain differential operators. We discretize these differential operators and perform an eigenfunction decomposition of the obtained matrix. Then we use these eigenfunctions and eigenvalues to compute approximation of the kernel matrix, and draw samples from the GP with the covariance matrix being the kernel matrix; see also discussions above Remark \ref{connect to whittle matern}. This is similar to the spectral expansion of a kernel function and the Mercer decomposition of a GP.}

{Practical applications of hierarchical GPR} require weighting statistical efficiency against computational complexity.
Although the regularity models covered in this paper appear to produce well-behaved
EB and KF loss functions with easily identifiable global minimizers, models with high dimensional parameter space  typically require using algorithms such as gradient descent which do not
come with  theoretical guarantees on the identification of global minimizers. Furthermore,  when the size of the data is large, computation becomes a limiting factor, and subsampling offers a traditional remedy when combined with gradient descent, but again theoretical guarantees
are not typically to be expected. The stochastic algorithm 
 presented in \cite{owhadi2019kernel} for KF can be interpreted as an SGD algorithm aimed at minimizing the average loss 
\[\bE_{\pi_1}\bE_{\pi_2} \mathsf{L}^{\mathrm{KF}}(\theta,\pi_1\cX,\pi_2\pi_1\cX,\ud)\, ,\]
via draws from the distribution of $\pi_1$ and $\pi_2$ ($\pi_1\cX$ is a random subsampling of $\cX$, and $\pi_2\pi_1\cX$ is a further random subsampling of $\pi_1\cX$). The efficacy of an analogous strategy for EB remains unclear due to the presence of the
log determinant term in the loss. It is of future interest to explore further the computational aspects of the EB and KF approaches to hierarchical learning.
\section{Model Misspecification}
\label{sec: experiments, robustness}
All our preceding experiments are focused on the well-specified case: the function $\ud$ is drawn from the GP model assumed in the
estimation, or equivalently, the model for $\ud$ and for the kernel family $K_{\theta}$ in defining the loss functions are matched. This subsection studies model misspecification. We consider two possible ways to misspecify the model: (1) the function $\ud$ is drawn from a GP which is different from
that used in defining the loss function; (2) the function $\ud$ is a fixed deterministic function. The second case
may arise, for example, if the function comes from a solution of a PDE with some physical data, and there is no natural stochastic context for its
provenance.  The aim of this subsection is to study the behavior of the 
EB and KF estimators to compare their robustness to model misspecification.

\subsection{Stochastic model misspecification for recovering regularity}
\label{sec: misspecification, recover regularity}
In this subsection, we assume $\ud$ is drawn from $\cN(0,(-\nabla\cdot (a\nabla \cdot))^{-s})$, while the GP used in defining the EB and KF estimators is still $\cN(0,(-\Delta)^{-t})$. This results in a model misspecification corresponding to the well-specified model in Subsection \ref{subsubsec: Recovery of regularity parameter for variable coefficient elliptic operator}. As in Subsection \ref{subsubsec: Recovery of regularity parameter for variable coefficient elliptic operator}, we select $a$ as in \eqref{eq:aref} 
and we set $s=2.5$ to draw the sample $\ud$.
  Figure \ref{fig: Histogram of estimator of s, mis-specification, for different draws} shows the histograms of the minimizers of the EB and KF loss functions obtained from  $50$ independent draws from the Gaussian Process.
Despite misspecification, the EB and KF estimators are still concentrated around $2.5$ and $1$, respectively.
We also calculate the variance:
\[\frac{\text{Var}(s^{\text{EB}})}{s^2} \approx 5.9\times 10^{-4} \quad \text{and}\quad \frac{\text{Var}(s^{\text{KF}})}{\left((s-d/2)/2\right)^2}\approx 6.8 \times 10^{-4}\, . \]
In this example, the (normalized) variance of KF of EB are of similar magnitude. This is different from the well-specified case in Subsection \ref{subsubsec: Recovery of regularity parameter for variable coefficient elliptic operator} where the variance of EB is much smaller than KF.

\begin{figure}[ht]
    \centering
    \includegraphics[width=6cm]{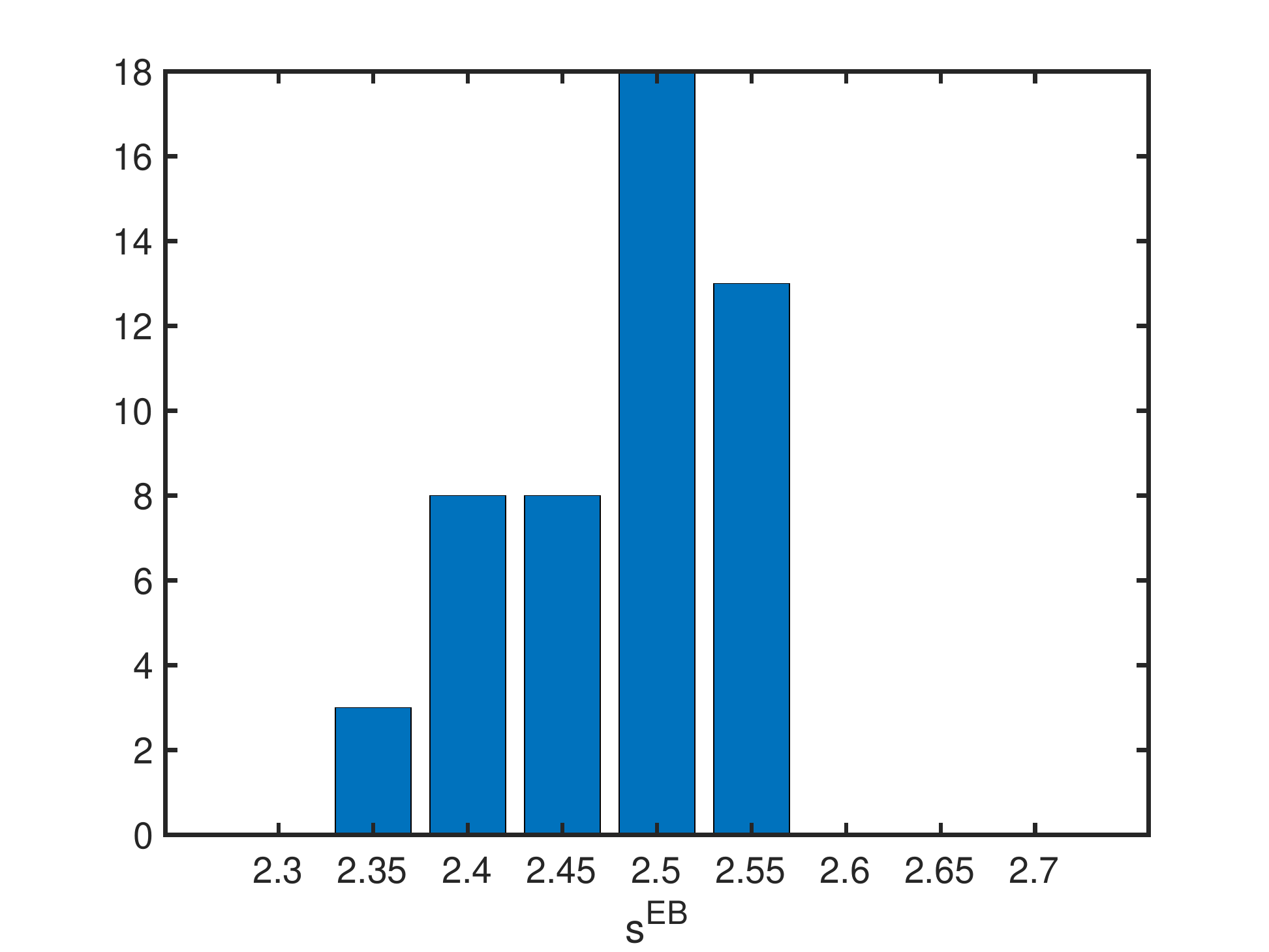}
    \includegraphics[width=6cm]{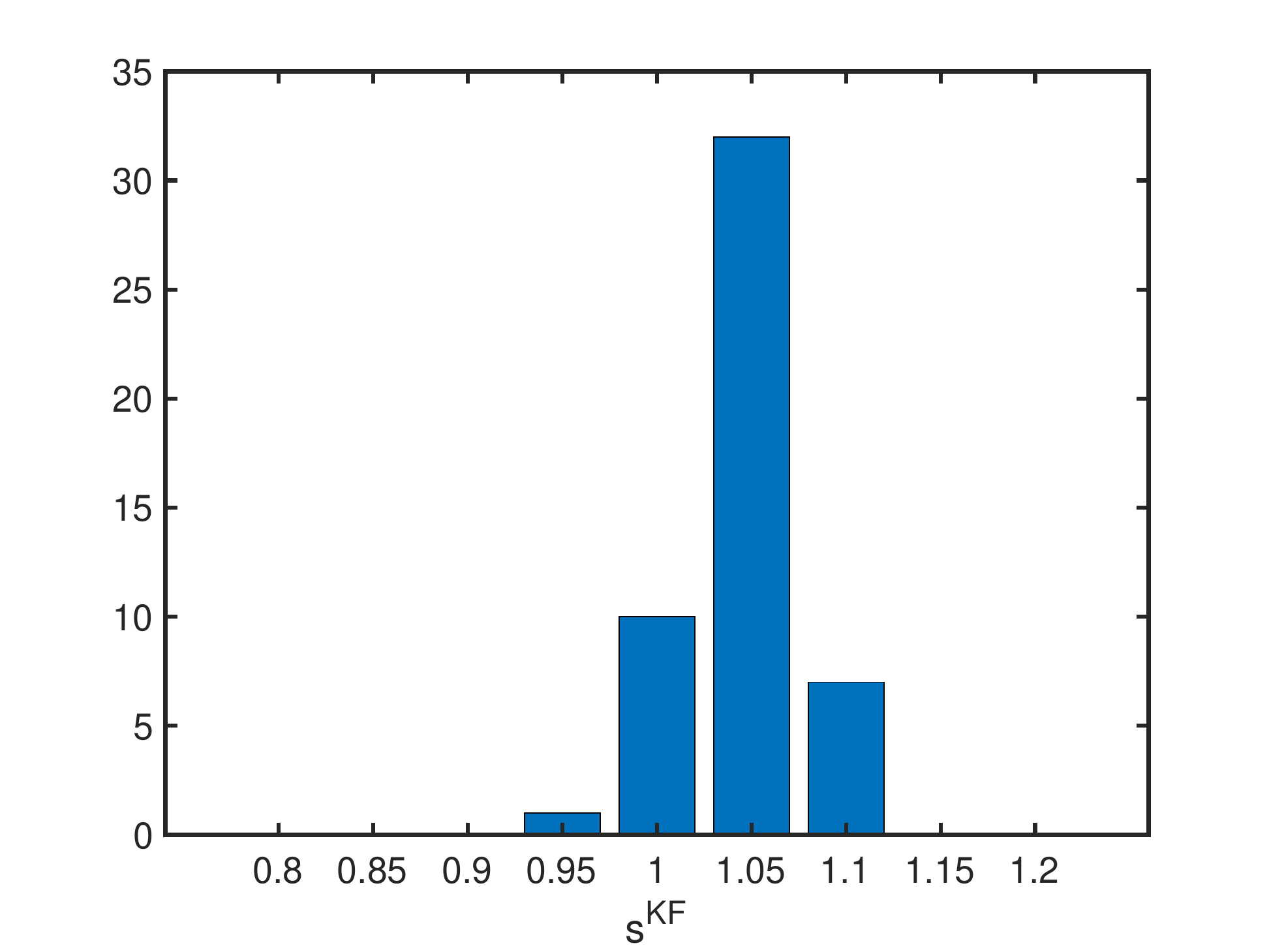}
    \caption{Histogram of the regularity estimators under model misspecification. Left: EB; right: KF}
    \label{fig: Histogram of estimator of s, mis-specification, for different draws}
    \end{figure}

\subsection{Stochastic model misspecification for recovering discontinuity} 
\label{sec: Stochastic model misspecification for recovering discontinuity}
In this subsection, we consider the model misspecifications that correspond to the well-specified case in Subsection \ref{sec: Recovery of discontinuity position for conductivity field}.
For the GP defining the EB and KF estimators we use the centred
Gaussian with covariance operator given by \eqref{eq:cas} with 
$s=5$; meanwhile  $\ud$ is drawn  from the centred
Gaussian with covariance operator given by \eqref{eq:cas1}; thus
we are in a misspecified version of the setting arising in Subsection \ref{sec: Recovery of discontinuity position for conductivity field} and,
as there, our aim is to recover the point of discontinuity. We illustrate the loss functions for a single draw of $\ud$ in Figure \ref{fig: loss function of estimator of discontinuity, mis-specification}. These plots are not sensitive to the particular draw of $\ud$ and illustrate the robustness of KF (and the lack of robustness of EB) to this misspecification. Indeed, the EB estimator gives $0.3$ which is the lower boundary of the compact parameter space used
in the minmization, while the KF estimator picks the true parameter $0.5$. The loss function of KF, shown in Figure \ref{fig: loss function of estimator of discontinuity, mis-specification}, exhibits a sharp global minimizer at $\theta=0.5$. 
\begin{figure}[ht]
    \centering
    \includegraphics[width=6cm]{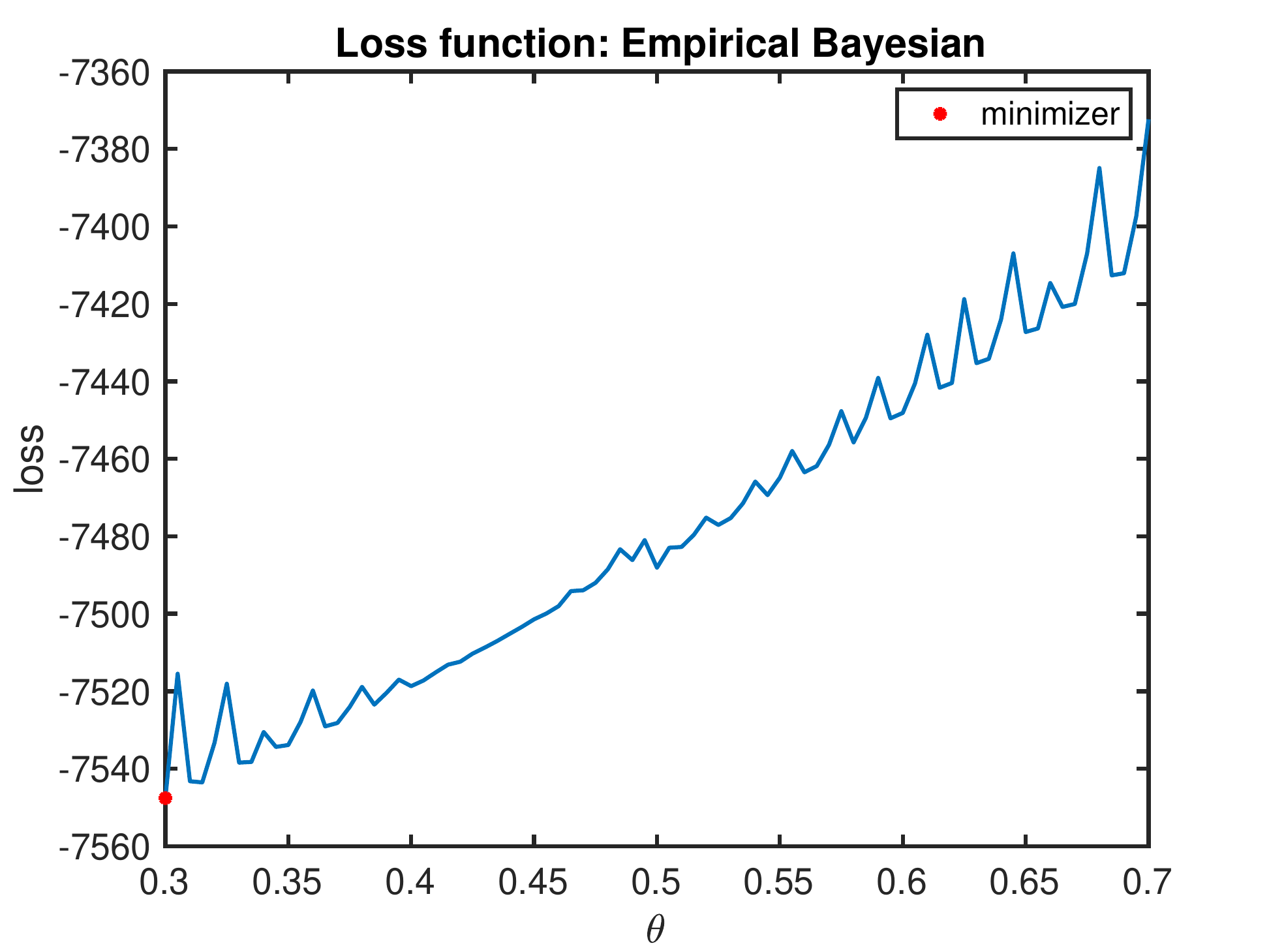}
    \includegraphics[width=6cm]{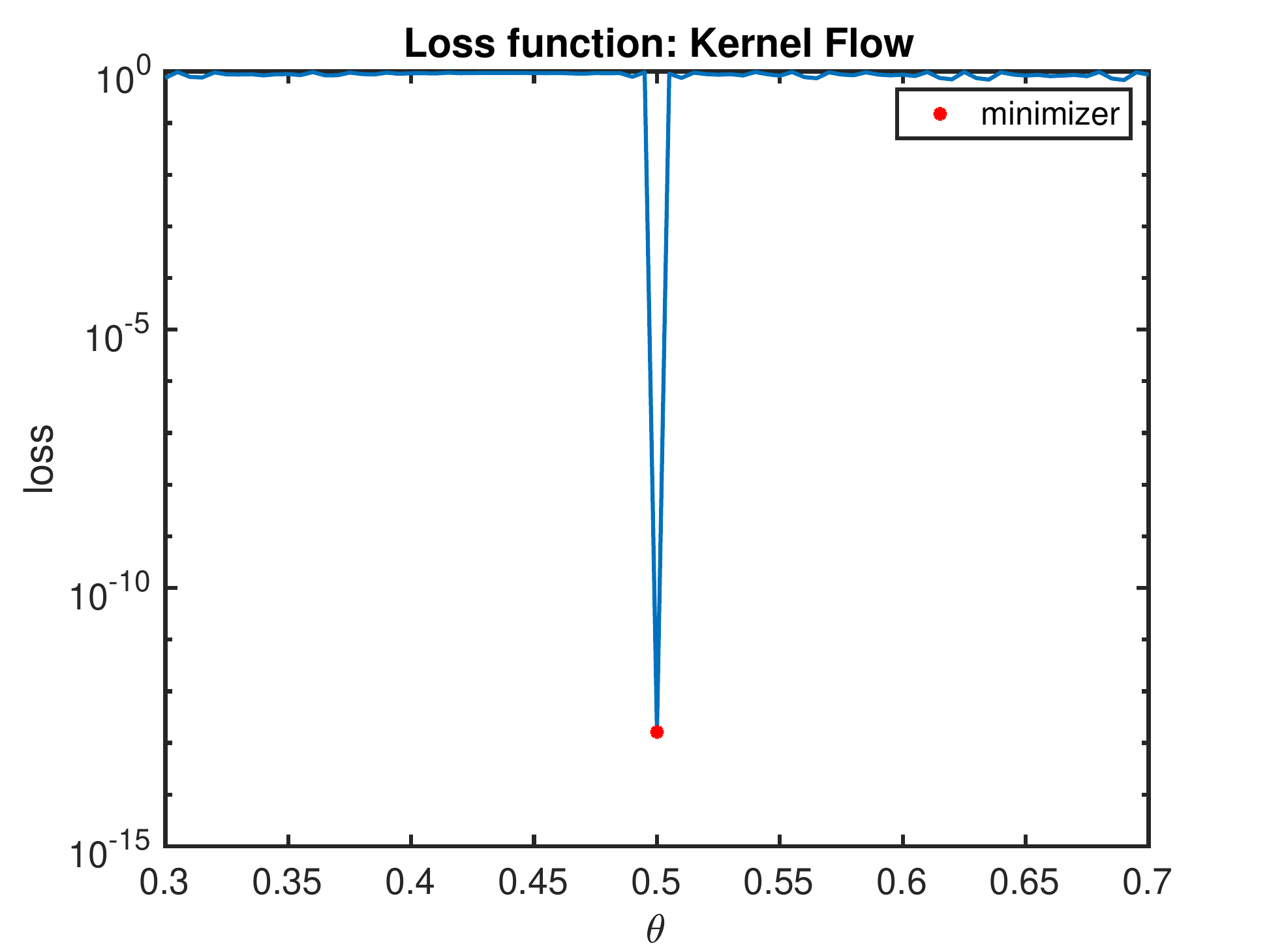}
    \caption{Loss function for estimating the discontinuity parameter under model misspecification. Left: EB; right: KF}
    \label{fig: loss function of estimator of discontinuity, mis-specification}
    \end{figure}

\subsection{Deterministic model}
 In this subsection, we consider the EB and KF estimators for the parameter $t$ 
 in the GP model $\cN(0,(-\Delta)^{-t})$ where $\Delta$ is equipped with  homogeneous Dirichlet boundary conditions on $[0,1]$. However, rather than choosing $\ud$ that is drawn from the GP $\cN(0,(-\Delta)^{-s})$ for some $s$ (as we did in Section \ref{sec: recover regularity}), we choose it  be the solution to the equation $(-\Delta)^{s} \ud(\cdot) =\delta(\cdot-1/2)$, i.e., $\ud$ is the Green function corresponding to the differential operator $(-\Delta)^{s}$ and evaluated at $y=1/2$. Since $\ud$ has no stochastic background, we understand this situation as a deterministic model misspecification.

We observe the value of $\ud$ on the $2^9$ equidistributed points of the total $2^{10}$ grid points used for discretization. We conduct numerical experiments to find the value of the EB and KF estimators. Our experiments show that the EB estimator returns $2s$ and the KF estimator returns $s$ for this one dimensional example. The loss function in the case  $s=1.2$ is shown in Figure \ref{fig: Loss function for estimating the regularity parameter under deterministic}.
\begin{figure}[ht]
    \centering
    \includegraphics[width=6cm]{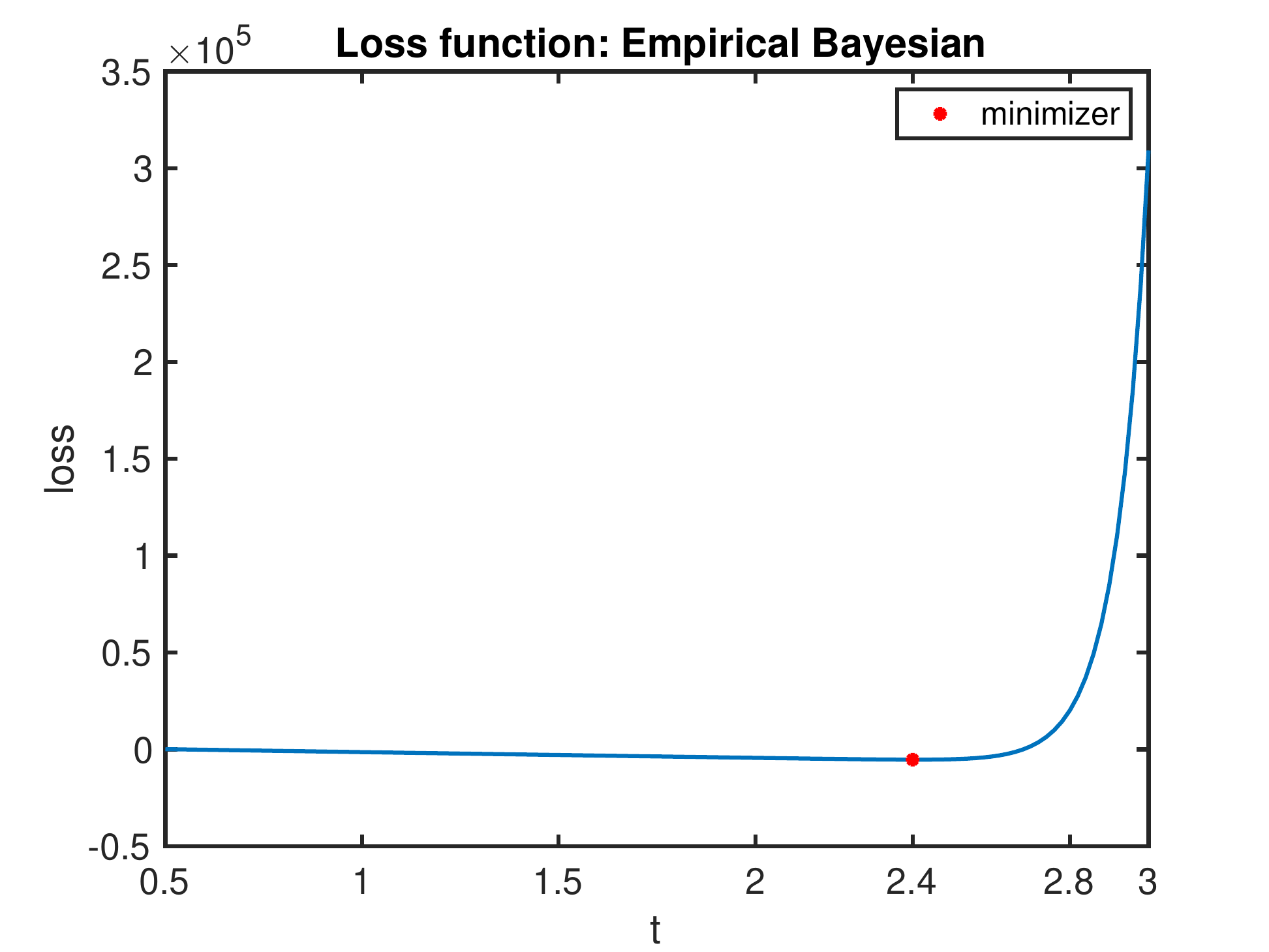}
    \includegraphics[width=6cm]{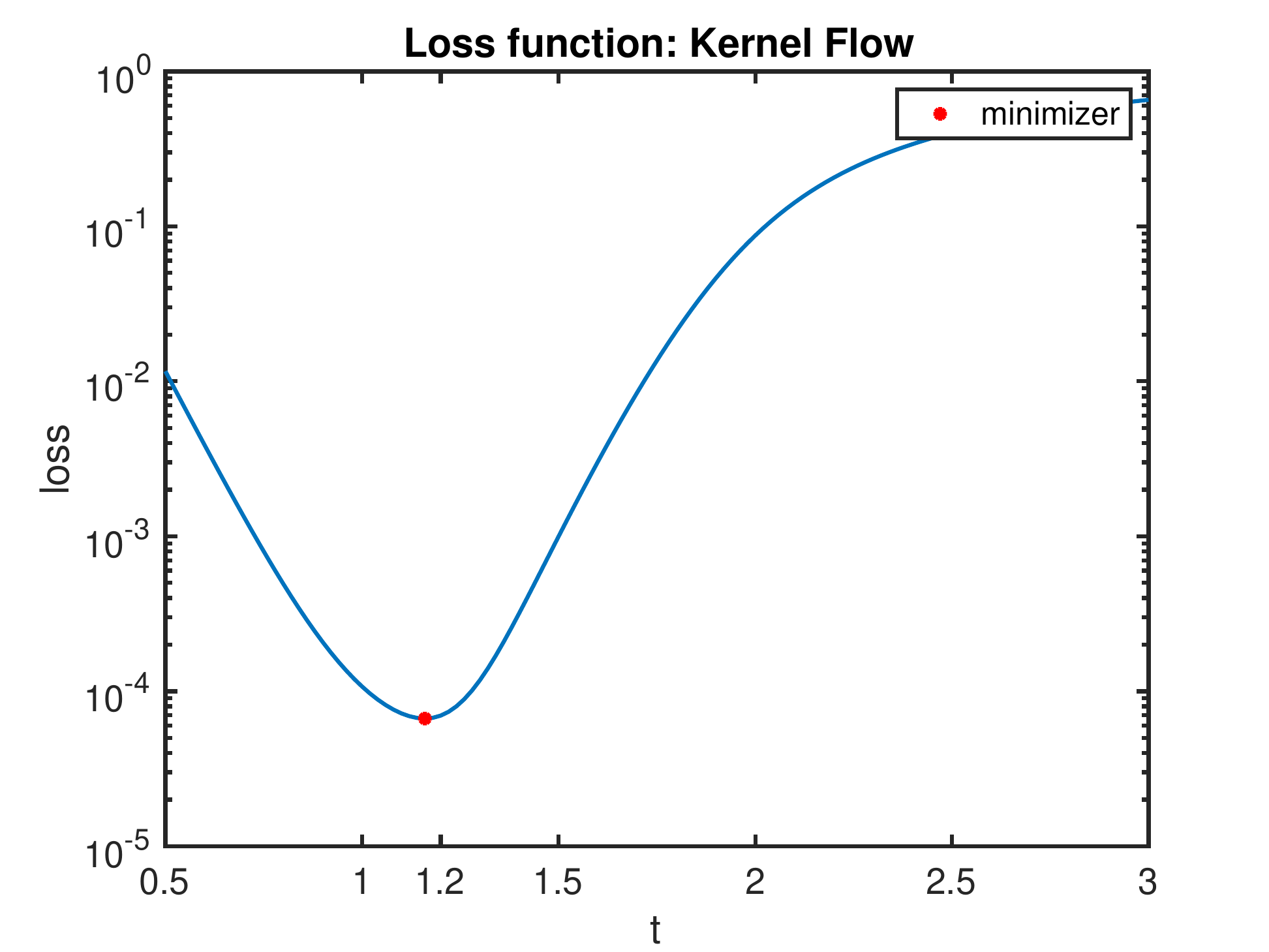}
    \caption{Loss function for estimating the regularity parameter under deterministic $\ud$. Left: EB; right: KF}
    \label{fig: Loss function for estimating the regularity parameter under deterministic}
\end{figure} 

We now describe some regularity considerations in order to understand the observed phenomenon. In this one dimensional example, $\delta(\cdot-1/2)$ belongs to $H^{\eta}([0,1])$ for any $\eta < -1/2$, so the solution $u \in H^{2s+\eta}([0,1])$ for any $\eta < -1/2$. It is of critical regularity $2s-1/2$, but this criticality is not homogeneous: it is caused by the presence of a singularity induced by the Dirac function.

 The discussion in Section \ref{sec: recover regularity} implies KF will recover $s-1/4$ while EB recovers $2s$ for a function with homogeneous critical regularity $2s-1/2$. However, the experiments here show that KF recovers $s$ while EB recovers $2s$, for this function with critical regularity $2s-1/2$; unlike the setting in Section \ref{sec: recover regularity}, here the ground truth lacks spatial homogeneity. This 
suggests that the KF estimator for the regularity parameter is sensitive
to whether the regularity of the target function is spatially homogeneous or not. This fact is not surprising, considering the vast literature on adaptive approximation for functions with singularities, which implies the presence of a singularity will exert considerable influence on the approximation error
resulting from minimizing the KF loss function. In this example, the optimal approximation in KF error comes at $t=s$. We can understand this phenomenon as follows.  Recall $\ud=(-\Delta)^{-s}\delta(\cdot-1/2)$. Using $\cN(0,(-\Delta)^{-t})$ in the GPR is equivalent to using the basis functions $\text{span}_{j \in J_q} \{(-\Delta)^{-t}\updelta(\cdot-x_j)\}$ (as in Section \ref{sec: set-up of recovering regularity}) with $x_i$ being the data points indexed by $j \in J_q$, to approximate $\ud$. When $t=s$ and one of the $x_j=1/2$, the ground truth will just be in the basis functions set, so it is straightforward to imagine $t=s$ leads to the smallest approximation error, and KF picks this value.
 
 We understand the fact that EB still picks $t=2s$ by making the following observation: there are only two terms in the EB loss function. The log determinant term remains the same for each $t$ when $\ud$ changes. For the norm term $\|u(\cdot,t,q)\|_t^2$, the blow-up rate depends on the regularity of $\ud$. Here, it makes no difference whether the regularity of $\ud$ is spatially homogeneous or not. 

{ \subsection{Discussions}
The above numerical experiments reveal complicated behavior of EB and KF with respect to model misspecification. In the second experiment, we found that KF is robust while EB is not, for a certain type of GP model misspecification. This appears natural since EB is based on probabilistic modeling whilst KF is purely based on approximation theoretic criteria. In Subsection \ref{sec: Stochastic model misspecification for recovering discontinuity} the prior used in EB is mutually singular with respect to the GP that $\ud$ is drawn from and it is not
suprising that EB is fragile. On the other hand, KF  does not require probabilistic modeling to motivate it,
and so its robustness to misspecifications behaves differently. Indeed, in the second experiment, the discontinuity point influences the approximation accuracy a lot, and even the kernel used in defining KF is misspecified, KF still succeeds in selecting the correct parameter,  as it focuses on the approximation accuracy rather than statistical inference.}

{
In the well-specified cases, e.g. experiments in Section \ref{sec: recover regularity}, EB outperforms KF in terms of the variance of estimators. Therefore, if $\ud$ is a random object and we know the prior correctly, then EB should be a preferable choice for estimating parameters. If this is not the case and misspecification occurs, EB might be vulnerable and KF could be a potential alternative.
}

\section{Concluding Remarks}
\label{sec: discussions}
In this paper, we have studied the Empirical Bayes and Kernel Flow 
approaches to hyperparameter learning. The first approach is based on statistical considerations, while the second approach originates from an approximation theoretic viewpoint. Their distinct objectives lead them to different behaviors and different interpretations of optimality.

For the Mat\'ern-like process model, we made a detailed theoretical study of the recovery of the regularity parameter. We proved the EB estimator converges to $s$, while the KF estimator converges to $\frac{s-d/2}{2}$, both results holding in probability in the large data limit if the regularity of the GP that $\ud$ draws from is $s$. Our experiments illustrate that, in terms of the $L^2$ error $\|u(\cdot,t,q)-\ud\|_0^2$, the parameter $t=\frac{s-d/2}{2}$ relates to the minimal $t$ that achieves the fast error rate while $t=s$ relates to the $t$ that achieves the smallest error, averaged over the GP $\ud \sim \cN(0,(-\Delta)^{-s})$. This demonstrates the different drivers that guide the EB and KF methods in selecting the parameters. The statistical and approximation theoretic principles behind them lead to the differences between them.

In the theoretical study, we developed a Fourier analysis toolkit for this problem, and as a byproduct, we showed the consistency of recovering $\sigma$ in the Mat\'ern-like process for the EB method. Recovery of the lengthscale parameter and recovery of several parameters simultaneously was studied via numerical experiments. It is of future interest to perform theoretical studies explaining these empirically observed phenomena. Furthermore, the theory in this paper is based on an equidistributed design for the data location, and the generalization to randomized design remains a potential further direction. Also, our focus in this paper is on the noiseless observation setting, and an extension to the noisy case is of future theoretical interest.

Our numerical experiments for additional well-specified and misspecified models extend the scope of this paper beyond the Mat\'ern-like kernels. Both the two estimators work very well in the well-specified models we consider;
we would like to explore this more in the future, both theoretically and numerically, potentially in more complex models that are present in machine learning. The variance and robustness of the estimators behave differently for the misspecified models. The variabilities in robustness are in line with our expectation since these estimators follow from different decision rules; these rules can vary considerably in sensitivity to model mismatches of different kinds. In practice, users should choose the correct approach to avoid high sensitivity to likely model errors present.

As a summary, this paper demonstrates some basic aspects of the difference between Bayesian and approximation theoretic approaches for hierarchical learning. Generally, it is of interest to study EB and KF for other types of models and to study other parameter selection criteria based on the two principles beyond EB and KF, such as a fully Bayesian approach or another choice of $\mathsf{d}$ for the approximation, and identify their pros and cons under different scenarios. We are interested in exploring the theoretical and practical performance of methods under such a framework, and we believe that a diversity in such methods will enable
users to deal with the model misspecification that is to be expected
in many applications.

\vspace{0.3in}

\noindent{\bf Acknowledgements} YC gratefully acknowledges the support of the Caltech Kortchack Scholar Program. HO gratefully acknowledges support from  AFOSR (grant FA9550-18-1-0271) and ONR (grant N00014-18-1-2363). AMS is grateful to AFOSR (grant FA9550-17-1-0185) and NSF (grant DMS 18189770) for financial support. 
YC, HO and AMS gratefully acknowledge support from AFOSR MURI (FA9550-20-1-0358).

\bibliographystyle{amsplain}

\section{Appendix: Proofs}
\label{sec: proof}
\subsection{Proof of Proposition \ref{prop: characterize F(t,q)}}
\label{Proof of Proposition prop: characterize F(t,q)}
\begin{proof}
        Let $\varphi_j(x)=(-\Delta)^{-t}\updelta(x-x_j)$ and in particular $\varphi_0(x)=(-\Delta)^{-t}\updelta(x)$. We have for $m \in \bZ^d$,
        \begin{equation*}
        \hat{\varphi}_0(m)=\begin{cases}
        (4\pi^2)^{-t}|m|^{-2t}, &m \neq 0\\
        0, &m=0\, .
        \end{cases}
        \end{equation*} 
        We introduce the translation operator $\tau_{j2^{-q}}$ which acts on function $u: \bT^d \to \bR$ and is defined by $$(\tau_{j2^{-q}}u)(x)=u(x_1-j_12^{-q},x_2-j_22^{-q},...,x_d-j_d2^{-q})$$ for $j=(j_1,j_2,...,j_d) \in \bZ^d$ and $x=(x_1,x_2,...,x_d) \in \bR^d$. Then, for $j \in J_q$, we have the relation $\updelta(\cdot-x_j)=\tau_{j2^{-q}}\updelta(\cdot)$. Using the property of the Fourier coefficients, we obtain
         $$\hat{\varphi}_j(m)=\hat{\varphi}_0(m)e^{-2\pi i\left<j2^{-q},m\right>}=
         \begin{cases}
         (4\pi^2)^{-t}|m|^{-2t}e^{-2\pi i\left<j2^{-q},m\right>}, &m \neq 0\\
         0, &m=0\, .
         \end{cases}$$
         By definition, $\hat{\cF}_{t,q}$ is the span of such $\hat{\varphi}_j$ for $j \in J_q$. Hence, for any $g \in \hat{\cF}_{t,q}$, it can be written as a linear combination of these functions. Equivalently, there exists a $2^q$-periodic function $p$ such that
         \begin{equation*}
         g(m)=\begin{cases}
         |m|^{-2t}p(m), &m \neq 0\\
         0, &m=0\, .
         \end{cases}
         \end{equation*} 
         This gives the desired representation of $g$.
            \end{proof}
        
    \subsection{Proof of Theorem \ref{thm: Fourier representation for Pu}}
        \label{Proof of Theorem thm: Fourier representation for Pu}
        \begin{proof}
        By Proposition \ref{prop: characterize F(t,q)}, there exists a $2^q$-periodic function $p_1(m)$ on $\bZ^d$, such that, 
        \[\hat{u}(m,t,q)=
        \begin{cases}
        |m|^{-2t}p_1(m), &m \neq 0\\
        0, &m=0\, .
        \end{cases}\]
        By the definition of GPR, we have $[\ud(\cdot)-u(\cdot,t,q),\updelta(\cdot-x_j)]=0$ for every data point $x_j$. In the Fourier domain, according to the characterization of $\hat{\cF}_{t,q}$, this orthogonality leads to
        \begin{equation}
        \label{eqn: u-Pu orthogonal Fourier domain}
        \sum_{m \in \bZ^d} (\hat{u}(m)-\hat{u}(m,t,q))p(m)=0
        \end{equation}
        for $p:\bZ^d \to \bC$ being any $2^q$-periodic function. 
        Recalling Definition \ref{def: periodization}, we have
        \begin{equation}
        \label{eqn:def of Tu}
        (T_q\hat{u})(m)=\sum_{\beta \in \bZ^d} \hat{u}(m+2^q\beta)\, .
        \end{equation}
        The fact that the above sum converges may
        be seen as a consequence of the Cauchy–Schwarz inequality and the regularity of $u$  (recall $t \geq d/2 +\delta$). 
        Using \eqref{eqn:def of Tu} and the representation of $\hat{u}(m,t,q)$, we reformulate \eqref{eqn: u-Pu orthogonal Fourier domain} as
        \[\sum_{m \in B_q^d} \left((T_q\hat{u})(m)-M_q^t(m)p_1(m)\right)p(m)=0\, . \]
        The above formula holds for any $2^q$-periodic function $p$. Let $g(m)=(T_q\hat{u})(m)-M_q^t(m)p_1(m)$, then we get that $g$ is a $2^q$-periodic function on $\bZ^d$ and that
        \[\sum_{m \in B_q^d} g(m)p(m)=0\]
        holds for any $2^q$-periodic function $p$. This implies that $g(m)=0$. Hence, we get
        \[p_1(m)=\frac{(T_q\hat{u})(m)}{M_q^t(m)}\, .\]
        Plugging this expression into the above representation formula for $\hat{u}(m,t,q)$ leads to
        \[\hat{u}(m,t,q)=
        \begin{cases}
        0, & \text{if} \ m = 0\\
        |m|^{-2t}\frac{(T_q\hat{u})(m)}{M_q^t(m)}, & \text{else}\, .
        \end{cases} \]
        This completes the proof.
    \end{proof}

    \subsection{Proof of Lemma \ref{lemma: estimate of the M t q term}}
    \label{Proof of Lemma lemma: estimate of the M t q term}
    \begin{proof}
    Recall the definition
        \begin{equation*}
         M_q^t(m):=
         \begin{cases} \sum_{\beta \in \bZ^d \backslash \{0\}} |2^q\beta|^{-2t}, & \text{if}  \ m=j\cdot2^q \ \text{for some} \ j \in \bZ^d\\
         \sum_{\beta \in \bZ^d} |m+2^{q}\beta|^{-2t}, & \text{else}\, .
         \end{cases}
         \end{equation*}
    Because of the periodicity of $M_q^t$, we need only to study  $m \in B_q^d$. We split it into two cases.
    \begin{enumerate}
        \item If $m=0$, then $M_q^t(m)=\sum_{\beta \in \bZ^d \backslash \{0\}} |2^q\beta|^{-2t}=2^{-2qt} \sum_{\beta \in \bZ^d \backslash \{0\}} |\beta|^{-2t} \simeq 2^{-2qt}$.
        \item If $m \in B_q^d \backslash \{0\}$, then $M_q^t(m)=\sum_{\beta \in \bZ^d} |m+2^{q}\beta|^{-2t}=|m|^{-2t}+\sum_{\beta \in \bZ^d\backslash \{0\}} |m+2^{q}\beta|^{-2t}$. Since $B_q^d=[-2^{q-1},2^{q-1}-1]^{\otimes d}$, each component of $m \in B_q^d$ is bounded by $2^{q-1}$ in amplitude, and therefore each component of $2^{-q}m$ is bounded by $1/2$ in amplitude. So, it follows that \[\sum_{\beta \in \bZ^d\backslash \{0\}} |m+2^{q}\beta|^{-2t}=2^{-2qt}\sum_{\beta \in \bZ^d\backslash \{0\}} |2^{-q}m+\beta|^{-2t}\simeq 2^{-2qt}\, . \]
    Then, we get $|m|^{-2t} \leq M_q^t(m) \lesssim |m|^{-2t}+2^{-2qt} \lesssim |m|^{-2t}$ where we have used the fact that $|m| \lesssim 2^{q}$. Therefore, it holds that $M_q^t(m) \simeq |m|^{-2t}$.
    \end{enumerate}
     As a byproduct of the above proof, we also get $M_q^t(m)-|m|^{-2t} \simeq 2^{-2qt}$.
    \end{proof}
    
    \subsection{Proof of Lemma \ref{lemma: uniform convergence of series}}
    \begin{proof}
    \label{Proof of Lemma lemma: uniform convergence of series}
        First, we prove the pointwise convergence (i.e., for each fixed $r$), then move on to prove uniform convergence. To achieve this, we calculate the variance:
        \begin{align*}
        \mathrm{Var}(\alpha(r,q))&\simeq 2^{-2rq}\sum_{m \in B_q^d \backslash \{0\}} |m|^{2r-2d}\\ &\lesssim 2^{-2rq} \int_1^{2^q} x^{2r-2d+d-1}\, \rd x=2^{-2rq} \int_1^{2^q} x^{2r-d-1}\, \rd x\, .
        \end{align*}
        For $r = d/2$, the integral gives $\log (2^q)=q\log 2$; for $r \neq d/2$, it is $\frac{1}{2r-d}(2^{q(2r-d)}-1)$. In both cases, we have $\lim_{q \to \infty} \mathrm{Var}(\alpha(r,q))=0$. Thus, $\alpha(r,q)$ converges in $L^2$ to the limit of its expectation, which we may
        calculate as follows:
        \[\lim_{q\to \infty}\bE \alpha(r,q)=\lim_{q \to \infty} \sum_{m \in B_q^d \backslash \{0\}} (2^{-q})^d|2^{-q}m|^{r-d}=\int_{[-1/2,1/2]^d} |y|^{r-d}\, \rd y:=\gamma(r)>0 \, .  \]
        Hence, we get $\lim_{q\to \infty} \alpha(r,q) = \gamma(r)>0$ in $L^2$ for every $r \in [\epsilon,1/\epsilon]$, and the convergence also holds in probability. We may now proceed to show uniform convergence. We rely on Exercise 3.2.3 in \cite{van1996weak}. Based on that, it suffices to prove $\alpha(r,q)$ is uniformly Lipschitz continuous as a function of $r$ for $q \in \bN$. Pick any $r_1,r_2 \in [\epsilon,1/\epsilon]$, then
        \begin{align*}
        &|\alpha(r_1,q)-\alpha(r_2,q)|\\
        =&\sum_{m \in B_q^d \backslash \{0\}} 2^{-qd}|(|2^{-q}m|^{r_1-d}-|2^{-q}m|^{r_2-d})|\\
        \leq& \sum_{m \in B_q^d \backslash \{0\}} 2^{-qd}|r_1-r_2|(|2^{-q}m|^{\epsilon-d}+|2^{-q}m|^{1/\epsilon-d})|\log (2^{-q}m)|\xi_m^2\, ,
        \end{align*} 
        where in the last step we have used the fact that $||2^{-q}m|^{r_1-d}-|2^{-q}m|^{r_2-d}|=||2^{-q}m|^{\eta-d} \log (2^{-q}m)(r_1-r_2)|$ for some $\eta$ that lies between $r_1$ and $r_2$, and we use the bound $r_1,r_2 \in [\epsilon,1/\epsilon]$. Now, we define the random series:
        \[\mathsf{L}(q):= 2^{-qd}\sum_{m \in B_q^d \backslash \{0\}} (|2^{-q}m|^{\epsilon-d}+|2^{-q}m|^{1/\epsilon-d})|\log (2^{-q}m)|\xi_m^2\, . \]
        We calculate its variance as follows:
        \begin{align*}
        \mathrm{Var}(\mathsf{L}(q))&\simeq 2^{-2dq}\sum_{m \in B_q^d \backslash \{0\}} (|2^{-q}m|^{2\epsilon-2d}+|2^{-q}m|^{2/\epsilon-2d})\log^2 |2^{-q}m|\\
        &\lesssim 2^{-qd} \left(\int_{2^{-q}}^1 t^{2\epsilon-2d+d-1} \log^2 t \, \rd t + \int_{2^{-q}}^1 t^{2/\epsilon-2d+d-1} \log^2 t \, \rd t \right)\\
        &= 2^{-qd} \int_{2^{-q}}^1 (t^{2\epsilon-d-1}+t^{2/\epsilon-d-1}) \log^2 t\, \rd t\, ,\\
        &\lesssim 2^{-qd} \int_{2^{-q}}^1 (t^{\epsilon-d-1}+t^{1/\epsilon-d-1})\, \rd t \lesssim 2^{-q\epsilon}\, .
        \end{align*}
        The last term will go to $0$ as $q$ goes to infinity. Thus, $\mathsf{L}(q)$ converges in $L^2$ (and thus in probability) to $\mathsf{L}^*=\lim_{q\to \infty} \bE \mathsf{L}(q)$, which is
        \begin{align*}
        \lim_{q \to \infty} \bE \mathsf{L}(q)&= \lim_{q \to \infty} 2^{-qd}\sum_{m \in B_q^d \backslash \{0\}} (|2^{-q}m|^{\epsilon-d}+|2^{-q}m|^{1/\epsilon-d})\log^2 |2^{-q}m|\\
        &=\int_{[-1/2,1/2]^d}(|y|^{\epsilon-d}+|y|^{1/\epsilon-d})\log^2|y|\, \rd y\\
        &\lesssim \int_{[-1/2,1/2]^d}(|y|^{\epsilon/2-d}+|y|^{1/(2\epsilon)-d})\, \rd y < \infty \, .
        \end{align*}
        Using Markov's inequality we deduce that, for any $\epsilon' > 0$, it holds that
        \[\bP(|\mathsf{L}(q)-\mathsf{L}^*|\geq \epsilon')\leq \frac{\bE |\mathsf{L}(q)-\mathsf{L}^*|^2}{(\epsilon')^2}\leq \frac{2^{-q\epsilon}}{(\epsilon')^2}\, . \]
        Thus,
        \[\sum_{q=1}^{\infty} \bP(|\mathsf{L}(q)-\mathsf{L}^*|\geq \epsilon')\leq \sum_{q=1}^{\infty}\frac{2^{-q\epsilon}}{(\epsilon')^2} < \infty\, . \]
        From the Borel-Cantelli lemma it follows that $\lim_{q \to \infty} \mathsf{L}(q)=\mathsf{L}^*$ almost surely, and therefore $\mathsf{L}(q)$ is bounded uniformly for $q$ almost surely. Since $|\alpha(r_1,q)-\alpha(r_2,q)|\leq \mathsf{L}(q)|r_1-r_2|$, it follows that $\alpha(r,q)$ is uniformly Lipschitz continuous as a function of $r$ for $q \in \bN$. Invoking Exercise 3.2.3 in \cite{van1996weak} concludes this case.\\
        For the case $r=0$, we follow the same strategy as in the previous case. First, we calculate the corresponding variance:
        \begin{align*}
        \mathrm{Var}(\alpha(0,q))&\simeq\frac{1}{q^2} \sum_{m \in B_q^d \backslash \{0\}} |m|^{-2d}\\
        &\lesssim \frac{1}{q^2} \int_1^{2^{q}} x^{-2d+d-1}\, \rd x \lesssim \frac{1}{q^2}
        \end{align*}
        where the last term goes to $0$ as $q$ goes to infinity. Then, we calculate the expectation:
        \begin{align*}
        \bE \alpha(0,q)=\frac{1}{q} \sum_{m \in B_q^d \backslash \{0\}} |m|^{-d}\, .
        \end{align*}
        The limit when $q \to \infty$ is identified through the following
        calculations:
        \begin{align*}
            \lim_{q \to \infty} \frac{1}{q} \sum_{m \in B_q^d \backslash \{0\}} |m|^{-d} &= \lim_{q \to \infty} \sum_{m \in B_{q+1}^d \backslash B_q^d} |m|^{-d}\\
            &=\lim_{q \to \infty} 2^{-qd} \sum_{m \in B_{q+1}^d \backslash B_q^d} |2^{-q}m|^{-d}\\
            &=\int_{[-1,1]^d\backslash[-1/2,1/2]^d} |x|^{-d}\, \rd x < \infty;
        \end{align*}
        here we have used the definition of the Riemann integral. Finally, we conclude that $\lim_{q \to \infty} \alpha(0,q)=\gamma(0)$ in probability for $\gamma(0) \in (0,\infty)$.
    \end{proof}
    
    \subsection{Proof of Proposition \ref{prop: bound on the log term}}
    \label{Proof of Proposition prop: bound on the log term}
        \begin{proof}
        First, we have the relation \[\log \det K(t,q)=\log \det K(t,q-1)+ \log \det  (K(t,q)/ K(t,q-1))\] where $K(t,q)/ K(t,q-1)$ is the Schur complement of $K(t,q-1)$ in $K(t,q)$. Due to the variational property of the Schur complement (see Lemma 13.24 in \cite{owhadi2019operator}), the smallest and largest eigenvalues of $K(t,q)/ K(t,q-1)$ satisfy (in the dual norm $\|\cdot\|_{-t}$)
        \begin{equation} 
        \label{eqn: Schur complement variational}
        \begin{aligned}
        \lambda_{\min}(K(t,q)/ K(t,q-1))&\geq \inf_{y \in \bR^{|J_q|}}\frac{\|\sum_{j\in J_q} y_j \updelta(x-x_j)\|_{-t}^2}{|y|^2}\, ,\quad \text{and}\\
        \lambda_{\max}(K(t,q)/ K(t,q-1))&\\
        =\sup_{y \in \bR^{|J_q|}} \inf_{z \in \bR^{|J_{q-1}|}} & \frac{\|\sum_{j\in J_q} y_j \updelta(x-x_j)-\sum_{j' \in J_{q-1}} z_{j'} \updelta(x-x_{j'})\|_{-t}^2}{|y|^2}\, .
        \end{aligned}
        \end{equation}
        These two formulae will be crucial in the subsequent 
        analysis.
        We start by estimating the smallest and largest eigenvalues of the Schur complement. Let $w=(-\Delta)^{-t}\sum_{j\in J_q} y_j \updelta(x-x_j)$, whose Fourier coefficients are
        \begin{equation}
        \hat{w}(m)=\begin{cases}
        0, & \text{if} \ m = 0\\
        (4\pi^2)^{-t}|m|^{-2t}g(m)
        , & \text{else} \, ,
        \end{cases}
        \end{equation}
        where, the function $g(m)$ is defined by
        \begin{equation}
        g(m)=\sum_{j \in J_q} y_j \exp(2\pi i\langle j2^{-q}, m \rangle)\, .
        \end{equation}
        For the smallest eigenvalue, we write
        \begin{align*}
        \|\sum_{j\in J_q} y_j \updelta(x-x_j)\|_{-t}^2=\|w\|_t^2 &=(4\pi^2)^{t}\sum_{m \in \bZ^d \backslash \{0\}} |m|^{2t}|\hat{w}(m)|^2\\
        &=(4\pi^2)^{-t}\sum_{m \in \bZ^d \backslash \{0\}} |m|^{-2t}|g(m)|^2.
        \end{align*}
        Notice that 
        \[\sum_{m \in \bZ^d \backslash \{0\}} |m|^{-2t}|g(m)|^2 = \sum_{m \in B_q^d} M_q^t(m)|g(m)|^2 \gtrsim 2^{-2tq} \sum_{m \in B_q^d} |g(m)|^2\] and
        \begin{equation}
        \label{eqn: summation of g(m)}
        \begin{aligned}
        \sum_{m \in B_q^d} |g(m)|^2&=\sum_{m \in B_q^d} |\sum_{j \in J_q} y_j \exp(2\pi i\langle j2^{-q}, m \rangle |^2\\
        &=\sum_{m \in B_q^d} \sum_{j \in J_q} \sum_{l \in J_q} y_jy_l \exp(2\pi i\langle (j-l)2^{-q}, m \rangle\\
        &=\sum_{j \in J_q} \sum_{l \in J_q} y_jy_l\sum_{m \in B_q^d}\exp(2\pi i\langle (j-l)2^{-q}, m \rangle\\
        &\simeq 2^{qd}|y|^2\, .
        \end{aligned}
        \end{equation}
      In the last line we have used the fact that
        \begin{equation*}
        \sum_{m \in B_q^d}\exp(2\pi i\langle (j-l)2^{-q}, m \rangle=\begin{cases}
        0, & \text{if} \ j-l\neq 0\\
        \sum_{m \in B_q^d}1 \simeq 2^{qd}
        , & \text{if}\ j-l=0 \, .
        \end{cases}
        \end{equation*}
        Thus, combining the above results, we obtain the bound on the smallest eigenvalue
        \[\lambda_{\min}(K(t,q)/ K(t,q-1))\gtrsim 2^{-q(2t-d)}\, . \]
        We then move to consider the largest eigenvalue. First, notice that
        \[  \inf_{z \in \bR^{|J_{q-1}|}}\|\sum_{j\in J_q} y_j \updelta(x-x_j)-\sum_{j' \in J_{q-1}} z_{j'} \updelta(x-x_{j'})\|_{-t}^2=\inf_{v \in \cF_{t,q-1}} \|w-v\|_t^2\, . \]
        Naturally, one can express the optimal $v$ in the above variational formulation using the Fourier series representation explained before. However, this will lead to many interactions between different frequencies. To make the analysis cleaner, we adopt another strategy. We first approximate the function $w$ by a band-limited function, whose projection into $\cF_{t,q-1}$ will be more concise. Precisely, define a band limited version of $w$, written as $w_h$, by
        \begin{equation}
        \hat{w}_h(m)=
        \begin{cases}
        \hat{w}(m), & \text{if} \ m \in B_{q-1}^d\\
        0
        , & \text{if}\ m \in (B_{q-1}^d)^c \, .
        \end{cases}
        \end{equation}
        To estimate $\inf_{v \in \cF_{t,q-1}} \|w-v\|_t^2$, we follow the two steps below:

         \textit{Step 1}: we prove $\|w-w_h\|_t^2 \lesssim 2^{-q(2t-d)}|y|^2$. Let us calculate the quantity directly:
        \begin{align*}
        \|w-w_h\|_t^2&=(4\pi^2)^{-t}\sum_{m \in (B_{q-1}^d)^c}|m|^{-2t}|g(m)|^2\\
        &=(4\pi^2)^{-t}\left(\sum_{m \in \bZ^d\backslash \{0\}} |m|^{-2t}|g(m)|^2 - \sum_{m \in B_{q-1}^d} |m|^{-2t}|g(m)|^2\right)\\
        &=(4\pi^2)^{-t}\left(\sum_{m \in B_q^d} M_q^t(m)|g(m)|^2 - \sum_{m \in B_{q-1}^d} |m|^{-2t}|g(m)|^2\right)\\
        &\lesssim 2^{-2qt} \sum_{m \in B_q^d} |g(m)|^2 \lesssim 2^{-q(2t-d)}|y|^2.
        \end{align*}
        Here we have used the fact that $M_q^t(m)-|m|^{-2t} \lesssim 2^{-2qt}$ for $m \in B_{q-1}^d$ and $M_q^t(m)\lesssim 2^{-2qt}$ for $m \in B_q^d \backslash B_{q-1}^d$, according to the results in Lemma \ref{lemma: estimate of the M t q term}. In the last line, the bound \eqref{eqn: summation of g(m)} is applied.
        
       \textit{Step 2}: We prove $\inf_{v \in \cF_{t,q-1}} \|w_h-v\|_t^2 \lesssim 2^{-q(2t-d)}|y|^2$. Based on Theorem \ref{thm: Fourier representation for Pu}, we know the optimal $v$ for this variational problem has the Fourier coefficients
        \[\hat{v}(m)=
        \begin{cases}
        0, & \text{if} \ m = 0\\
        |m|^{-2t}\frac{(T_{q-1}\hat{w}_h)(m)}{M_{q-1}^t(m)}, & \text{else}\, .
        \end{cases} \]
        Then, using the Fourier representation of the norm, we get
        \begin{align*}
        &\|w_h-v\|_t^2\\
        \simeq&\sum_{m \in \bZ^d \backslash \{0\}} |m|^{2t}|\hat{w}_h(m)-\hat{v}(m)|^2\\
        =&\sum_{m \in B_{q-1}^d \backslash \{0\}} |m|^{-2t}|g(m)-\frac{(T_{q-1}\hat{w}_h)(m)}{M_{q-1}^t(m)}|^2+\sum_{m \in (B_{q-1}^d)^c} |m|^{-2t}|\frac{(T_{q-1}\hat{w}_h)(m)}{M_{q-1}^t(m)}|^{2}\, .
        \end{align*}
        For the first term, since $w_h$ is band-limited, we know if $m \in B_{q-1}^d \backslash \{0\}$, then $(T_{q-1}\hat{w}_h)(m)=|m|^{-2t}g(m)$. Thus, we can write this term as
        \begin{align*}
        &\sum_{m \in B_{q-1}^d \backslash \{0\}} |m|^{-2t}|g(m)|^2\Bigl(1-\frac{|m|^{-2t}}{M_{q-1}^t(m)}\Bigr)^2\\
        =&\sum_{m \in B_{q-1}^d \backslash \{0\}} |m|^{-2t}|g(m)|^2\Bigl(\frac{M_{q-1}^t(m)-|m|^{-2t}}{M_{q-1}^t(m)}\Bigr)^2\\
        \overset{a)}{\lesssim}& \sum_{m \in B_{q-1}^d \backslash \{0\}} |m|^{-2t}|g(m)|^2\Bigl(\frac{2^{-4tq}}{|m|^{-4t}}\Bigr) \\
        \overset{b)}{\lesssim}& \sum_{m \in B_{q-1}^d \backslash \{0\}} 2^{-2tq}|g(m)|^2 \lesssim 2^{-q(2t-d)}|y|^2
        \end{align*}
        where in $a)$, we have used the fact that $M_{q-1}^t(m)-|m|^{-2t}\simeq 2^{-2tq}$ and $M_{q-1}^t(m)\simeq |m|^{-2t}$ for $m \in B_{q-1}^d \backslash \{0\}$ based on Lemma \ref{lemma: estimate of the M t q term}. In $b)$, we have used $|m|\lesssim 2^q$. The last inequality is obtained by recalling \eqref{eqn: summation of g(m)}.\\
        For the second term, we write
        \begin{align*}
        &\sum_{m \in (B_{q-1}^d)^c} |m|^{-2t}|\frac{(T_{q-1}\hat{w}_h)(m)}{M_{q-1}^t(m)}|^{2}\\
        =&\sum_{m \in \bZ^d \backslash \{0\}} |m|^{-2t}|\frac{(T_{q-1}\hat{w}_h)(m)}{M_{q-1}^t(m)}|^{2} - \sum_{m \in B_{q-1}^d\backslash \{0\}} |m|^{-2t}|\frac{(T_{q-1}\hat{w}_h)(m)}{M_{q-1}^t(m)}|^{2}\\
        \overset{c)}{=}&\sum_{m \in B_{q-1}^d\backslash \{0\}}(M_{q-1}^t(m)-|m|^{-2t})|\frac{(T_{q-1}\hat{w}_h)(m)}{M_{q-1}^t(m)}|^{2}\\
        =&\sum_{m \in B_{q-1}^d\backslash \{0\}}(M_{q-1}^t(m)-|m|^{-2t})|\frac{|m|^{-2t}g(m)}{M_{q-1}^t(m)}|^{2}\\
        \lesssim &\, 2^{-2tq} \sum_{m \in B_{q-1}^d\backslash \{0\}} |g(m)|^2 \lesssim 2^{-q(2t-d)}|y|^2\, ,
        \end{align*}
        where in $c)$, we have used the periodicity of the function $\frac{(T_{q-1}\hat{w}_h)(m)}{M_{q-1}^t(m)}$. 
        
        Now, combining Step 1 and 2 leads to the conclusion
        \[ \inf_{v \in \cF_{t,q-1}} \|w-v\|_t^2 \lesssim 2^{-q(2t-d)}|y|^2\, , \]
        and in particular, it implies 
         \[\lambda_{\max}(K(t,q)/ K(t,q-1))\lesssim 2^{-q(2t-d)}\, . \]
        As a consequence of the upper and lower bounds for the eigenvalues of the matrix $K(t,q)/ K(t,q-1)$, we deduce that they are all on the scale of $2^{-q(2t-d)}$. Let $C$ be a constant independent of $t,q$ such that $C^{-1}2^{-q(2t-d)} \preceq K(t,q)/ K(t,q-1) \preceq C2^{-q(2t-d)} $. Then,
        \begin{align*}
        (2^{qd}-2^{(q-1)d})((2t-d)(-q)\log 2-C)&\leq \log \det K(t,q)/ K(t,q-1) \\
        &\leq (2^{qd}-2^{(q-1)d})((2t-d)(-q)\log 2+C) \, .   
        \end{align*}
        Using the implied bounds on the recursion relation, we get    
        \[(2t-d)g_1(q)-Cg_2(q)+K(t,0)\leq \log \det K(t,q) \leq (2t-d)g_1(q)+Cg_2(q)+K(t,0) \, , \]
        where $g_1(q)=\sum_{k=1}^q (2^{kd}-2^{(k-1)d})(-k\log 2)$ and $g_2(q)=(2^{qd}-1)(2t-d)$. Summing the series in $g_1(q)$ leads to $g_1(q)\simeq -q2^{qd}\log2 \simeq -q2^{qd}$. The proof of Proposition \ref{prop: bound on the log term} is completed.
        \begin{remark}
         {The above technique of using the Schur complements is quite general and could be potentially applied to other operators such as heterogeneous Laplacians; see \cite{owhadi2019operator}. However, for the homogeneous Laplacian on the torus in this paper, we may also prove the result via a simpler approach. The key observation is that there is an explicit formula for the spectrum of $K(t,q)$, as also
         exploited in \cite[Sec.~6.7]{stein99book}. Indeed, using the formula for the spectrum given in Lemma \ref{lem:add} below, we get
        \begin{equation*}
            \begin{aligned}
                \log \det K(t,q) &= \sum_{m \in B_q^d} \log \left(2^{qd}(4\pi^2)^{-t}M_q^t(m)\right)\\
                & = qd2^{qd}\log 2 -2^{qd}t\log(4\pi^2) + \sum_{m \in B_q^d} \log M_q^t(m)\, .
            \end{aligned}
        \end{equation*}
        By Lemma \ref{lemma: estimate of the M t q term}, it holds that
        \begin{equation*}
        M_q^t(m)\simeq \begin{cases}
        2^{-2qt},\ \text{if} \ m = 0\\
        |m|^{-2t},\  \text{if} \ m \in B_q^d \backslash \{0\}\, .
        \end{cases}
    \end{equation*}
    That is, there exists a constant $C$ independent of $t$ such that \[-2t\log |m| - \log C \leq \log M_q^t(m)\leq -2t\log|m| + \log C\] for $m \in B_q^d \backslash \{0\}$, and $-2^{qt}\log 2 - \log C \leq \log M_q^t(0)\leq  -2^{qt}\log 2 + \log C$. Since \[\sum_{m \in B_q^d \backslash \{0\}} \log |m| \simeq \int_0^{2^q} r^{d-1}\log r\, \rd r \simeq q2^{qd}\, , \] 
    and $2^{qd} =o(q2^{qd})$, we get \[-(2t-d)q2^{qd} -C2^{qd}\lesssim \log \det K(t,q)\lesssim  -(2t-d)q2^{qd} +C2^{qd}\, .\]
    This completes the alternative proof of Proposition \ref{prop: bound on the log term}.}
    \end{remark}
    \end{proof}
    
    \begin{lemma} \label{lem:add}
    {The eigenvalues of $K(t,q)$ are 
        $2^{qd}(4\pi^2)^{-t}M_q^t(m)$ for $m \in B_q^d$, where $M_q^t(m)$ is defined in \eqref{eqn: def of M_q^t(m)}, with the corresponding eigenfunctions $\phi_m(\cX_q) \in \bR^{2^{qd}}$.
        }
        \end{lemma}
        \begin{proof} {We can prove this claim using Mercer's decomposition as follows.  First, for $x_i,x_j \in
        \cX_q$, it holds that
        \begin{equation*}
        \begin{aligned}
        K(t,q)_{i,j} &= \sum_{m \in \bZ^d \backslash \{0\}} (4\pi^2)^{-t}|m|^{-2t} \phi_m(x_i)\phi^*_m(x_j)  \\
        & = \sum_{m \in B_q^d} (4\pi^2)^{-t} M_q^t(m) \phi_m(x_i)\phi^*_m(x_j)
        \end{aligned}
        \end{equation*}
        where we have used the fact that $\phi_{m+2^q\beta}(x_i) = \phi_m(x_i)$ for any $\beta \in \bZ^d$ and $x_i \in \cX_q$. Thus, for every $n \in B_q^d$, we get
        \begin{equation*}
        \begin{aligned}
            \sum_{x_j \in \cX_q} K(t,q)_{i,j}\phi_n(x_j)&= \sum_{m \in B_q^d} (4\pi^2)^{-t} M_q^t(m) \phi_m(x_i)\sum_{x_j \in \cX_q} \phi^*_m(x_j)\phi_n(x_j)\\
            & = \sum_{m \in B_q^d} (4\pi^2)^{-t} M_q^t(m) \phi_m(x_i) 2^{qd}\updelta_{mn}\\
            & = 2^{qd}(4\pi^2)^{-t}M_q^t(m) \phi_n(x_i)\, ,
            \end{aligned}
        \end{equation*}
        where in the second equality we used the property of Fourier series.
        This implies $\phi_n(\cX_q)$ is an eigenfunction. The proof of the lemma is completed.}
        \end{proof}

    \subsection{Proof of Theorem \ref{thm: Consistency of Empirical Bayesian estimator}}
    \label{Proof of Theorem thm: Consistency of Empirical Bayesian estimator}
    \begin{proof}
        Recall the definition, \[s^{\text{EB}}(q)=\argmin_t \mathsf{L}^{\text{EB}}(t,q):=\|u(\cdot,t,q)\|_t^2 + \log \det K(t,q)\, . \]
        Define a rescaled version of the loss function by
        \[\tilde{L}_{\mathrm{EB}}(t,q)=\frac{1}{|g_1(q)|}\mathsf{L}^{\text{EB}}(t,q)=\underbrace{\frac{1}{|g_1(q)|}\|u(\cdot,t,q)\|_t^2}_{\circled{1}} + \underbrace{\frac{1}{|g_1(q)|}\log \det K(t,q)}_{\circled{2}}\, .\] 
        
        We note that by Proposition \ref{prop: bound on the log term}, we have $|g_1(q)|\sim q2^{qd}$. Now, we estimate the growth rate of $\circled{1}$ and $\circled{2}$ separately. From Proposition \ref{prop: |u(x,t,q)|^2 bound} and \ref{prop: bound on the log term}, we get
        \[\circled{1} \simeq \underbrace{\frac{1}{q}2^{-q(2s-2t+d)}\xi_0^2}_{\circled{3}}+\underbrace{\frac{1}{q}2^{-q(2s-2t)}\sum_{m \in B_q^d \backslash \{0\}} 2^{-q(2t-2s+d)}|m|^{2t-2s}\xi_m^2}_{\circled{4}}\, , \]
        and for the $\log \det$ part, it holds that
        \[d-2t+\frac{-Cg_2(q)+K(t,0)}{|g_1(q)|}\leq \circled{2}\leq d-2t+\frac{Cg_2(q)+K(t,0)}{|g_1(q)|}\, . \]
        It follows that $\lim_{q \to \infty}\circled{2}=d-2t$. Thus, our remaining task is to analyze terms $\circled{3}$, $\circled{4}$ in $\circled{1}$. We split the problem into four cases. \\
        \textit{Case 1}: $t=s$. It is easy to see $\lim_{q\to\infty}\circled{3}=0$ and 
        \begin{align*}
        \circled{4}=\frac{1}{q}2^{-qd} \sum_{m \in B_q^d \backslash \{0\}} \xi_m^2 =\frac{1}{q}\alpha(d,q)\, ,
        \end{align*} 
        so that $\lim_{q\to\infty}\circled{4}=0$. Here we use the definition of $\alpha$ in Lemma \ref{lemma: uniform convergence of series}. Therefore, $\lim_{q \to \infty} \tilde{L}_{\mathrm{EB}}(s,q)=d-2s$.\\
        \textit{Case 2}: $1/\delta\geq t\geq s+\epsilon$. We have $\circled{3}\geq 0$. The term \circled{4} can be written as
        \begin{align*}
        \circled{4}=\frac{1}{q2^{-q(2t-2s)}}\alpha(2t-2s+d,q)\, ,
        \end{align*}
        where we recall the definition of the function $\alpha$ in Lemma \ref{lemma: uniform convergence of series}. 
        According to this lemma, we get the uniform convergence \[\lim_{q\to \infty} \alpha(2t-2s+d,q)=\gamma(2t-2s+d) > 0\] in probability. In the meantime, $\lim_{q \to \infty} q2^{-q(2t-2s)}=0$. So, $\lim_{q \to \infty}  \circled{4} = \infty$ in probability, and uniformly in $ 1/\delta\geq t\geq s+\epsilon$. In terms of $\tilde{L}_{\mathrm{EB}}(t,q)$, this corresponds to $\lim_{q \to \infty}  \tilde{L}_{\mathrm{EB}}(t,q) = \infty$.\\
        \textit{Case 3}: $s-\epsilon \geq t\geq s-d/2+\epsilon$. In this case, $2t-2s+d \geq \epsilon$ so Lemma \ref{lemma: uniform convergence of series} can be applied. We write the term
        \begin{align*}
        \circled{4}=\frac{2^{-q(2s-2t)}}{q}\alpha(2t-2s+d,q)\, .
        \end{align*}
        This will converge to $0$ as $q$ goes to infinity, since $\lim_{q \to \infty}\frac{2^{-q(2s-2t)}}{q} =0 $ and $\lim_{q \to \infty}\alpha(2t-2s+d,q)=\gamma(2t-2s+d)\in (0,\infty)$. The term \circled{3} also converges to $0$. Thus, $\lim_{q \to \infty}  \tilde{L}_{\mathrm{EB}}(t,q) = d-2t$ in probability, and uniformly for $s-\epsilon \geq t\geq s-d/2+\epsilon$. \\
        \textit{Case 4}: $s-d/2+\epsilon \geq t \geq d/2+\delta$. We still have that \circled{3} converges to $0$. For term \circled{4}, we have
        \[\circled{4}=\frac{2^{-qd}}{q}\sum_{m \in B_q^d\backslash \{0\}} |m|^{2t-2s}\xi_m^2 \leq \frac{2^{-qd}}{q}\sum_{m \in B_q^d\backslash \{0\}} |m|^{2(s-d/2+\epsilon)-2s}\xi_m^2\, \]
        where we have used the monotonicity of the function $|m|^{2t-2s}$ with respect to $t$. Then, it reduces to the case $t=s-d/2+\delta$, which is covered by Case 3. Hence, we have $\lim_{q\to \infty} \circled{4}=0$ uniformly for $s-d/2+\delta \geq t \geq d/2+\delta$. Therefore, we get $\lim_{q \to \infty}  \tilde{L}_{\mathrm{EB}}(t,q) = d-2t$ in probability, and uniformly for $s-d/2+\delta \geq t \geq d/2+\delta$.
        
        Let us make a summary of the arguments above. We have established that, for any small $\epsilon > 0$, $\lim_{q \to \infty} \tilde{L}_{\mathrm{EB}}(t,q) =\infty$ uniformly for $1/\delta\geq t\geq s+\epsilon$, and $\lim_{q \to \infty} \tilde{L}_{\mathrm{EB}}(t,q) =d-2t$ uniformly for $s-\epsilon\geq t\geq d/2+\delta$, and $\lim_{q \to \infty} \tilde{L}_{\mathrm{EB}}(s,q) =d-2s$. All the convergence is in probability. Note that $s^{\mathrm{EB}}$ is the minimizer of $L_{\mathrm{EB}}(t,q)$, hence also of $\tilde{L}_{\mathrm{EB}}(t,q)$. The above convergence results for $\tilde{L}_{\mathrm{EB}}(t,q)$ imply that  $s^{\mathrm{EM}}\in (s-\epsilon,s+\epsilon)$ with probability $1$ as $q$ goes to infinity, for any $\epsilon>0$. Thus, we must have
        \[\lim_{q \to \infty} s^{\mathrm{EB}}(q)=s\, . \]
        The proof is complete.
    \end{proof}
    \subsection{Proof of Proposition \ref{prop: Lower bound on the interaction term}} \label{Proof of Proposition prop: Lower bound on the interaction term} \begin{proof}
    In order to write the interaction terms as a random series with some desired independence pattern for the random variables involved, we need to consider the geometry of the lattice carefully. We introduce another set $S_q:=\{m \in \bZ: -2^{q-2}\leq m \leq 3\times2^{q-2}-1\}$ and let $S_q^d=S_q\otimes S_q\otimes \cdots\otimes S_q$ denote the tensor product of $d$ multiples of $S_q$. The set $S_q$ is a shift of $B_q$, and $S_q^d$ is a shift of $B_q^d$.
    
        Define the set $B_{q-1}^d+2^{q-1}k:=\{m+2^{q-1}k:m\in B_{q-1}^d\}$ for $k \in \bZ^d$. We have the relation 
        \[S_q^d=\bigcup_{k \in \bZ^d_2}(B_{q-1}^d+2^{q-1} k)\]
        where $\bZ^d_2=\{0,1\}^d$. Note that for $k_1\neq k_2$, the intersection between  $B_{q-1}^d+2^{q-1}k_1$ and $B_{q-1}^d+2^{q-1}k_2$ is empty.
        
        Using \eqref{eqn: uq-uq-1 representation} and the periodicity of the functions involved, we get
        \begin{align*}
        &\|u(\cdot,t,q)-u(\cdot,t,q-1)\|_t^2\\
        =&(4\pi^2)^t\sum_{m \in B_q^d} M_q^t(m)\left(\frac{T_q\hat{u}(m)}{M_q^t(m)}-\frac{T_{q-1}\hat{u}(m)}{M_{q-1}^t(m)}\right)^2\\
        =&(4\pi^2)^t\sum_{m \in S_q^d} M_q^t(m)\left(\frac{T_q\hat{u}(m)}{M_q^t(m)}-\frac{T_{q-1}\hat{u}(m)}{M_{q-1}^t(m)}\right)^2\\
        =&(4\pi^2)^t\sum_{k \in \bZ_2^d}\sum_{m \in (B_{q-1}^d+2^{q-1}k)} M_q^t(m)\left(\frac{T_q\hat{u}(m)}{M_q^t(m)}-\frac{T_{q-1}\hat{u}(m)}{M_{q-1}^t(m)}\right)^2\, .
        \end{align*}
        Recall the relation
        \[T_{q-1}\hat{u}(m)=\sum_{l \in \bZ^d_2} T_q\hat{u}(m+2^{q-1}l)\, , \]
        based on which we get
        \begin{align*}
        &\frac{T_q\hat{u}(m)}{M_q^t(m)}-\frac{T_{q-1}\hat{u}(m)}{M_{q-1}^t(m)}\\
        =&(\frac{1}{M_q^t(m)}-\frac{1}{M^t_{q-1}(m)})T_q\hat{u}(m)-\frac{1}{M^t_{q-1}}\sum_{l \in \bZ_2^d \backslash \{0\}} T_q \hat{u}(m+2^{q-1}l)\, .
        \end{align*}
        Since $\ud \sim \cN(0,(-\Delta)^{-s})$, it holds $\hat{u}(m) \sim \cN(0,(4\pi^2)^{-s}|m|^{-2s})$. Moreover, for different $m$, these Gaussian random variables are independent from each other. Thus, for a fixed $k$ and for $m \in (B_{q-1}^d+2^{q-1}k)$, the Gaussian random variables
        \[\frac{T_q\hat{u}(m)}{M_q^t(m)}-\frac{T_{q-1}\hat{u}(m)}{M_{q-1}^t(m)} \]
        are independent from each other. Furthermore, by calculating their variance, we can write
        \begin{align*}
        &M_q^t(m)\left(\frac{T_q\hat{u}(m)}{M_q^t(m)}-\frac{T_{q-1}\hat{u}(m)}{M_{q-1}^t(m)}\right)^2\\=&
        (4\pi^2)^{-s}\left[(\frac{1}{M_q^t(m)}-\frac{1}{M^t_{q-1}(m)})^2M_q^t(m)M_q^s(m)+\frac{M_q^t(m)}{(M_{q-1}^t(m))^2}\sum_{l \in \bZ^d_2 \backslash \{0\}} M_q^s(m+2^{q-1}l)\right]\xi_{k,m}^2\\
        =&(4\pi^2)^{-s}\left[\frac{M_q^s(m)(M_q^t(m)-M_{q-1}^t(m))^2}{M_q^t(m)(M_{q-1}^t(m))^2}+\frac{M_q^t(m)}{(M_{q-1}^t(m))^2}\sum_{l \in \bZ^d_2 \backslash \{0\}} M_q^s(m+2^{q-1}l)\right]\xi_{k,m}^2\\
        =&:A_{k,m}\xi_{k,m}^2
         \end{align*}
        where $\{\xi_{k,m}\}_m$ are independent unit scalar Gaussian random variables. Clearly, we have the lower bound
        \[A_{k,m}\geq (4\pi^2)^{-s} \frac{M_q^t(m)}{(M_{q-1}^t(m))^2}M_q^s(m-2^{q-1}k)\, .\] 

        Thus, denoting $e_1=(1,0,...,0)\in \bZ^d$, we get
        \begin{align*}
        &\|u(\cdot,t,q)-u(\cdot,t,q-1)\|_t^2\\
        \geq &(4\pi^2)^{t-s}\sum_{k \in \bZ_2^d}\sum_{m \in (B_{q-1}^d+2^{q-1}k)}  \frac{M_q^t(m)}{(M_{q-1}^t(m))^2} M_q^s(m-2^{q-1}k)\xi_{k,m}^2\\
        \geq& (4\pi^2)^{t-s}\sum_{m \in (B_{q-1}^d+2^{q-1}e_1)} \frac{M_q^t(m)}{(M_{q-1}^t(m))^2} M_q^s(m-2^{q-1}e_1)\xi_{e_1,m}^2\\
        =&  (4\pi^2)^{t-s}\sum_{m \in (B_{q-1}^d+2^{q-1}e_1)} \frac{M_q^t(m)}{(M_{q-1}^t(m-2^{q-1}e_1))^2} M_q^s(m-2^{q-1}e_1)\xi_{e_1,m}^2\\
        \gtrsim& \sum_{m \in (B_{q-1}^d \backslash \{0\}+2^{q-1}e_1)} \frac{2^{-2qt}}{|m-2^{q-1}e_1|^{-4t}}|m-2^{q-1}e_1|^{2s} \xi_{e_1,m}^2\\
        =&\sum_{m \in B_{q-1}^d \backslash \{0\}}2^{-2qt}|m|^{4t-2s}\xi_{e_1,m+2^{q-1}e_1}^2\, .
        \end{align*}
        In the above derivation, we have used the fact that for $m \in B_{q-1}^d$, it holds that $M_q^s(m)\simeq |m|^{-2s}, M_{q-1}^t(m)\simeq |m|^{-2t}$, and in particular, $M_q^t(m)\simeq |m|^{-2t} \simeq 2^{-2qt}$ for $m \in (B_{q-1}^d \backslash \{0\}+2^{q-1}e_1)$. Renaming the subscripts in $\xi_{e_1, m+2^{q-1}e_1}$ completes the proof.
    \end{proof}
    
    \subsection{Proof of Proposition \ref{prop: Upper bound on the interaction term}}
    \label{Proof of Proposition prop: Upper bound on the interaction term}
    \begin{proof}
        We need to upper bound $A_{k,m}$ for $k \in \bZ^d_2, m \in B_{q-1}^d+2^{q-1}k$, which is defined in the proof of Proposition \ref{prop: Lower bound on the interaction term}. First, we have 
        \[ \sum_{l \in \bZ^d_2 \backslash \{0\}} M_q^s(m+2^{q-1}l)=M_{q-1}^s(m)-M_q^s(m)\, , \]
        and the estimate $0\leq M_{q-1}^t(m)-M_{q}^t(m)\leq M_{q-1}^t(m)$ for any $d/2+\delta\leq t \leq 1/\delta$. Based on this observation, for $k \in \bZ^d \backslash \{0\}$ and $m \in B_{q-1}^d\backslash \{0\}+2^{q-1}k$, we have the bound
        \begin{align*}
        A_{k,m}&\lesssim \frac{M_q^s(m)}{M_q^t(m)}+M_q^t(m)\frac{M_{q-1}^s(m)}{(M_{q-1}^t(m))^2}\\
        &\lesssim 2^{-q(2s-2t)}+2^{-2tq}|m-2^{q-1}k|^{4t-2s}
        \end{align*}
        where we have used the fact that for $m \in B_{q-1}^d\backslash \{0\}+2^{q-1}k$, it holds that $M_q^s(m)\simeq 2^{-2sq}, M_q^t(m)\simeq 2^{-2tq}, M^s_{q-1}(m)\simeq |m-2^{q-1}k|^{-2s}, M^t_{q-1}(m)\simeq |m-2^{q-1}k|^{-2t}$, according to Lemma \ref{lemma: estimate of the M t q term}. For $m=2^{q-1}k$, we get $A_{k,m} \lesssim 2^{-q(2s-2t)}$. So in general, we can write $A_{k,m}\lesssim 2^{-q(2s-2t)}+2^{-2tq}|m-2^{q-1}k|^{4t-2s}$ for $m \in B_{q-1}^d+2^{q-1}k$ where we use the convention that $|m|^{\alpha} = 0$ for $m=0$ and any $\alpha \in \bR$ to make the notation more compact.
        
        When $k=0$, using Lemma \ref{lemma: estimate of the M t q term} again, we get for $m \in B_{q-1}^d \backslash \{0\}$,
        \begin{align*}
        A_{k,m}&\lesssim \frac{M_q^s(m)(M_q^t(m)-M_{q-1}^t(m))^2}{M_q^t(m)(M_{q-1}^t(m))^2}+\frac{M_q^t(m)}{(M_{q-1}^t(m))^2}(M_{q-1}^s(m)-M_q^s(m))\\
        &\lesssim \frac{|m|^{-2s}2^{-4tq}}{|m|^{-6t}}+\frac{|m|^{-2t}}{|m|^{-4t}}2^{-2sq}\\
        &=|m|^{6t-2s}2^{-4tq}+|m|^{2t}2^{-2sq}\\
        &\lesssim 2^{-2tq}|m|^{4t-2s}+2^{-q(2s-2t)}\, ,
        \end{align*}
        where in the last line we used the relation $|m|\lesssim 2^{q}$. For $m=0$, based on the above calculation, we can get $A_{k,m} \lesssim 2^{-q(2s-2t)}$. Thus, generally, we can write $A_{k,m}\lesssim 2^{-2tq}|m|^{4t-2s}+2^{-q(2s-2t)}$ for $m \in B_{q-1}^d$ by using the notational convention above.
        
        Combining these estimates, we arrive at
        \begin{align*}
        &\|u(\cdot,t,q)-u(\cdot,t,q-1)\|_t^2\\
        \lesssim& \sum_{k \in \bZ_2^d} \sum_{m \in (B_{q-1}^d+2^{q-1}k)} A_{k,m}\xi_{k,m}^2\\
        \lesssim& \sum_{k \in \bZ_2^d} \sum_{m \in (B_{q-1}^d+2^{q-1}k)} (2^{-q(2s-2t)}+2^{-2tq}|m-2^{q-1}k|^{4t-2s})\xi_{k,m}^2\\
        =&\sum_{k \in \bZ_2^d} \sum_{m \in B_{q-1}^d} (2^{-q(2s-2t)}+2^{-2tq}|m|^{4t-2s})\xi_{k,m+2^{q-1}k}^2 \, .
        \end{align*}
        After a change of notation, we get the desired estimate.
    \end{proof}
    \subsection{Proof of Theorem \ref{thm: Consistency of the Kernel Flow estimator}}
    \label{Proof of Theorem thm: Consistency of the Kernel Flow estimator}
    \begin{proof}
        Recall \[s^{\text{KF}}(q)=\argmin_{t \in [d/2+\delta,1/\delta]} \mathsf{L}^{\text{KF}}(t,q):= \frac{\|u(\cdot,t,q)-u(\cdot,t,q-1)\|_t^2}{\|u(\cdot,t,q)\|_t^2}\, . \]
        We analyze the denominator and numerator separately. We start with the numerator. Let
        \[V_1(t,q)=\frac{1}{q}2^{q(s-d/2)} \|u(\cdot,t,q)-u(\cdot,t,q-1)\|_t^2\, . \]
        \textit{Case 1}: $t=\frac{s-d/2}{2}$. We derive an upper bound on $V_1$. By Proposition \ref{prop: Upper bound on the interaction term},
        \[\|u(\cdot,t,q)-u(\cdot,t,q-1)\|_t^2 \lesssim \sum_{k \in \bZ_2^d} \sum_{m \in B_{q-1}^d} (2^{-q(2s-2t)}+2^{-2tq}|m|^{4t-2s})\xi_{k,m}^2 \, .\]
        Take $t=\frac{s-d/2}{2}$. For each $k \in \bZ_2^d$, consider the term 
        \begin{align*}
        V_1^k(t,q)&=\frac{1}{q}2^{q(s-d/2)}\sum_{m \in B_{q-1}^d} (2^{-q(2s-2t)}+2^{-2tq}|m|^{4t-2s})\xi_{k,m}^2\\
        &=\frac{1}{q}2^{q(s-d/2)}\sum_{m \in B_{q-1}^d} (2^{-q(s+d/2)}+2^{-q(s-d/2)}|m|^{-d})\xi_{k,m}^2\\
        &=\frac{1}{q} \sum_{m \in B_{q-1}^d} (2^{-qd}+|m|^{-d})\xi_{k,m}^2\\
        &\lesssim \frac{1}{q} \sum_{m \in B_{q-1}^d} |m|^{-d}\xi_{k,m}^2\, .
        \end{align*}
        By Lemma \ref{lemma: uniform convergence of series}, $\lim_{q \to \infty} \frac{1}{q} \sum_{m \in B_{q-1}^d} |m|^{-d}\xi_{k,m}^2 = \gamma(0) \in (0,\infty)$. Thus, $V_1^k(t,q)$ remains bounded for $q \in \bN$. Since $V_1(t,q)=\sum_{k \in \bZ_2^d} V_1^k(t,q)$, it follows that $V_1(t,q)$ remains bounded for $q \in \bN$, in the case $t=\frac{s-d/2}{2}$. \\
        \textit{Case 2}: $1/\delta \geq t \geq \frac{s-d/2}{2}+\epsilon$. We provide a lower bound of $V_1$ here.
        Using Proposition \ref{prop: Lower bound on the interaction term}, we get
        \begin{align*}
        V_1(t,q)&\gtrsim \frac{1}{q}2^{q(s-d/2)} \sum_{m \in B_{q-1}^d \backslash \{0\}}2^{-2tq}|m|^{4t-2s}\xi_{m}^2\\
        &=\frac{1}{q}2^{q(s-d/2-2t)}\sum_{m \in B_{q-1}^d \backslash \{0\}}|m|^{4t-2s}\xi_{m}^2\\
        &=\frac{1}{q}2^{q(s-d/2-2t)}2^{(q-1)(4t-2s+d)}\cdot \left(2^{-(q-1)(4t-2s+d)}\sum_{m \in B_{q-1}^d \backslash \{0\}}|m|^{4t-2s}\xi_{m}^2\right)\\
        &=\frac{1}{q}2^{(q/2-1)(4t-2s+d)}\alpha(4t-2s+d,q-1)\, .
        \end{align*}
        By Lemma \ref{lemma: uniform convergence of series}, $\lim_{q \to \infty} \alpha(4t-2s+d,q-1)= \gamma(4t-2s+d) > 0$ uniformly for $1/\delta \geq t \geq \frac{s-d/2}{2}+\epsilon$. Since $\lim_{q \to \infty}\frac{1}{q}2^{(q/2-1)(4t-2s+d)}=\infty$, we get $\lim_{q \to \infty} V_1(t,q)=\infty$ and its growth rate is  $\gtrsim \frac{1}{q}2^{(q/2-1)(4t-2s+d)}$.
        \\ \textit{Case 3}: $\frac{s-d/2}{2}-\epsilon \geq t \geq d/2+\delta$. We provide a lower bound on $V_1$ here.
        Similarly to our analysis in Case 2, we have
        \begin{align*}
        V_1(t,q)&\gtrsim \frac{1}{q}2^{q(s-d/2-2t)}\sum_{m \in B_{q-1}^d \backslash \{0\}}|m|^{4t-2s}\xi_{m}^2\\
        & \gtrsim \frac{1}{q}2^{q(s-d/2-2t)} \xi_1^2\, .
        \end{align*}
        Then, it holds that
        \begin{align*}
        \bP(\frac{1}{q}2^{q(s-d/2-2t)} \xi_1^2\geq 2^{q(s-d/2-2t)/2})=\bP(\xi_1^2 \geq q 2^{-q(s-d/2-2t)/2}) \to 1
        \end{align*}
        as $q \to \infty$. Thus, we get $\lim_{q \to \infty} V_1(t,q)=\infty$ uniformly for this range of $t$ and the growth rate is  $\gtrsim 2^{q(s-d/2-2t)/2}$. We have finished the  analysis of the numerator. Now we proceed to analyze the denominator, which comprises the norm term. From Proposition \ref{prop: |u(x,t,q)|^2 bound}, we have
        \begin{equation}
            \label{eqn: KF proof, numerator}
            \|u(\cdot,t,q)\|_t^2 \simeq 2^{-q(2s-2t)}\xi_0^2+\sum_{m \in B_q^d \backslash \{0\}} |m|^{2t-2s}\xi_m^2\, ,
        \end{equation}
        where $\{\xi_m \}_{m \in B_q^d}$ are independent unit scalar Gaussian random variables. Recall that our final target in this theorem is to show that, for any $\epsilon >0$,
        \[\lim_{q \to \infty} \bP[s^{\mathrm{KF}}(q) \in (\frac{s-d/2}{2}-\epsilon,\frac{s-d/2}{2}+\epsilon)] = 1\, . \]
        Let $I_{\epsilon}=[d/2+\delta,1/\delta]/[\frac{s-d/2}{2}-\epsilon,\frac{s-d/2}{2}+\epsilon]$. By rewriting the loss function, it suffices to show
        \[\lim_{q \to \infty} \bP[ \frac{V_1(\frac{s-d/2}{2},q)}{\|u(\cdot,\frac{s-d/2}{2},q)\|_{\frac{s-d/2}{2}}^2}\geq \inf_{t \in I_{\epsilon}}\frac{V_1(t,q)}{\|u(\cdot,t,q)\|_t^2}]=0\, . \]
        Let us write 
        \begin{equation}
        \label{eqn: def of r(t,q)}
        r(t,q)=\frac{V_1(t,q)}{V_1(\frac{s-d/2}{2},q)}\cdot \frac{\|u(\cdot,\frac{s-d/2}{2},q)\|_{\frac{s-d/2}{2}}^2}{\|u(\cdot,t,q)\|_t^2}\, ,  
        \end{equation}
        then all we need is to show
        \[\lim_{q \to \infty} \bP[\inf_{t \in I_{\epsilon}} r(t,q)\leq 1] = 0\, . \]
        For $t \in I_{\epsilon}^1=[d/2+\delta,\frac{s-d/2}{2}-\epsilon]$, according to the analysis for the numerator, we have that for some constant $C$ independent of $q$,
        \begin{equation}
        \label{eqn: proof KF V1(t,q)}
        \lim_{q \to \infty} \bP[\inf_{t \in I_{\epsilon}^1} \frac{V_1(t,q)}{2^{q(s-d/2-2t)/2}}\geq C]=1\, ,    
        \end{equation}
        and also, $V_1(\frac{s-d/2}{2},q)$ remains uniformly bounded for $q \in \bN$. Furthermore, the equation \eqref{eqn: KF proof, numerator} implies the following relation: 
        \begin{equation}
        \label{eqn: proof KF bound on norm term}
        \inf_{t \in I_{\epsilon}^1} \frac{\|u(\cdot,\frac{s-d/2}{2},q)\|_{\frac{s-d/2}{2}}^2}{\|u(\cdot,t,q)\|_t^2} \gtrsim 1\, ,   
        \end{equation}
        due to the inequality $t \leq \frac{s-d/2}{2}-\epsilon$. 
        Combining the above two estimates in \eqref{eqn: proof KF V1(t,q)}\eqref{eqn: proof KF bound on norm term}, and recalling the expression for $r(t,q)$ in \eqref{eqn: def of r(t,q)}, we get 
        \begin{equation}
        \label{eqn: KF proof bound on r(t,q) 1}
            \lim_{q\to \infty} \bP[\inf_{t \in I_{\epsilon}^1} r(t,q) \leq 1] = 0\, .
        \end{equation}
        Then, let $I_{\epsilon}^2=[\frac{s-d/2}{2}+\epsilon,1/\delta]$. We also need to show $\lim_{q\to \infty} \bP[\inf_{t \in I_{\epsilon}^2} r(t,q) \leq 1] = 0$, or equivalently, 
        \[\lim_{q \to \infty} \bP [\frac{\|u(\cdot,\frac{s-d/2}{2},q)\|_{\frac{s-d/2}{2}}^2}{V_1(\frac{s-d/2}{2},q)} \leq \sup_{t \in I_{\epsilon}^2}\frac{\|u(\cdot,t,q)\|_t^2}{V_1(t,q)} ]=0\, . \]
        Since $V_1(\frac{s-d/2}{2},q)$ remains bounded according to the result in the above Case 1, it suffices to show 
        \[\lim_{q \to \infty} \sup_{t \in I_{\epsilon}^2}  \frac{\|u(\cdot,t,q)\|_t^2}{V_1(t,q)} = 0\]
        in probability. Using the estimate of $V_1(t,q)$ in Case 2 that $V_1(t,q)\gtrsim \frac{1}{q}2^{(q/2-1)(4t-2s+d)}$, it suffices to show
        \[ \lim_{q \to \infty} \sup_{t \in I_{\epsilon}^2}  q2^{-(q/2-1)(4t-2s+d)}\|u(\cdot,t,q)\|_t^2 = 0\, .\]
        To achieve this, we recall the expression of the norm term and write 
        \begin{align*}
            &q2^{-(q/2-1)(4t-2s+d)}\|u(\cdot,t,q)\|_t^2\\
            \simeq & q2^{-q(s+d)+4t-2s+d}\xi_0^2+q2^{-(q/2-1)(4t-2s+d)}\sum_{m \in B_q^d \backslash \{0\}} |m|^{2t-2s}\xi_m^2
        \end{align*}
        Clearly, the first term on the right hand side converges to $0$, so we only need to deal with the second term. Let
        \[\beta(t,q)=q2^{-(q/2-1)(4t-2s+d)}\sum_{m \in B_q^d \backslash \{0\}} |m|^{2t-2s}\xi_m^2\, . \]
        Consider $t \in [s-d/2+\epsilon',1/\delta]$ where $\epsilon'$ is a parameter to be tuned. We have $2t-2s +d\geq\epsilon'>0$ so we are able to write
        \begin{align*}
            \beta(t,q)&=q2^{-(q/2-1)(4t-2s+d)}2^{q(2t-2s+d)}\alpha(2t-2s+d,q)\\
            &=q2^{-q(s-d/2)+4t-2s+d}\alpha(2t-2s+d,q)\, .
        \end{align*}
        By Lemma \ref{lemma: uniform convergence of series}, $\lim_{q \to \infty} \alpha(2t-2s+d,q) = \gamma(2t-2s+d)$ in probability uniformly for $t \in [s-d/2+\epsilon',1/\delta]$. Since $\lim_{q \to \infty}q2^{-(q/2-1)(4t-2s+d)}2^{q(2t-2s+d)}=0 $, we get $\lim_{q \to \infty} \sup_{t \in [s-d/2+\epsilon',1/\delta]} \beta(t,q) = 0$. \\
        For $t \in [\frac{s-d/2}{2}+\epsilon,s-d/2+\epsilon']$, we have the estimate 
        \[q2^{-(q/2-1)(4t-2s+d)} \leq \left(q2^{-(q/2-1)(4t-2s+d)}\right)_{t=\frac{s-d/2}{2}+\epsilon} = q2^{-2q\epsilon+4\epsilon} \]
        and 
        \[\sum_{m \in B_q^d \backslash \{0\}} |m|^{2t-2s}\xi_m^2 \leq  \sum_{m \in B_q^d \backslash \{0\}} |m|^{-d+2\epsilon'}\xi_m^2 \]
        where we have used the fact that $t$ is upper bounded by $s-d/2+\epsilon'$. Hence, 
        \begin{align*}
            \sup_{t \in [\frac{s-d/2}{2}+\epsilon,s-d/2+\epsilon']}\beta(t,q)&\leq q2^{-2q\epsilon+4\epsilon}\sum_{m \in B_q^d \backslash \{0\}} |m|^{-d+2\epsilon'}\xi_m^2\\
            &=q2^{-2q\epsilon+4\epsilon}2^{2q\epsilon'}\alpha(2\epsilon',q)\, .
        \end{align*}
    Now, we set $\epsilon'=\epsilon/2$ such that $\lim_{q\to \infty} q2^{-2q\epsilon+4\epsilon}2^{2q\epsilon'} = 0$. Lemma \ref{lemma: uniform convergence of series} leads to $\lim_{q \to \infty} \alpha(2\epsilon',q)=\gamma(2\epsilon')<\infty$, from which we can conclude $\lim_{q \to \infty} \sup_{t \in I_{\epsilon}^2} \beta(t,q) = 0$. Therefore, we get
    \begin{equation}
    \label{eqn: KF proof bound on r(t,q) 2}
    \lim_{q \to \infty} \bP[\inf_{t \in I^2_{\epsilon}} r(t,q)\leq 1] = 0\, .   
    \end{equation}
    Combining \eqref{eqn: KF proof bound on r(t,q) 1} and \eqref{eqn: KF proof bound on r(t,q) 2} gives
    \begin{equation}
    \lim_{q \to \infty} \bP[\inf_{t \in I_{\epsilon}} r(t,q)\leq 1] = 0\, .   
    \end{equation}
    Based on the definition of $r(t,q)$ in \eqref{eqn: def of r(t,q)} and the arguments therein, we obtain
    \[\lim_{q \to \infty} \bP[s^{\mathrm{KF}}(q) \in (\frac{s-d/2}{2}-\epsilon,\frac{s-d/2}{2}+\epsilon)] = 1\, , \]
    from which the consistency of the KF estimator follows.
    \end{proof}



\end{document}